\documentclass[12pt]{article}
\usepackage[cp1251]{inputenc}
\usepackage[T2A]{fontenc}
\usepackage[english]{babel}
\usepackage{amssymb,amsmath,amsthm}
\usepackage[nodisplayskipstretch]{setspace}
\usepackage{enumitem}
\usepackage{geometry}
\usepackage{titlesec}
\usepackage[square,numbers,sort&compress]{natbib}
\usepackage[colorlinks,linkcolor=blue,pagecolor=blue, citecolor=blue,urlcolor=black]{hyperref}
\usepackage{hypcap,breakurl}
\usepackage{verbatim}
\usepackage{euscript}

\titlelabel{\thetitle.\quad} 

\allowdisplaybreaks[4] 

\numberwithin{equation}{section}

\theoremstyle{definition}
 \newtheorem{definition}{Definition}[section]
 \newtheorem{remark}{Remark}[section]

\theoremstyle{plain}
 \newtheorem{theorem}{Theorem}[section]
 
 \newtheorem{lemma}{Lemma}[section]
 
 \newtheorem*{lemma*}{Lemma}

 \newtheorem{corollary}{Corollary}[section]

 \newtheorem{statement}{Statement}[section]

\begin{document}
\newcommand{\Es}{\EuScript}
\newcommand{\scrscr}{\scriptscriptstyle}
\newcommand{\diag}{\mathop{\mathrm{diag}}}
\newcommand{\mE}{{\mathcal{E}}}
\newcommand{\mV}{{\mathcal{V}}}
\newcommand{\mI}{{\mathcal{I}}}
\newcommand{\mT}{{\mathcal{T}}}

\newcommand{\R}{{\mathbb R}}
\newcommand{\N}{{\mathbb N}}
\newcommand{\mC}{{\mathbb C}}
\newcommand{\Z}{{\mathbb Z}}
\newcommand{\Rn}{{{\mathbb R}^n}}
\newcommand{\Rm}{{{\mathbb R}^m}}
\newcommand{\Cn}{{{\mathbb C}^n}}
\newcommand{\Cm}{{{\mathbb C}^m}}

\newcommand{\T}{{\mathop{\mathrm{T}}}}
\newcommand{\tr}{\mathop{\mathrm{tr}}\nolimits}
\newcommand{\sgn}{\mathop{\mathrm{sgn}}}
\newcommand{\Ker}{\mathop{\mathrm{Ker}}}
\newcommand{\rank}{\mathop{\mathrm{rank}}}
\newcommand{\res}{\mathop{\mathrm{Res}}}
\newcommand{\ee}{\mathrm{e}}
\newcommand{\ii}{\mathrm{i}} 
\newcommand{\dd}{\textup{d}}
\newcommand{\la}{\lambda}
\renewcommand{\Im}{\operatorname{Im}}
\renewcommand{\Re}{\operatorname{Re}}
\newcommand{\ord}{\mathrm{O}}
\newcommand{\grad}{\mathop{\mathrm{grad}}}

\newcommand{\lin}{\mathop{\mathrm{Lin}}}

\title{\textbf{Criterion of the global solvability of regular and singular differential-algebraic equations}}
\author{\large \textbf{Maria Filipkovska}  \\
 \it\small  Friedrich-Alexander-Universit\"{a}t Erlangen-N\"{u}rnberg,\; Cauerstrasse 11, 91058 Erlangen, Germany\\
 \it\small  B. Verkin Institute for Low Temperature Physics and Engineering \\
 \it\small  of the National Academy of Sciences of Ukraine,\; Nauky Ave. 47, 61103 Kharkiv, Ukraine\\
 \small maria.filipkovska@fau.de;\; filipkovskaya@ilt.kharkov.ua}

\date{ }
\maketitle

 \begin{abstract}
For regular and nonregular (singular) semilinear differential-algebraic equations (DAEs), we prove theorems on the existence and uniqueness of global solutions and on the blow-up of solutions, which allow one to identify the sets of initial values for which the initial value problem has global solutions and/or for which solutions is blow-up in finite time, as well as the regions that the solutions cannot leave. Together these theorems provide a criterion of the global solvability of semilinear DAEs. As a consequence, we obtain conditions for the global boundedness of solutions.

 \medskip\noindent
{\small\emph{Key words}:\; singular pencil;  regular pencil; differential-algebraic equation; degenerate differential equation;  global solution; instability; blow-up

 }
 \end{abstract}

 \section{Introduction}

Differential-algebraic equations (DAEs) arise from the modelling of processes and systems in mechanics, control problems, gas industry, chemical kinetics, radio engineering, economics and other fields (see, e.g., \cite{BGHHST,Kunkel_Mehrmann,Vlasenko1} and references therein). It is well known that DAEs are used in modelling various objects whose structure is described by directed graphs, e.g., electrical, gas and neural networks. The application of DAEs in electrical circuit theory is described, for example, in \cite{Kunkel_Mehrmann,Vlasenko1,Fil.UMJ,Fil.MPhAG,Rut,rut-sing} and in gas network theory is shown, e.g., in \cite{BGHHST,KSSTW22}.
In the papers \cite{Fil.MPhAG} and \cite{Fil.UMJ}, nonlinear electrical circuits described by regular and singular (i.e., nonregular) semilinear DAEs, respectively, have been studied by using the proved (in these works) theorems on the Lagrange stability, which also give conditions for the global solvability, and on the Lagrange instability. In \cite{Fil.Sing-GN}, a gas network model similar to that presented in the work \cite{KSSTW22} has been considered. This model is described by an overdetermined system of differential and algebraic equations and is written in the form of a nonregular semilinear DAE.
Note that underdetermined and overdetermined systems of differential and algebraic equations are particular cases of the systems which can be written as nonregular semilinear DAEs.

There are many works devoted to the study of the local solvability, structure and the Lyapunov stability of equilibrium positions of DAEs and the methods of their numerical solution (see, e.g., the monographs \cite{Kunkel_Mehrmann,Vlasenko1} and references therein). DAEs are also called descriptor equations (or systems), degenerate differential equations and differential equations (or dynamical systems) on manifolds.  Most of the works on DAEs are related to the study of regular DAEs.

The existence and uniqueness of local solutions of nonregular (singular)  semilinear DAEs in Banach spaces has been studied in \cite{rut-sing}, where the decomposition of a singular pencil into regular and purely singular components, which was called the RS-splitting of the pencil, and the block representations of the singular component in two special cases have been presented.
Nonregular DAEs in finite-dimensional spaces have been studied, e.g., in \cite{Kunkel_Mehrmann}, where the concept of the ``strangeness index'' of a DAE and a pair of matrices (or matrix functions) have been used.
The present work uses the concept of an index (not the strangeness index) only for a regular block of the characteristic pencil (in the case when the DAE is singular) or the regular characteristic pencil (in the case when the DAE is regular) (see explanations in Section \ref{Preliminaries} and Appendix \ref{Appendix}).

Using a generalized canonical form (GCF) of time-varying nonregular linear DAEs, their solvability and some important issues related to the application of the least squares method to their numerical solution have been studied in \cite{ChistyakovCh11}. A time-invariant nonregular linear DAE in GCF take the form of the DAE with the characteristic pencil (i.e., the matrix pencil associated with the DAE) in the Weierstrass-Kronecker canonical form. This form (see \cite{GantmaherII}) is usually used for solving linear time-invariant DAEs. In the present paper, we use the special block form of the singular operator pencil that consists of the singular and regular blocks where invertible and zero subblocks are separated out, and the block form of the corresponding matrix pencil is different from the canonical form presented~in~\cite{GantmaherII}.

 \smallskip
In the present paper, we generalize and weaken conditions of the existence and uniqueness of global solutions of regular and singular (nonregular) semilinear DAEs (see Section \ref{SectGlobSolv}) that have been obtained in early works \cite{Fil.MPhAG,Fil.UMJ,Fil.Sing-GN}. Also, we specifying conditions of the blow-up of solutions of the DAEs (see Section \ref{SectLagrUnst}). The main differences are that the obtained (in this paper) conditions allow one to identify the sets of initial values for which the initial value problem has global solutions and/or for which solutions is blow-up in finite time, as well as the regions which the solutions cannot leave. As a consequence, more general conditions than in early works are found for the global boundedness of solutions (see Section \ref{SectGlobSolv}).  Together, the conditions for the global solvability and the blow-up of solutions provide a criterion of the global solvability (see Section \ref{CritGlobSolv}).  Obtained results allow one to solve problems that cannot be studied using the theorems from \cite{Fil.MPhAG,Fil.UMJ,Fil.Sing-GN} or other known theorems on the global solvability of DAEs. In Section \ref{Sect-Example} we consider a series of simple examples that demonstrate the application of the obtained theorems to the study of semilinear DAEs. In Section \ref{Preliminaries}, a problem statement and some notations and definitions are given, and a system of  differential equations (DEs) and algebraic equations (AEs), equivalent to the DAE under consideration, is presented. To reduce the DAE to the system of DEs and AEs, the special block form of the characteristic operator pencil, the direct decompositions of spaces, which reduce the pencil to this block form, and the projectors onto the subspaces from these decompositions are used. The main results related to this block form and the associated direct decompositions of spaces and projectors are given in Appendix~\ref{Appendix}.

\emph{The following standard notation will be used}. By $\mathrm{L}(X,Y)$ we denote the space of continuous linear operators from $X$ to $Y$; $\mathrm{L}(X,X)=\mathrm{L}(X)$, and similarly, $C((a,b),(a,b))=C(a,b)$;  $I_X$ denotes the identity operator in the space $X$; $\Es A^{(-1)}$ denotes the semi-inverse operator of an operator $\Es A$ \,($A^{-1}$ denotes the inverse operator of $A$); $\mathcal R(A)$ is the range of an operator $A$; $\Ker(A)$ is the kernel of an operator $A$;
$\overline{D}$ is the closure of a set $D$;
$D^c$, where $D$ is a set in a space $X$, is the complement of $D$ relative to $X$, i.e., $D^c=X\setminus D$;
$\partial D$ is the boundary of a set $D$;
$L_1\dot+ L_2$ is the direct sum of the linear spaces $L_1$ and $L_2$;
 $\lin\{x_i\}_{i=0}^N$ is the linear span of a system $\{x_i\}_{i=0}^N$;
$X'$ is the conjugate space of $X$ (it is also called an adjoint or dual space);
$A^\T$ denotes the transposed operator (i.e., the adjoint operator acting in real linear spaces to which the transposed matrix correspond) or the transposed matrix; $\delta_{ij}$ is the Kronecker delta;
the symbol $<<$ means ``much less than'';
$\partial_x$ denotes the partial derivative $\frac{\partial}{\partial x}$; $\partial_{(x_1,x_2,...,x_n)}=\frac{\partial}{\partial(x_1,x_2,...,x_n)}= \left(\frac{\partial}{\partial x_1},\frac{\partial}{\partial x_2},...,\frac{\partial}{\partial x_n}\right)$; a dot~$\dot{ }$ denotes the time derivative $\frac{d}{dt}$.
Note that both $A\subset B$ and $A\subseteq$ mean that $A$ is a subset of $B$, i.e., $A$ can be a proper subset of $B$ ($A\ne B$) or be equal to $B$; if $A$ is a proper subset of $B$, we write $A\subsetneqq B$.

 \smallskip
Let $f\colon J\to Y$ where $J$ is an interval in $\R$ and $Y$ is a normed linear space. If $J=[a,b)$, $b\le+\infty$ \,($J=(a,b]$, $a\ge -\infty$), then the derivative of the function $f$ at the point $a$ (at the point $b$) is understood as the derivative on the right (on the left) at this point (see, e.g., \cite{Schwartz1}). If $f\colon [a,b)\to Y$ is continuously differentiable on $(a,b)$ and in addition the right-hand derivative (i.e., the derivative on the right) of $f$ exists at the point $a$ and is continuous from the right at $a$, then $f$ is said to belong $C^1([a,b),Y)$.

 \section{Problem statement. Preliminaries}\label{Preliminaries}

In this section, we give a problem statement and introduce some notations and definitions which will be used in what follows (cf. \cite{Fil.MPhAG,Fil.UMJ,Fil.Sing-GN,LaSal-Lef}).
 \smallskip

We consider a differential-algebraic equation (or degenerate differential equation) of the form
 \begin{equation}\label{DAE}
\frac{d}{dt}[Ax]+Bx=f(t,x),
 \end{equation}
where $f\in C(\mT\times D,\Rm)$, $\mT\subseteq [0,\infty)$ is an interval,  $D\subseteq \Rn$ is an open set, and $A,\, B\in \mathrm{L}(\Rn,\Rm)$ (or $A$, $B$ are $m\times n$-matrices). We use the initial condition
 \begin{equation}\label{ini}
x(t_0)=x_0.
 \end{equation}
In the work, the case when $m=n$ and the operator $A$ is, in general, degenerate (noninvertible) and the case when $m\ne n$ are studied.  In this paper we use the DAE terminology and therefore the equation \eqref{DAE} is called a \emph{semilinear differential-algebraic equation}.

The pencil $\lambda A+B$ of operators or matrices $A$, $B$ (where $\lambda$ is a complex parameter) is called \emph{regular} if $n=m=\rank(\lambda A+B)$ (or $n=m$ and $\det(\lambda A+B)\not \equiv 0$); otherwise, that is,  if  $n\ne m$ or $n=m$ and  $\rank(\lambda A+B)<n$, the pencil is called \emph{singular}, or \emph{nonregular}, or \emph{irregular}. Recall that the rank of the operator pencil $\lambda A+B$ is the dimension of its range. The rank of the matrix pencil $\lambda A+B$ equals the maximum number of its linearly independent columns or rows (i.e.,  the maximum number of columns or rows of the pencil that are linearly independent set of vectors for some $\lambda=\lambda_0$) and equals the largest among the orders of the pencil minors that do not vanish identically. Obviously, the ranks of the pencil of operators and the pencil of corresponding matrices coincide.

The pencil $\lambda A+B$, associated with the linear part $\frac{d}{dt}[Ax]+Bx$ of the DAE \eqref{DAE}, is called \emph{characteristic}.
If the characteristic pencil $\lambda A+B$ is singular (respectively, regular), then the DAE is called  \emph{singular} (respectively, \emph{regular}), or \emph{nonregular}, or \emph{irregular}.

The nonregular semilinear DAE \eqref{DAE} corresponds to an overdetermined system of differential and algebraic equations (i.e., the number of equations $m$ is greater than the number of unknowns $n$)  if $\rank(\lambda A+B)=n<m$, to an underdetermined system (i.e., the number of equations is less than the number of unknowns) if $\rank(\lambda A+B)=m<n$, and to a closed system (i.e., the number of equations is equal to the number of unknowns) if $n=m$ and  $\rank(\lambda A+B)<n$.

\smallskip
The function $x(t)$ is called a \emph{solution of the equation \eqref{DAE} on $[t_0,t_1)\subseteq \mT$} if $x\in C([t_0,t_1),D)$, $(Ax)\in C^1([t_0,t_1),\Rm)$ and $x(t)$ satisfies \eqref{DAE} on $[t_0,t_1)$. If the function $x(t)$ additionally satisfies the initial condition \eqref{ini}, then it is called a \emph{solution of the initial value problem} (\emph{IVP}) \eqref{DAE}, \eqref{ini}.

A solution $x(t)$ (of an equation or inequality) is called \emph{global} if it exists on the interval $[t_0,\infty)$ (where $t_0$ is a given initial value).

A solution $x(t)$ \emph{has a finite escape time} or \emph{is blow-up in finite time} and is called \emph{Lagrange unstable} if it exists on some finite interval $[t_0,T)$ and is unbounded, that is, there exists $T<\infty$ such that $\lim\limits_{t\to T-0} \|x(t)\|=+\infty$.\;
A solution $x(t)$ is called \emph{Lagrange stable} if it is global and bounded, that is, $x(t)$ exists on the interval  $[t_0,\infty)$ and $\sup\limits_{t\in [t_0,\infty)} \|x(t)\|<\infty$.

The DAE \eqref{DAE} is called \emph{Lagrange unstable} (respectively, \emph{Lagrange stable}) \emph{for the initial point $(t_0,x_0)$} if the solution of the IVP \eqref{DAE}, \eqref{ini} is Lagrange unstable (respectively, Lagrange stable)  for this initial point.
The DAE \eqref{DAE} is called \emph{Lagrange unstable} (respectively, \emph{Lagrange stable}) if each solution of the IVP \eqref{DAE}, \eqref{ini} is Lagrange unstable  (respectively, Lagrange stable).

A convex set containing a point $x_0\in X$ that is contained in a ball $\{x\in X\mid \|x-x_0\|\le\delta\}$ (where $\delta\ge 0$) or coincides with it will be called a \emph{neighborhood} of the point $x_0$ and will be denoted by $N_\delta(x_0):=V_\delta(x_0)$ (it is possible that $N_\delta(x_0)=\{x_0\}$; in this case the neighborhood is degenerate). A neighborhood of some point that is an open (respectively, closed) set will be called an \emph{open} (respectively, \emph{closed}) neighborhood. By $U_\delta(x_0)$ and $\overline{N_\delta(x_0)}$ we denote the open neighborhood and closed neighborhood, respectively. Note that $\overline{U_\delta(x_0)}$ denotes the closure of the open neighborhood $U_\delta(x_0)$ (accordingly, $\delta> 0$). Sometimes we will denote a neighborhood (respectively, open neighborhood, closed neighborhood) of the point $x_0$ simply by $N(x_0)$ (respectively, $U(x_0)$, $\overline{N(x_0)}$), without indicating the radius of the ball which contains it.

Also, if $t\in [a,b]$, $a,b\in \R$, $a\ne b$, then by an open neighborhood $U_\delta(a)$ of the point $a$ we mean a semi-open interval $[a,a+\delta)$, $0<\delta<b-a$, and, similarly, by an open neighborhood $U_\delta(b)$ we mean a semi-open interval $(b-\delta,b]$, $0<\delta<b-a$.

 \smallskip
Let $X,\, Y$ be a normed linear space, $M\subseteq X$, and $J\subset \R$ be an interval. We use these notations in the definitions given below.

A mapping $f(t,x)$ of a set $J\times M$ into $Y$ is said to \emph{satisfy locally a Lipschitz condition} (or to be \emph{locally Lipschitz continuous}) \emph{with respect to $x$ on} $J\times M$ if for each (fixed) $(t_*,x_*)\in J\times M$ there exist open neighborhoods $U(t_*)$, $\widetilde{U}(x_*)$ of the points $t_*$, $x_*$ and a constant $L\ge 0$ such that $\|f(t,x_1)-f(t,x_2)\|\le L\|x_1-x_2\|$ for any $t\in U(t_*)$, $x_1,x_2\in \widetilde{U}(x_*)$.

We denote by $\rho(M_1,M_2)=\inf\limits_{x_1\in M_1,\, x_2\in M_2}\|x_2-x_1\|$  the distance between the closed sets $M_1$, $M_2$ in $X$ ($\rho(x,M)=\inf\limits_{y\in M}\|x-y\|$ denotes the distance from the point $x\in X$ to the set $M$).

A function $W\in C(M,\R)$, where $M\ni 0$, is called \emph{positive definite} if $W(0)=0$ and $W(x)>0$ for all $x\in M\setminus \{0\}$. A function $V\in C(J\times M,\R)$ is called \emph{positive definite} if $V(t,0)\equiv 0$ and there exists a positive definite function $W\in C(M,\R)$ such that $V(t,x)\ge W(x)$ for all $t\in J$, $x\in M\setminus \{0\}$.

A scalar function $V\in C(Z,\R)$ is called \emph{positive} if $V(z)>0$  for all $z\in Z$, where $Z$ is a set in a normed linear space.

A positive function $v\in C^1([a,\infty),(0,\infty))$  satisfying a differential inequality ${\dot{v}\le \chi(t,v)}$ (or ${\dot{v}\ge \chi(t,v)}$\,), where $\chi\in C([a,\infty)\times (0,\infty),\R)$, on $[a,\infty)$ is called a \emph{positive solution} of the inequality on $[a,\infty)$.

  \subsection{Reduction of a singular (nonregular) DAE to a system of ordinary differential and algebraic equations}\label{ReductionSingDAE}

The special block form of the singular characteristic pencil $\lambda A+B$ and the direct decompositions of spaces, which reduce the pencil to this block form, are used to reduce the singular DAE \eqref{DAE} to the system of ordinary differential equations (ODEs) and algebraic equations. The direct decompositions of spaces generate the projectors onto the subspaces from these decompositions, and the converse is also true. The main results related to the mentioned block form of the singular pencil and the corresponding direct decompositions of spaces and projectors  are given in Appendix \ref{Appendix}.

Throughout the paper, it is assumed that the regular block $\lambda A_r+ B_r$ from \eqref{penc} is a \emph{regular pencil of index not higher than 1} (see the definition in Appendix \ref{Appendix}).

First, recall that the function $f(t,x)$ from \eqref{DAE} is defined on $\mT\times D$, where $D\subseteq \Rn$ is an open set, and introduce the direct decomposition of $D$, similar to the direct decomposition \eqref{Xssrr}. The pair $P_1$, $P_2$ and pair $S_1$, $S_2$ of mutually complementary projectors generate the decomposition of the set $D$ into the direct sum of subsets
 \begin{equation}\label{Dssrr}
D=D_{s_1}\dot+ D_{s_2}\dot + D_1\dot + D_2,\qquad  D_{s_i}=S_iD,\quad D_i=P_iD,\quad i=1,2.
 \end{equation}
Obviously, $D_{s_i}\subseteq X_{s_i}$, $D_i\subseteq X_i$ ($i=1,2$), where $X_{s_i}$, $X_i$ are defined in \eqref{ssr}, \eqref{rr} (see Appendix \ref{Appendix}), and $D_{s_i}$, $D_i$  are open sets.

By using the projectors \eqref{ProjS}, \eqref{ProjR} (see Appendix \ref{Appendix}), the singular semilinear DAE \eqref{DAE} is reduced to the equivalent system (cf. \cite{Fil.UMJ})
 \begin{align}
\frac{d}{dt} (A S_1 x) &=F_1 [f(t,x)-B x], \label{DAEsysProj1} \\
\frac{d}{dt} (A P_1 x) &=Q_1 [f(t,x)-B x], \label{DAEsysProj2} \\
0 &= Q_2 [f(t,x)-B x], \label{DAEsysProj3} \\
0 &= F_2 [f(t,x)-B x], \label{DAEsysProj4}
 \end{align}
where $F_1=F$, $F_2=0$ if $\rank(\lambda A+B)= m<n$, and $S_1=S$ ($S_2=0$) if $\rank(\lambda A+B)= n<m$ (see Appendix \ref{Appendix} for details). The results obtained in the paper will be formulated for the general case when $\rank(\lambda A+B) < n$ and $\rank(\lambda A+B) < m$. The properties of the projectors allow one to immediately obtain the corresponding results for the cases when $\rank(\lambda A+B)= m<n$ or $\rank(\lambda A+B)= n<m$.

The system \eqref{DAEsysProj1}--\eqref{DAEsysProj4} is equivalent to
 \begin{equation}\label{DAEsysExtOp}
\begin{split}
& \frac{d}{dt} (\Es A_{gen} x_{s_1})+ \Es B_{gen} x_{s_1} + \Es B_{und} x_{s_2} = F_1 f(t,x), \\
& \frac{d}{dt} (\Es A_1 x_{p_1})+ \Es B_1 x_{p_1} = Q_1 f(t,x),\\
& \Es B_2 x_{p_2} =Q_2 f(t,x), \\
& \Es B_{ov}x_{s_1}=F_2 f(t,x),
\end{split}
 \end{equation}
where the operators defined in \eqref{ABssExtend} and \eqref{ABrrExtend} are used and $x_{s_i}=S_i x\in D_{s_i}$, $x_{p_i}=P_i x\in D_i$, $i=1,2$, are the components of the vector $x\in D$ represented as \eqref{xsr}, i.e., ${x=x_{s_1}+x_{s_2}+x_{p_1}+x_{p_2}}$.
Note that the representation of $x$ in the form $x=x_{s_1}+x_{s_2}+x_{p_1}+x_{p_2}$ \eqref{xsr} is uniquely determined for each $x\in \Rn$.
Further, the system \eqref{DAEsysExtOp} can be reduced to the equivalent system (as in \cite{Fil.Sing-GN})
\begin{align} 
&  \dot{x}_{s_1} =\Es A_{gen}^{(-1)}\big(F_1 f(t,x)-\Es B_{gen} x_{s_1}-\Es B_{und} x_{s_2}\big),   \label{DAEsysExtDE1}\\
& \dot{x}_{p_1} =\Es A_1^{(-1)} \big(Q_1 f(t,x)-\Es B_1 x_{p_1}\big),    \label{DAEsysExtDE2}\\
& \Es B_2^{(-1)} Q_2 f(t,x)- x_{p_2} = 0,   \label{DAEsysExtAE1}\\
& F_2 f(t,x)-\Es B_{ov}x_{s_1} = 0,  \label{DAEsysExtAE2}
\end{align}
where $\Es A_{gen}^{(-1)}$, $\Es A_1^{(-1)}$, $\Es B_2^{(-1)}$ are the semi-inverse operators which can be found by using the relations \eqref{InvAgen}, \eqref{InvA1} and \eqref{InvB1} (see Appendix \ref{Appendix}), and $x_{s_i}\in D_{s_i}$, $x_{p_i}\in D_i$, $i=1,2$,  ${x=x_{s_1}+x_{s_2}+x_{p_1}+x_{p_2}}$ as above (recall that a dot~$\dot{ }$ denotes the time derivative $\frac{d}{dt}$).

In addition, the system \eqref{DAEsysExtOp} can be reduced to the equivalent system
 \begin{align}
& \dot{x}_{s_1} =A_{gen}^{-1} \big(F_1 f(t,x) - B_{gen} x_{s_1} - B_{und} x_{s_2}\big),   \label{DAEsysDE1}\\
& \dot{x}_{p_1} =A_1^{-1} \big(Q_1 f(t,x) - B_1 x_{p_1}\big),    \label{DAEsysDE2}\\
& B_2^{-1} Q_2 f(t,x) - x_{p_2} = 0,   \label{DAEsysAE1}\\
& F_2 f(t,x) - B_{ov} x_{s_1} = 0,  \label{DAEsysAE2}
 \end{align}
where the operators defined in Statement~\ref{STABssr} and in Remark~\ref{Rem-RegPenc}, \eqref{A_iB_i} (see Appendix \ref{Appendix}) are used (note that $A_{gen}^{-1}=\Es A_{gen}^{(-1)}\big|_{Y_{s_1}}$, $A_1^{-1}=\Es A_1^{(-1)}\big|_{Y_1}$, $B_2^{-1}=\Es B_2^{(-1)}\big|_{Y_2}$, $B_1=\Es B_1\big|_{X_1}$, $B_{gen}=\Es B_{gen}\big|_{X_{s_1}}$, $B_{und}=\Es B_{und}\big|_{X_{s_2}}$, $B_{ov}=\Es B_{ov}\big|_{X_{s_1}}$) and the projectors $F_i$, $Q_i$ on the subspaces $Y_{s_i}$, $Y_i$, $i=1,2$, are considered as the operators from $\Rm$ into $Y_{s_i}$,~$Y_i$, respectively, (i.e., $F_i\in\mathrm{L}(\Rm,Y_{s_i})$, $Q_i\in\mathrm{L}(\Rm,Y_i)$\,) that have the same projection properties as the projectors $F_i,Q_i\in\mathrm{L}(\Rm)$ defined in Appendix \ref{Appendix} (see \eqref{ProjS}, \eqref{ProjR}), that is, $F_i y=F_i y_{s_i}=y_{s_i}\in Y_{s_i}$  and $Q_i y=Q_i y_{p_i}=y_{p_i}\in Y_i$, $i=1,2$, for any $y\in \Rm$.
Here we consider the projectors $F_i$ and $Q_i$ as the operators belonging to $\mathrm{L}(\Rm,Y_{s_i})$ and $\mathrm{L}(\Rm,Y_i)$, $i=1,2$, since the spaces $Y_{s_1}$, $Y_1$, $Y_2$ are the domains of definition of the (induced) operators $A_{gen}^{-1}$, $A_1^{-1}$, $B_2^{-1}$ respectively, and in addition ${B_{ov}\in \mathrm{L}(X_{s_1},Y_{s_2})}$.
The described differences in the definitions of the projectors $F_i$, $Q_i$ are, in general, formal and become significant only in the transition from the operators to the corresponding matrices. Therefore, we keep the same notations for the projectors $F_i$, $Q_i$ ($i=1,2$) both for the case when $F_i,Q_i\in\mathrm{L}(\Rm)$ and for the case when they are considered as the operators $F_i\in\mathrm{L}(\Rm,Y_{s_i})$, $Q_i\in\mathrm{L}(\Rm,Y_i)$, $i=1,2$.

Note that we identify the representations of $x\in \Rn$ in the form $x=x_{s_1}+x_{s_2}+x_{p_1}+x_{p_2}$ (i.e., \eqref{xsr}) and in the form $x=(x_{s_1},x_{s_2},x_{p_1},x_{p_2})$, where $x_{s_i}\in X_{s_i}$ and $x_{p_i}\in X_i$, $i=1,2$, as indicated in Remark \ref{Rem-Correspond} (see Appendix \ref{Appendix}). The correspondence between these representations is established in Remark \ref{Rem-Correspond}.  In the equations \eqref{DAEsysDE1}--\eqref{DAEsysAE2}, $x_{s_i}=S_i x$ and $x_{p_i}=P_i x$ ($x\in D$, $x_{s_i}\in D_{s_i}$,  $x_{p_i}\in D_i$, $i=1,2$) as well as in the equations equivalent to them, and, taking into account the established correspondence between the representations of $x$, we set $x=x_{s_1}+x_{s_2}+x_{p_1}+x_{p_2}=(x_{s_1},x_{s_2},x_{p_1},x_{p_2})$ for convenience. Therefore, in what follows, we will sometimes omit appropriate explanations, assuming that for $x$ one representation can be replaced by another.
For clarity, note that the projectors $S_i\colon \Rn\to X_{s_i}$, $P_i\colon \Rn\to X_i$, $i=1,2$, (see \eqref{ProjS}, \eqref{ProjR} in Appendix \ref{Appendix}) are, in general, the operators in $\Rn$, and for the representation of $x\in \Rn$  as the sum $x=x_{s_1}+x_{s_2}+x_{p_1}+x_{p_2}$ (with respect to the direct sum $X_{s_1}\dot+ X_{s_2}\dot+X_1 \dot+X_2$), where $x_{s_i}=S_i x$ and $x_{p_i}=P_i x$, we consider these projectors as the operators $S_i, P_i\in \mathrm{L}(\Rn)$, $i=1,2$, however, for the representation of $x$ as the ordered collection (column vector) $x=(x_{s_1},x_{s_2},x_{p_1},x_{p_2})$ (with respect to the direct product $X_{s_1}\times X_{s_2}\times X_1 \times X_2$) we consider these projectors as the operators $S_i\in \mathrm{L}(\Rn,X_{s_i})$, $P_i\in \mathrm{L}(\Rn,X_i)$, $i=1,2$, with the preservation of their projection properties (i.e., $S_i x=S_i x_{s_i}=x_{s_i}\in X_{s_i}$  and $P_i x=P_i x_{p_i}=x_{p_i}\in X_i$, $i=1,2$, for any $x\in \Rn$). Since the mentioned differences in the definitions of the projectors $S_i$, $P_i$ are formal and become significant only in the transition from the operators to the corresponding matrices, we keep the same notations for the projectors $S_i$, $P_i$ in both cases.

Below, when proving the theorems on global solvability, we use the norm $\|\cdot\|$ in $X_{s_1}\dot+ X_{s_2}\dot+X_1 \dot+X_2$, defined by
$\|x\|=\|x_{s_1}\|+\|x_{s_2}\|+ \|x_{p_1}\|+\|x_{p_2}\|$ where $\|x_{s_1}\|=\|x_{s_1}\|_{X_{s_1}}$, $\|x_{s_2}\|=\|x_{s_2}\|_{X_{s_2}}$, $\|x_{p_1}\|=\|x_{p_1}\|_{X_1}$ and $\|x_{p_2}\|=\|x_{p_2}\|_{X_2}$ denote the norms of the components $x_{s_1}$, $x_{s_2}$, $x_{p_1}$ and $x_{p_2}$ in the subspaces $X_{s_1}$, $X_{s_2}$, $X_1$ and $X_2$, respectively. Taking into account the established correspondence $x=x_{s_1}+x_{s_2}+x_{p_1}+x_{p_2}= (x_{s_1},x_{s_2},x_{p_1},x_{p_2})$, the norm $\|x\|$ of $x\in X_{s_1}\times X_{s_2}\times X_1\times X_2$ is defined in the same way and it coincides with the above-defined norm of the corresponding element $x\in X_{s_1}\dot+ X_{s_2}\dot+X_1 \dot+X_2$. Similarly, in $\R\times X_{s_1}\times X_{s_2}\times X_1\times X_2$ or $\R\times \Rn$ we use the norm $\|(t,x)\|=\|t\|+\|x_{s_1}\|+\|x_{s_2}\|+ \|x_{p_1}\|+\|x_{p_2}\|$. Generally, the inequality $\|x\|_\Rn\le \|x_{s_1}\|_\Rn+\|x_{s_2}\|_\Rn+ \|x_{p_1}\|_\Rn+\|x_{p_2}\|_\Rn$ holds for any norm $\|\cdot\|_\Rn$ in $\Rn=X_{s_1}\dot+ X_{s_2}\dot+X_1 \dot+X_2$ due to the representation \eqref{xsr}, and in this sense the chosen norm is ``maximal''.

It is shown above that the singular semilinear DAE \eqref{DAE} can be reduced to the equivalent system \eqref{DAEsysExtDE1}--\eqref{DAEsysExtAE2} or \eqref{DAEsysDE1}--\eqref{DAEsysAE2}.
Note that the derivative $\dot{V}_{\eqref{DAEsysExtDE1},\eqref{DAEsysExtDE2}}$ of a scalar function $V\in C^1(\mT\times K_{s1},\R)$, where $K_{s1}\subseteq D_{s_1}\times D_1$ is an open set, along the trajectories of the equations  \eqref{DAEsysExtDE1}, \eqref{DAEsysExtDE2} has the form (cf. \cite{Fil.Sing-GN})
 \begin{align}
&\dot{V}_{\eqref{DAEsysExtDE1},\eqref{DAEsysExtDE2}}(t,x_{s_1},x_{p_1}) = \partial _t V(t,x_{s_1},x_{p_1})+ \partial_{(x_{s_1},x_{p_1})} V(t,x_{s_1},x_{p_1})\cdot \Upsilon(t,x_{s_1},x_{s_2},x_{p_1},x_{p_2}) =   \nonumber\\
& =\partial _t V(t,x_{s_1},x_{p_1})+\partial_{x_{s_1}} V(t,x_{s_1},x_{p_1})\cdot  \left[\Es A_{gen}^{(-1)}\big(F_1 f(t,x)-\Es B_{gen} x_{s_1}-\Es B_{und} x_{s_2}\big)\right] +  \nonumber\\
& +\partial_{x_{p_1}} V(t,x_{s_1},x_{p_1})\cdot \left[\Es A_1^{(-1)} \big(Q_1 f(t,x)-\Es B_1 x_{p_1}\big)\right],  \label{dVDAEsing}\\
& \Upsilon(t,x_{s_1},x_{s_2},x_{p_1},x_{p_2})=\begin{pmatrix}
\Es A_{gen}^{(-1)}\big(F_1 f(t,x)-\Es B_{gen} x_{s_1}-\Es B_{und} x_{s_2}\big) \\
\Es A_1^{(-1)} \big(Q_1 f(t,x)-\Es B_1 x_{p_1}\big) \end{pmatrix}, \label{Upsilon}
 \end{align}
where $x=x_{s_1}+x_{s_2}+x_{p_1}+x_{p_2}$ \,($x_{s_i}=S_i x$, $x_{p_i}=P_i x$, $i=1,2$),  $(x_{s_1},x_{p_1})\in K_{s1}$, $x_{s_2}\in D_{s_2}$, $x_{p_2}\in D_2$.

  \subsection{Reduction of a regular semilinear DAE to a system of ordinary differential and algebraic equations}

Let the DAE \eqref{DAE} be regular, i.e., the characteristic pencil $\lambda A+B$ be regular ($n=m=\rank(\lambda A+B)$).
In this case, the pair $P_1$, $P_2$ of mutually complementary projectors (see Appendix \ref{Appendix} and, in particular, Remark \ref{Rem-RegCase})  generate the decomposition of the open set $D$ into the direct sum of subsets
 \begin{equation}\label{Drr}
D=D_1\dot + D_2,\qquad D_i=P_iD,\quad i=1,2,
 \end{equation}
where $D_i\subseteq X_i$ ($i=1,2$) are open sets.

We assume that the regular pencil $\lambda A+B$ has index not higher than 1 (see Appendix \ref{Appendix}). Then the regular semilinear DAE \eqref{DAE} can be reduced to the equivalent system of the form \eqref{DAEsysProj2},~\eqref{DAEsysProj3}, i.e.,
\begin{align}
 \frac{d}{dt} (A P_1 x) &=Q_1 [f(t,x)-B x], \label{DAEsysProjReg1} \\
 0 &=Q_2 [f(t,x)-B x]. \label{DAEsysProjReg2}
\end{align}
Further, the system \eqref{DAEsysProjReg1}, \eqref{DAEsysProjReg2} can be reduced to the equivalent system \eqref{DAEsysExtDE2},~\eqref{DAEsysExtAE1}, i.e.,
\begin{align}
& \dot{x}_{p_1}=\Es A_1^{(-1)} \big(Q_1 f(t,x)-\Es B_1 x_{p_1}\big),    \label{DAEsysExtReg1}\\ 
& \Es B_2^{(-1)} Q_2 f(t,x)- x_{p_2}=0,   \label{DAEsysExtReg2} 
\end{align}
or to the equivalent system (as in \cite{Fil.MPhAG})
 \begin{align}
& \dot{x}_{p_1}=G^{-1}\big(Q_1 f(t,x)-Bx_{p_1}\big), \label{DAEsysReg1G}  \\
& G^{-1}Q_2 f(t,x)-x_{p_2}=0, \label{DAEsysReg2G}
 \end{align}
where $x_{p_i}=P_i x\in D_i$, $i=1,2$, $x=x_{p_1}+x_{p_2}$ (see \eqref{xrr}), and $G$ is the operator \eqref{Proj.G} (see Appendix \ref{Appendix}). Note that $\Es A_1^{(-1)}Q_i=G^{-1}Q_i$, $i=1,2$,  $\Es B_2^{(-1)} Q_2=G^{-1}Q_2$ and $Bx_{p_1}=BP_1 x_{p_1}=\Es B_1x_{p_1}$.

In addition, the system \eqref{DAEsysProjReg1}, \eqref{DAEsysProjReg2} can be reduced to the equivalent system \eqref{DAEsysDE2}, \eqref{DAEsysAE1}, i.e.,
 \begin{align}
& \dot{x}_{p_1} =A_1^{-1} \big(Q_1 f(t,x) - B_1 x_{p_1}\big),    \label{DAEsysReg1}\\
& B_2^{-1}Q_2 f(t,x)-x_{p_2} = 0.   \label{DAEsysReg2}
 \end{align}
where the projectors $Q_i$ on the subspaces $Y_i$, $i=1,2$, are considered as the operators $Q_i\in\mathrm{L}(\Rn,Y_i)$ that have the same projection properties as the projectors $Q_i$ (by definition, $Q_i\in\mathrm{L}(\Rn)$) defined in Appendix \ref{Appendix} (see \eqref{ProjR} and Remark~\ref{Rem-Correspond}), that is, $Q_i y=Q_i y_{p_i}=y_{p_i}\in Y_i$, $i=1,2$, for any $y\in \Rn$. A similar remark regarding projectors is given for the system \eqref{DAEsysDE1}--\eqref{DAEsysAE2}, and in general, for the regular DAE \eqref{DAE}, remarks similar to those given in Section \ref{ReductionSingDAE} for the singular DAE hold.

Thus, the regular semilinear DAE \eqref{DAE} can be reduced to the equivalent system \eqref{DAEsysExtReg1}--\eqref{DAEsysExtReg2}, or \eqref{DAEsysReg1G}--\eqref{DAEsysReg2G}, or \eqref{DAEsysReg1}--\eqref{DAEsysReg2}.

Note that the derivative $\dot{V}_\eqref{DAEsysExtReg1}$ of a scalar function $V\in C^1(\mT\times K_1,\R)$, where $K_1\subset D_1$ is an open set ($D_1$ is defined in \eqref{Drr}), along the trajectories of the equation \eqref{DAEsysExtReg1} has the form
 \begin{align}
\dot{V}_\eqref{DAEsysExtReg1}(t,x_{p_1}) & =\partial _t V(t,x_{p_1})+ \partial_{x_{p_1}} V(t,x_{p_1})\cdot \Pi(t,x_{p_1},x_{p_2}), \label{dVDAEreg} \\
& \Pi(t,x_{p_1},x_{p_2})=\Es A_1^{(-1)}\big(Q_1 f(t,x_{p_1}+x_{p_2})-\Es B_1 x_{p_1}\big). \label{Pi}
 \end{align}
Since $\Pi(t,x_{p_1},x_{p_2})$ can be rewritten in the form $\Pi(t,x_{p_1},x_{p_2})=G^{-1}\big(Q_1 f(t,x_{p_1}+x_{p_2})-Bx_{p_1}\big)$, then $\dot{V}_\eqref{DAEsysExtReg1}=\dot{V}_\eqref{DAEsysReg1G}$ is the derivative of the function $V$ along the trajectories of the equation \eqref{DAEsysReg1G} as well. Generally, in the case when the DAE \eqref{DAE} is regular, we can set $S_i=F_i=0$, $i=1,2$, and then the derivatives \eqref{dVDAEsing} and \eqref{dVDAEreg} are equivalent.

  \subsection{Consistency condition}

Consider the manifold (cf. \cite{Fil.UMJ}) associated with the \emph{singular} semilinear DAE \eqref{DAE}:
 \begin{equation}\label{L_tSing}
L_{t_*}=\{(t,x)\in [t_*,\infty)\times \Rn\mid (F_2+Q_2)[f(t,x)-Bx]=0\},
 \end{equation}
where ${t_*\in \mT}$.  It can be represented as $L_{t_*}=\!\{(t,x)\!\in\! [t_*,\infty)\times \Rn \mid F_2[f(t,x)-Bx]=0,\; Q_2[f(t,x)-Bx]=0\}$ or
 $$
L_{t_*}=\{(t,x)\in [t_*,\infty)\times \Rn \mid \text{$(t,x)$ satisfies the equations \eqref{DAEsysExtAE1}, \eqref{DAEsysExtAE2}}\,\}.
 $$
Thus, a point $(t,x)\in \mT\times D$  belongs to $L_{t_*}$ if and only if it satisfies the equations \eqref{DAEsysExtAE1}, \eqref{DAEsysExtAE2} or the equivalent ones.

Also, we consider the manifold (cf. \cite{Fil.MPhAG}) associated with the \emph{regular} semilinear DAE \eqref{DAE}:
 \begin{equation}\label{L_treg}
L_{t_*}=\{(t,x)\in [t_*,\infty)\times \Rn \mid Q_2[f(t,x)-Bx]=0\},
 \end{equation}
where ${t_*\in \mT}$, which can be represented as
$$
L_{t_*}=\{(t,x)\in [t_*,\infty)\times \Rn \mid \text{$(t,x)$ satisfies the equation \eqref{DAEsysExtReg2}}\}
$$
(recall that the equations \eqref{DAEsysExtReg2} and \eqref{DAEsysExtAE1} are the same).  If the DAE \eqref{DAE} is regular, then we can set $S_i=F_i=0$, $i=1,2$, and reduce the manifold \eqref{L_tSing} to  \eqref{L_treg}.

Thus, we do not change the notation of the manifold $L_{t_*}$ in the singular and regular cases, which allows us to avoid excess notation. It will be clear from the context what exactly is meant.

The initial values $t_0$, $x_0$ satisfying the \emph{consistency condition} $(t_0,x_0)\in L_{t_*}$, where $t_*\le t_0$, $t_*\in \mT$, are called \emph{consistent} (the initial point $(t_0,x_0)\in L_{t_*}$, where $t_*\le t_0$,   is called \emph{consistent}). In general, the point $(t,x)$ (the values $t$, $x$) satisfying the \emph{consistency condition} $(t,x)\in L_{t_*}$, where $t_*\in \mT$, we will call \emph{consistent}.

Obviously, an initial point $(t_0,x_0)$ for a solution of the IVP \eqref{DAE}, \eqref{ini} must belong to the manifold $L_{t_0}$ and, generally, the graph of the solution (i.e., the set of points $(t,x(t))$, where $t$ from the domain of definition of the solution $x(t)$) must lie in this manifold.

  \section{Global solvability}\label{SectGlobSolv}

  \subsection{Global solvability of regular semilinear DAEs}\label{SectGlobReg}

For a clearer presentation of the results we will first consider a regular semilinear DAE.

Below, the projectors and subspaces, described in Appendix \ref{Appendix}, as well as the definitions and constructions, given in Section \ref{Preliminaries}, are used. Recall that $D_i=P_iD$, $i=1,2$ (see \eqref{Drr}).

 We first formulate all statements (theorems and corollaries) and then prove them.
 \begin{theorem}\label{Th_GlobReg-LipschObl}
Let $f\in C(\mT\times D,\Rn)$, where $D\subseteq \Rn$ is some open set and $\mT=[t_+,\infty)\subseteq [0,\infty)$, and let the operator pencil $\lambda A+B$ be a regular pencil of index not higher than 1. Assume that there exists an open set $M_1\subseteq D_1$ and a set $M_2\subseteq D_2$ such that the following holds:
 \begin{enumerate}[label={\upshape\arabic*.}, ref={\upshape\arabic*}, itemsep=3pt,parsep=0pt,topsep=4pt,leftmargin=0.6cm]
\item\label{SoglReg} For any fixed  ${t\in \mT}$, ${x_{p_1}\in M_1}$ there exists a unique ${x_{p_2}\in M_2}$ such that ${(t,x_{p_1}+x_{p_2})\in L_{t_+}}$ \textup{(}the manifold $L_{t_+}$ has the form \eqref{L_treg} where $t_*=t_+$\textup{)}.

\item\label{InvReg-Lipsch} A function $f(t,x)$ satisfies locally a Lipschitz condition with respect to $x$ on $\mT\times D$.\; For~any fixed $t_*\in \mT$, ${x_*=x_{p_1}^*+x_{p_2}^*}$ \,\textup{(}${x_{p_i}^*=P_ix_*}$, ${i=1,2}$\textup{)}\, such that $x_{p_1}^*\in M_1$, $x_{p_2}^*\in M_2$ and ${(t_*,x_*)\in L_{t_+}}$, there exist open neighborhoods $U_\delta(t_*,x_{p_1}^*)= U_{\delta_1}(t_*)\times U_{\delta_2}(x_{p_1}^*)\subset \mT\times D_1$ and $U_\varepsilon(x_{p_2}^*)\subset D_2$  \,\textup{(}the numbers $\delta, \varepsilon>0$ depend on the choice of $t_*$, $x_*$\textup{)} and an invertible operator $\Phi_{t_*,x_*}\in \mathrm{L}(X_2,Y_2)$ such that for each $(t,x_{p_1})\in U_\delta(t_*,x_{p_1}^*)$ and each $x_{p_2}^1,\, x_{p_2}^2\in U_\varepsilon(x_{p_2}^*)$ the mapping
     \begin{equation}\label{tildeFReg}
    \widetilde{F}(t,x_{p_1},x_{p_2}):= Q_2f(t,x_{p_1}+x_{p_2})-B\big|_{X_2}x_{p_2}\colon \mT\times D_1\times D_2\to Y_2
     \end{equation}
   satisfies the inequality
     \begin{equation}\label{ContractiveMapReg}
   \|\widetilde{F}(t,x_{p_1},x_{p_2}^1)- \widetilde{F}(t,x_{p_1},x_{p_2}^2)- \Phi_{t_*,x_*} [x_{p_2}^1-x_{p_2}^2]\|\le q(\delta,\varepsilon)\|x_{p_2}^1-x_{p_2}^2\|,
     \end{equation}
    where $q(\delta,\varepsilon)$ is such that $\lim\limits_{\delta,\,\varepsilon\to 0} q(\delta,\varepsilon)<\|\Phi_{t_*,x_*}^{-1}\|^{-1}$.

\item\label{RegAttractor}  If $M_1\neq X_1$, then the following holds.

  The component $x_{p_1}(t)=P_1x(t)$ of each solution $x(t)$ with the initial point $(t_0,x_0)\in L_{t_+}$, for which $P_ix_0\in M_i$, $i=1,2$, can never leave $M_1$ \textup{(}i.e., it remains in $M_1$ for all $t$ from the maximal interval of existence of the solution\textup{)}.

\item\label{ExtensReg}  If $M_1$ is unbounded, then the following holds.

  There exists a number ${R>0}$ \,\textup{(}$R$ can be sufficiently large\textup{)}, a function ${V\in C^1\big(\mT\times M_R,\R\big)}$ positive on $\mT\times M_R$, where $M_R=\{x_{p_1}\in M_1\mid \|x_{p_1}\|> R\}$, and a function ${\chi\in C(\mT\times (0,\infty),\R)}$  such that:
   \begin{enumerate}[label={\upshape(\ref{ExtensReg}.\alph*)}, ref={(\ref{ExtensReg}.\alph*)},itemsep=2pt,parsep=0pt,topsep=2pt]
  \item\label{ExtensReg1}
   $\lim\limits_{\|x_{p_1}\|\to+\infty}V(t,x_{p_1})=+\infty$ uniformly in $t$ on each finite interval $[a,b)\subset \mT$;

  \item\label{ExtensReg2}
  for each $t\in \mT$, $x_{p_1}\in M_R$, $x_{p_2}\in M_2$  such that $(t,x_{p_1}+x_{p_2})\in L_{t_+}$, the derivative \eqref{dVDAEreg} of the function $V$ along the trajectories of the equation \eqref{DAEsysExtReg1} satisfies the inequality
   \begin{equation}\label{LagrDAEReg}
  \dot{V}_\eqref{DAEsysExtReg1}(t,x_{p_1})\le \chi\big(t,V(t,x_{p_1})\big);
   \end{equation}

  \item\label{GlobSolvReg}
   the differential inequality
    \begin{equation}\label{L1v}
     \dot{v}\le \chi(t,v)\qquad (t\in \mT)
    \end{equation}
   does not have positive solutions with finite escape time.
   \end{enumerate}
 \end{enumerate}
Then there exists a unique global \textup{(}i.e., on $[t_0,\infty)$\textup{)} solution of the IVP \eqref{DAE}, \eqref{ini} for each initial point $(t_0,x_0)\in L_{t_+}$ for which $P_ix_0\in M_i$, $i=1,2$.
 \end{theorem}

 \begin{corollary}\label{Coroll-GlobReg1}
Theorem~\ref{Th_GlobReg-LipschObl} remains valid if condition \ref{RegAttractor} is replaced by
\begin{enumerate}[label={\upshape\arabic*.}, ref={\upshape\arabic*}, itemsep=3pt,parsep=0pt,topsep=4pt,leftmargin=0.6cm]
 \addtocounter{enumi}{2}
\item\label{RegAttractCoroll} In the case when $M_1\neq X_1$, the following holds.

Let $D_{\Pi,i}\subseteq X_i$, $i=1,2$, be sets such that the function $\Pi$ of the form \eqref{Pi} is defined and continuous for all $t\in \mT$ and $x_{p_i}\in D_{\Pi,i}$, $i=1,2$, where $D_{\Pi,i}\supset D_i$, $D_{\Pi,1}\times D_{\Pi,2}\ne D_1\times D_2$, if the domain of definition of $\Pi$ can be extended to $\mT\times D_{\Pi,1}\times D_{\Pi,2}$ in this way, and $D_{\Pi,i}= D_i$ otherwise. 
Let $D_{c,i}\subseteq X_i$, $i=1,2$, be sets such that for any fixed ${t\in \mT}$, ${x_{p_1}\in D_{c,1}\supset M_1}$ there exists a unique $x_{p_2}\in D_{c,2}\supset M_2$ such that ${(t,x_{p_1}+x_{p_2})\in L_{t_+}}$ \textup{(}i.e., $(t,x_{p_1}+x_{p_2})$ satisfies \eqref{DAEsysProjReg2}\,\textup{)} and $D_{c,1}\times D_{c,2}\ne M_1\times M_2$ if such sets exist, and $D_{c,i}= M_i$ otherwise.
Further, let $\widetilde{D_i}:= D_{\Pi,i}\cap D_{c,i}$, $i=1,2$, if the function $\Pi$ depends on $x_{p_2}$, and $\widetilde{D_1}:= D_{\Pi,1}$ if $\Pi$ does not depend on $x_{p_2}$.

Below, $\Pi$ is considered as the function \eqref{Pi} with the domain of definition $\mT\times \widetilde{D_1}\times \widetilde{D_2}$ if it depends on $x_{p_2}$ and $\mT\times \widetilde{D_1}$ if it does not depend on $x_{p_2}$.

Assume that there exists a function $W\in C(\mT\times X_1,\R)$ and for each sufficiently small number ${r>0}$  there exists a closed set $K_r=\{x_{p_1}\in M_1\mid \rho(K_r,M_1^c)=r\}$  \,\textup{(}${M_1^c=X_1\setminus M_1}$,\, $\rho(K_r,M_1^c)=\! \inf\limits_{k\,\in\, K_r,\, m\,\in\, M_1^c} \|m-k\|$\,\textup{)} such that
 $$
W(t_1,x_{p_1}^1)<W(t_2,x_{p_1}^2)
 $$
for every $x_{p_1}^1\in K_r$, $x_{p_1}^2\in M_1^c\cap \widetilde{D_1}$ and $t_1,t_2\in \mT$ such that $t_1\le t_2$, and, in addition, $W(t,x_{p_1})$ has the continuous partial derivatives on $\mT\times K_r^c$ \,\textup{(}${K_r^c=X_1\setminus K_r}$\textup{)} and the inequality
\begin{equation}\label{RegAttractIneq}
\dot{W}_\eqref{DAEsysExtReg1}(t,x_{p_1})=\partial_t W(t,x_{p_1})+ \partial_{x_{p_1}} W(t,x_{p_1})\cdot \Pi(t,x_{p_1},x_{p_2})\le 0
\end{equation}
holds for each $t\in \mT$, $x_{p_1}\in K_r^c\cap \widetilde{D_1}$, $x_{p_2}\in \widetilde{D_2}$ such that ${(t,x_{p_1}+x_{p_2})\in L_{t_+}}$ \textup{(}if $\Pi$ does not depend on $x_{p_2}$, then \eqref{RegAttractIneq} holds for each $t\in \mT$, $x_{p_1}\in K_r^c\cap \widetilde{D_1}$\textup{)}.
\end{enumerate}
 \end{corollary}

 \begin{corollary}\label{Coroll-GlobReg2}
Theorem~\ref{Th_GlobReg-LipschObl} remains valid if condition \ref{ExtensReg} is replaced by
\begin{enumerate}[label={\upshape\arabic*.}, ref={\upshape\arabic*}, itemsep=3pt,parsep=0pt,topsep=4pt,leftmargin=0.6cm]
 \addtocounter{enumi}{3}
\item\label{ExtensRegCoroll} If $M_1$ is unbounded, then the following holds.

  There exists a number ${R>0}$, a function ${V\in C^1\big(\mT\times M_R,\R\big)}$ positive on ${\mT\times M_R}$, where ${M_R=\{x_{p_1}\in M_1\mid \|x_{p_1}\|> R\}}$, and functions ${k\in C(\mT,\R)}$, ${U\in C(0,\infty)}$ such that:\; ${\lim\limits_{\|x_{p_1}\|\to+\infty}V(t,x_{p_1})=+\infty}$ uniformly in $t$ on each finite interval ${[a,b)\subset \mT}$;\; for each ${t\in \mT}$, ${x_{p_1}\in M_R}$, ${x_{p_2}\in M_2}$  such that ${(t,x_{p_1}+x_{p_2})\in L_{t_+}}$, the inequality \,${\dot{V}_\eqref{DAEsysExtReg1}(t,x_{p_1})\le k(t)\, U\big(V(t,x_{p_1})\big)}$\, holds;\;  $\int\limits_{{\textstyle v}_0}^{\infty}\dfrac{dv}{U(v)} =\infty$\, \textup{(}$v_0>0$ is a constant\textup{)}.
\end{enumerate}
 \end{corollary}

 \begin{corollary}\label{Coroll-UstLagrReg}
If in the conditions of Theorem~\ref{Th_GlobReg-LipschObl} the sets $M_1$ and $M_2$ are bounded, then the equation \eqref{DAE} is Lagrange stable for the initial points $(t_0,x_0)\in L_{t_+}$ for which $P_ix_0\in M_i$, $i=1,2$.
 \end{corollary}

 \begin{theorem}\label{Th_GlobReg-Obl}
Theorem~\ref{Th_GlobReg-LipschObl} remains valid if condition \ref{InvReg-Lipsch} is replaced by
\begin{enumerate}[label={\upshape\arabic*.}, ref={\upshape\arabic*}, itemsep=3pt,parsep=0pt,topsep=4pt,leftmargin=0.6cm]
\addtocounter{enumi}{1}
\item\label{InvReg} A function $f(t,x)$ has the continuous partial derivative with respect to $x$ on $\mT\times D$.\;  For~any fixed $t_*\in \mT$, ${x_*=x_{p_1}^*+x_{p_2}^*}$ such that $x_{p_1}^*\in M_1$, $x_{p_2}^*\in M_2$ and ${(t_*,x_*)\in L_{t_+}}$, the operator  
     \begin{equation}\label{funcPhiInvReg}
    \Phi_{t_*,x_*}:=\left[\partial_x (Q_2f)(t_*,x_*)- B\right] P_2\colon X_2\to Y_2
     \end{equation}
    has the inverse $\Phi_{t_*,x_*}^{-1}\in \mathrm{L}(Y_2,X_2)$.
\end{enumerate}
 \end{theorem}
 \begin{remark}
The operator  $\Phi_{t_*,x_*}$ \eqref{funcPhiInvReg} is defined as an operator from $X_2$ into $Y_2$, however, in general, the operator defined by the formula from \eqref{funcPhiInvReg}, that is, $\widehat{\Phi}_{t_*,x_*}:=\left[\partial_x (Q_2f)(t_*,x_*)- B\right] P_2$ (where $t_*$, $x_*$ are fixed), is an operator from $\Rn$ into $\Rn$ with the rang $\widehat{\Phi}_{t_*,x_*}\Rn= Y_2$, and $\Phi_{t_*,x_*}=\widehat{\Phi}_{t_*,x_*}\big|_{X_2}$. Since it is assumed that $\Phi_{t_*,x_*}$ is invertible, then $\widehat{\Phi}_{t_*,x_*}\Rn=\widehat{\Phi}_{t_*,x_*}X_2=Y_2$ \,($X_1=\Ker(\widehat{\Phi}_{t_*,x_*})$).

The operator $\widehat{\Phi}_{t_*,x_*}$ has the semi-inverse $\widehat{\Phi}_{t_*,x_*}^{(-1)}$, i.e., the operator $\widehat{\Phi}_{t_*,x_*}^{(-1)}\in \mathrm{L}(\Rn)$ such that $\widehat{\Phi}_{t_*,x_*}^{(-1)}\Rn= \widehat{\Phi}_{t_*,x_*}^{(-1)}Y_2=X_2$ and $\Phi_{t_*,x_*}^{-1}=\widehat{\Phi}_{t_*,x_*}^{(-1)}\big|_{Y_2}$, which is defined by the relations \;$\widehat{\Phi}_{t_*,x_*}^{(-1)} \widehat{\Phi}_{t_*,x_*}= P_2$, $\widehat{\Phi}_{t_*,x_*} \widehat{\Phi}_{t_*,x_*}^{(-1)}= Q_2$ and $\widehat{\Phi}_{t_*,x_*}^{(-1)}= P_2\widehat{\Phi}_{t_*,x_*}^{(-1)}$.
 \end{remark}

 \begin{proof}[The proof of Theorem \ref{Th_GlobReg-LipschObl}]
Recall that the DAE \eqref{DAE} is equivalent to the system \eqref{DAEsysExtReg1}, \eqref{DAEsysExtReg2}, where the representation of $x$ in the form $x=x_{p_1}+x_{p_2}$, $x_{p_i}=P_i x\in D_i$, $i=1,2$, is uniquely determined for each $x$, and that the equation \eqref{DAEsysExtReg2} is equivalent to \eqref{DAEsysReg2}.

Also, note that there is a correspondence between $X_1 \dot+X_2$ and $X_1\times X_2$ which is established in the same way as in Remark \ref{Rem-Correspond} (see Appendix \ref{Appendix}) and we identify the representations of $x\in \Rn$ in the form  $x=x_{p_1}+x_{p_2}$ (see \eqref{xrr}) and the form $x=(x_{p_1},x_{p_2})$ where $x_{p_i}\in X_i$, $i=1,2$.

We denote
 $$
\tilde{f}(t,x_{p_1},x_{p_2})= f(t,x_{p_1}+x_{p_2})=f(t,x),
 $$
consider the mapping $F\in C(\mT\times D_1\times D_2,\,X_2)$ defined by
 \begin{equation}\label{funcFReg}
F(t,x_{p_1},x_{p_2}):= B_2^{-1}Q_2 \tilde{f}(t,x_{p_1},x_{p_2})-x_{p_2},
 \end{equation}
and write the equation \eqref{DAEsysReg2} in the form
 \begin{equation}\label{DAEsysReg2equiv}
F(t,x_{p_1},x_{p_2})=0.
 \end{equation}
Thus, the equation \eqref{DAEsysExtReg2} is equivalent to \eqref{DAEsysReg2equiv}. It follows from the definition of the manifold $L_{t_+}$ that $(t,x)\in L_{t_+}$  if and only if $(t,x)$ (where $x=x_{p_1}+x_{p_2}$) satisfies \eqref{DAEsysExtReg2} or \eqref{DAEsysReg2equiv}.

Note that the equation \eqref{DAEsysReg2equiv} can be rewritten as
 \begin{equation}\label{DAEsysReg2equiv2}
x_{p_2}=x_{p_2}-\Phi_{t_*,x_*}^{-1}\widetilde{F}(t,x_{p_1},x_{p_2}),
 \end{equation}
where $\widetilde{F}\in C(\mT\times D_1\times D_2,Y_2)$ is the mapping \eqref{tildeFReg} which can be represented as $\widetilde{F}(t,x_{p_1},x_{p_2})= Q_2\tilde{f}(t,x_{p_1},x_{p_2})-B_2x_{p_2}=B_2 F(t,x_{p_1},x_{p_2})$
(since $B_2=B\big|_{X_2}=\Es B_2\big|_{X_2}$), and $\Phi_{t_*,x_*}$ is the operator defined in condition \ref{InvReg-Lipsch}.

Let an open set $M_1\subseteq D_1$ and a set $M_2\subseteq D_2$ be such that the theorem conditions hold.
It follows from condition \ref{SoglReg} that for any fixed  $t_*\in \mT$, $x_{p_1}^*\in M_1$ there exists a unique $x_{p_2}^*\in M_2$ such that $(t_*,x_{p_1}^*+x_{p_2}^*)\in L_{t_+}$.

Below, we give lemma which is proved in a similar way as Lemma 3.1 from \cite{Fil.Sing-GN}.
 \begin{lemma}\label{Lemma-ImplicitFuncReg}
For any fixed $t_*\in \mT$, $x_{p_i}^*\in M_i$, $i=1,2$,  for which $(t_*,x_{p_1}^*+x_{p_2}^*)\in L_{t_+}$, there exist open neighborhoods $U_r(t_*,x_{p_1}^*)=U_{r_1}(t_*)\times U_{r_2}(x_{p_1}^*)\subset \mT\times D_1$, \,$U_\rho(x_{p_2}^*)\subset D_2$ and a unique function  $x_{p_2}=\nu(t,x_{p_1})\in C(U_r(t_*,x_{p_1}^*),U_\rho(x_{p_2}^*))$ which satisfies the equality $\nu(t_*,x_{p_1}^*)=x_{p_2}^*$ and a Lipschitz condition with respect to $x_{p_1}$  on $U_r(t_*,x_{p_1}^*)$ and is a solution of the equation \eqref{DAEsysReg2equiv} with respect to $x_{p_2}$, i.e., $F(t,x_{p_1},\nu(t,x_{p_1}))=0$ for all $(t,x_{p_1})\in U_r(t_*,x_{p_1}^*)$ (the numbers $\rho,r>0$ depend on the choice of $t_*$, $x_{p_1}^*$, $x_{p_2}^*$).
 \end{lemma}

From the above it follows that in some open neighborhood $U_r(t_*,x_{p_1}^*)$ of each (fixed) $(t_*,x_{p_1}^*)\in \mT\times M_1$ there exists a unique solution $x_{p_2}=\nu_{t_*,x_{p_1}^*}(t,x_{p_1})\in C(U_r(t_*,x_{p_1}^*),U_\rho(x_{p_2}^*))$, where $U_\rho(x_{p_2}^*)$ is an open neighborhood of the point $x_{p_2}^* \in M_2$ such that $(t_*,x_{p_1}^*+x_{p_2}^*)\in L_{t_+}$,  of the equation \eqref{DAEsysReg2equiv} and this solution satisfies a Lipschitz condition with respect to $x_{p_1}$ and the equality  $\nu_{t_*,x_{p_1}^*}(t_*,x_{p_1}^*)=x_{p_2}^*$.
Further, we introduce a function
\begin{equation}\label{eta}
\eta\colon \mT\times M_1\to M_2
\end{equation}
and define it by $\eta(t,x_{p_1})= \nu_{t_*,x_{p_1}^*}(t,x_{p_1})$ at the point \,$(t,x_{p_1})=(t_*,x_{p_1}^*)$ for each point $(t_*,x_{p_1}^*)\in \mT\times M_1$. Hence, the function $x_{p_2}=\eta(t,x_{p_1})$ is continuous in $(t,x_{p_1})$, satisfies locally a Lipschitz condition with respect to $x_{p_1}$ on $\mT\times M_1$ and is a unique solution of the equation \eqref{DAEsysReg2equiv} as well as the equivalent equation \eqref{DAEsysExtReg2} with respect to $x_{p_2}$  (for all $(t,x_{p_1})\in \mT\times M_1$).  Now we prove that the function $\eta$ is unique on the whole domain of definition (cf. \cite[Theorem 3.1]{Fil.MPhAG}, where an implicit function $\eta(t,z)$ which is a solution of the equation equivalent to \eqref{DAEsysReg2G} has been obtained under different conditions). Suppose that there exists another function $x_{p_2}=\sigma(t,x_{p_1})$ defined in the same way as $\eta$ and, accordingly, having the same properties as $\eta$, but differing from it at some point $(t_*,x_{p_1}^*)\in \mT\times M_1$. It follows from condition \ref{SoglReg} that $\sigma(t_*,x_{p_1}^*)=\eta(t_*,x_{p_1}^*)=x_{p_2}^*$, which contradicts the assumption. This holds for each $(t_*,x_{p_1}^*)\in \mT\times M_1$, and therefore $\eta(t,x_{p_1})\equiv\sigma(t,x_{p_1})$.

Substituting the function $x_{p_2}=\eta(t,x_{p_1})$ in \eqref{DAEsysExtReg1}, we get the equation
 \begin{equation}\label{DAEsysExtReg1eta}
\dot{x}_{p_1}=\widetilde{\Pi}(t,x_{p_1}),\quad \widetilde{\Pi}(t,x_{p_1})=\Es A_1^{(-1)}\big(Q_1 \widetilde{f}(t,x_{p_1},\eta(t,x_{p_1}))-\Es B_1 x_{p_1}\big),
 \end{equation}
where $\widetilde{\Pi}\colon \mT\times M_1\to X_1$ is continuous on $\mT\times M_1$ and satisfies locally a Lipschitz condition with respect to $x_{p_1}$ on $\mT\times M_1$, and $\widetilde{\Pi}(t,x_{p_1})=\Pi(t,x_{p_1},\eta(t,x_{p_1}))$ where $\Pi$ is defined in \eqref{Pi}.
Thus, for each initial point $(t_0,x_0)\in L_{t_+}$ for which $P_1x_0\in M_1$ (then $\eta(t_0,P_1x_0)=P_2x_0\in M_2$ by construction) there exists a unique solution $x_{p_1}=x_{p_1}(t)$ of the IVP for the equation \eqref{DAEsysExtReg1eta} with the initial condition
\begin{equation}\label{RegDEeta_ini}
x_{p_1}(t_0)=x_{p_1,0},\quad \text{where}\quad x_{p_1,0}=P_1x_0\in M_1,
\end{equation}
on some interval $[t_0,t_0+\delta_0)$, $\delta_0>0$. By the extension theorems (see \cite{Schwartz2,Hartman}), the solution $x_{p_1}(t)$ can be extended over a maximal interval of existence $J_{max}\subseteq [t_0,\infty)$ ($t_0\in J_{max}$) and the extended solution is a unique solution of the IVP \eqref{DAEsysExtReg1eta}, \eqref{RegDEeta_ini} on $J_{max}$. Hence, the IVP \eqref{DAE}, \eqref{ini} has a unique solution $x(t)=x_{p_1}(t)+\eta(t,x_{p_1}(t))$  on the maximal interval of existence $J_{max}$. Notice that, according to \cite{Hartman}, $J_{max}$ is a right maximal interval of existence since we extend the solution only to the right, however, for brevity, it will be called a maximal interval of existence.

 \smallskip
Let us prove that the maximal interval of existence $J_{max}$ of the solution $x_{p_1}(t)$ is $[t_0,\infty)$. The proofs differ depending on the form of the open set $M_1$.

  \smallskip
\emph{\textbf{Case 1.}} Assume that $M_1=X_1$ (accordingly, $M_1=D_1=X_1$). In this case, the proof is carried out by analogy with the proof of the existence of a global (on $[t_0,\infty)$) solution of a time-varying semilinear DAE which has been given in \cite[Theorem~3.1]{Fil.GSA} (there $X_1$ depends on $t$, but the idea of the proof is the same) or of a time-invariant semilinear DAE, but for special forms of the functions $V$ and $\chi$, which has been given in \cite[Theorem~3.1]{Fil.MPhAG}.  These proofs are based on the application of conditions of the type of condition \ref{ExtensReg}. Further, we give the brief proof of the required statement.
It follows from condition~\ref{ExtensReg} that there exists a number $R>0$, a function ${V\in C^1\big(\mT\times M_R,\R\big)}$ positive on $\mT\times M_R$, where $M_R=\{x_{p_1}\in X_1\mid \|x_{p_1}\|>R\}$, and a function $\chi\in C(\mT\times (0,\infty),\R)$ such that conditions \ref{ExtensReg1}, \ref{GlobSolvReg} hold and
 \begin{equation}\label{L1dVdtReg}  
\dot{V}_{\eqref{DAEsysExtReg1eta}}(t,x_{p_1})=\partial _t V(t,x_{p_1})+ \partial_{x_{p_1}} V(t,x_{p_1})\cdot \widetilde{\Pi}(t,x_{p_1})\le \chi\big(t,V(t,x_{p_1})\big)
 \end{equation}
for $t\in \mT$, $x_{p_1}\in M_R$. Then, using the theorem \cite[Chapter~IV, Theorem~XIII]{LaSal-Lef}, we obtain that the solution $x_{p_1}(t)$  of the IVP \eqref{DAEsysExtReg1eta}, \eqref{RegDEeta_ini} exists on $J_{max}=[t_0,\infty)$.

 \smallskip
\emph{\textbf{Case 2.}} Now, assume that $M_1$ is a proper subset of $X_1$, i.e., $M_1\subsetneqq X_1$ ($M_1\subseteq D_1$, $D_1\subseteq X_1$, $M_1\ne X_1$).

  \smallskip
By condition \ref{RegAttractor},  the solution $x_{p_1}(t)$ belongs to $M_1$ for all $t\in J_{max}$.

If the domain of definition of the function $\widetilde{\Pi}(t,x_{p_1})=\Pi(t,x_{p_1},\eta(t,x_{p_1}))$ from \eqref{DAEsysExtReg1eta}  can be extended (with respect to $x_{p_1}$), namely, if there exists a set $D_{\eta,1}\supset M_1$ ($D_{\eta,1}\subseteq X_1$) such that the function $\eta(t,x_{p_1})$ is uniquely defined and continuous for all $t\in \mT$, $x_{p_1}\in D_{\eta,1}$ and there exists a set $\widetilde{M}_1\supsetneqq M_1$ such that $\widetilde{\Pi}(t,x_{p_1})$ is defined and continuous for all $t\in \mT$, $x_{p_1}\in\widetilde{M}_1$, where $\widetilde{M}_1\subseteq D_{\eta,1}$, $D_{\eta,1}\ne M_1$ if the function $\Pi$ having the form \eqref{Pi} depends on $x_{p_2}$ and $\widetilde{M}_1\subseteq X_1$ if this function does not depend on $x_{p_2}$ (then we take $D_{\eta,1}= M_1$), then we consider $\widetilde{\Pi}$ as the function defined on $\mT\times \widetilde{M}_1$. Otherwise, we consider $\widetilde{\Pi}$ as the function defined on $\mT\times \widetilde{M}_1$ where $\widetilde{M}_1=M_1$, i.e., with the same domain of definition as above.

Consider the equation \eqref{DAEsysExtReg1eta}, where $\widetilde{\Pi}\in C(\mT\times \widetilde{M}_1,X_1)$ and $\widetilde{M}_1$ is the set mentioned above.  Obviously, the solution $x_{p_1}(t)$ of the IVP \eqref{DAEsysExtReg1eta}, \eqref{RegDEeta_ini}, which is obtained above, is a solution of the IVP \eqref{DAEsysExtReg1eta}, \eqref{RegDEeta_ini} where $\widetilde{\Pi}$ is defined on $\mT\times \widetilde{M}_1$. If $\widetilde{M}_1\supseteq \overline{M_1}$, then we consider the equation \eqref{DAEsysExtReg1eta} where $\widetilde{\Pi}\in C(\mT\times \overline{M_1},X_1)$. In this case, the IVP \eqref{DAEsysExtReg1eta}, \eqref{RegDEeta_ini} has a solution $x_{p_1}=\widehat{x}_{p_1}(t)$ on a (right) maximal interval of existence $J$ and by the extension theorems (e.g., \cite[p.~12--14, Theorem 3.1 and Corollary 3.2]{Hartman}) either $J=[t_0,\infty)$, or $J=[t_0,\beta)$ where $\beta<\infty$ and $\|\widehat{x}_{p_1}(t)\|\to \infty$ as $t\to\beta-0$, or  $J=[t_0,\beta]$ where $\beta<\infty$ and $\widehat{x}_{p_1}(\beta)\in \partial M_1$ (note that $\mT$ is the unbounded interval $[t_+,\infty)$).
Since the solution $x_{p_1}(t)$ belongs to $M_1$ for all $t$ from the maximal interval of existence $J_{max}$ due to condition \ref{RegAttractor}, then either $J_{max}=[t_0,\infty)$, or $J_{max}=[t_0,\beta)$ where $\beta<\infty$ and $\lim\limits_{t\to\beta-0}\|x_{p_1}(t)\|=\infty$.  By virtue of the extension theorem \cite[p.~12--13, Theorem 3.1]{Hartman} (notice that we can consider the open interval $(t_+,\infty)$ instead of $[t_+,\infty)=\mT$ without loss of generality since a solution of the IVP \eqref{DAEsysExtReg1eta}, \eqref{RegDEeta_ini} with the initial value $t_0=t_+$ exists on some interval $[t_+,t_++\delta_1)$, $\delta_1>0$, and that we extend the solution only to the right),\, the same holds if $\widetilde{M}_1=M_1$. Thus, either $J_{max}=[t_0,\infty)$, or $J_{max}=[t_0,\beta)$ where $\beta<\infty$ and $\lim\limits_{t\to\beta-0}\|x_{p_1}(t)\|=\infty$.

The further proof is carried out for various forms of the set $M_1\ne X_1$.

 \smallskip
\emph{Case 2.1.} Assume that the complement $M_1^c$ of the set $M_1$ ($M_1\ne X_1$) is bounded. Then there exists a number $R_0>0$ such that $\{x_{p_1}\in X_1\mid \|x_{p_1}\|\ge R_0\}\subset M_1$. As shown above, the solution $x_{p_1}(t)$ remains all the time in $M_1$. Due to condition~\ref{ExtensReg}, there exists a number $R\ge R_0$, a function ${V\in C^1\big(\mT\times M_R,\R\big)}$ positive on $\mT\times M_R$, where $M_R=\{x_{p_1}\in X_1\mid \|x_{p_1}\|>R\}$, and a function $\chi\in C(\mT\times (0,\infty),\R)$ such that conditions \ref{ExtensReg1}, \ref{GlobSolvReg} are satisfied and the derivative of $V$ along the trajectories of the equation \eqref{DAEsysExtReg1eta} satisfies the inequality \eqref{L1dVdtReg} for all $t\in \mT$, $x_{p_1}\in M_R$. Thus, as in case~1, using the theorem \cite[Chapter~IV, Theorem~XIII]{LaSal-Lef}, we obtain that the solution $x_{p_1}(t)$  of the IVP \eqref{DAEsysExtReg1eta}, \eqref{RegDEeta_ini} exists on $J_{max}=[t_0,\infty)$.

 \smallskip
\emph{Case 2.2.} Assume that $M_1$ ($M_1\ne X_1$) and its complement $M_1^c$ are unbounded sets. The solution $x_{p_1}(t)$, as shown above, remains all the time in $M_1$, and either $J_{max}=[t_0,\infty)$, or $J_{max}=[t_0,\beta)$, $\beta<\infty$, and $\lim\limits_{t\to\beta-0}\|x_{p_1}(t)\|=\infty$.  Due to condition~\ref{ExtensReg}, there exists a number $R>0$, a function ${V\in C^1\big(\mT\times M_R,\R\big)}$ positive on $\mT\times M_R$, where $M_R=\{x_{p_1}\in M_1\mid \|x_{p_1}\|> R\}$, and a function $\chi\in C(\mT\times (0,\infty),\R)$ such that conditions \ref{ExtensReg1}, \ref{GlobSolvReg} are satisfied and the inequality \eqref{L1dVdtReg} holds for all $t\in \mT$ and $x_{p_1}\in M_R$. Then, we carry out the proof by analogy with the proof of the theorem \cite[Chapter~IV, Theorem~XIII]{LaSal-Lef} and obtain that the solution $x_{p_1}(t)$ exists on $[t_0,\infty)$, i.e., $[t_0,\beta)=[t_0,\infty)$.

 \smallskip
\emph{Case 2.3.} Assume that the set $M_1$ is bounded. Since the solution $x_{p_1}(t)$  belongs to $M_1$ for all $t\in J_{max}=[t_0,\beta)$, $\beta\le \infty$ \,($J_{max}$ is the maximal interval of existence), then it is bounded on $J_{max}$, i.e., $\sup\limits_{t\in J_{max}}\|x_{p_1}(t)\|=const<\infty$. Consequently, it is impossible that $\beta<\infty$ and $\lim\limits_{t\to\beta-0}\|x_{p_1}(t)\|=\infty$. Then it follows from the above that $J_{max}=[t_0,\infty)$.

\smallskip
It is proved that $J_{max}=[t_0,\infty)$. Thus, for each initial point $(t_0,x_0)\in L_{t_+}$ for which $P_ix_0\in M_i$, $i=1,2$, the IVP \eqref{DAE}, \eqref{ini} has a unique global solution.
 \end{proof}

 \begin{proof}[The proof of Corollary \ref{Coroll-GlobReg1}]
We only need to prove that the solution $x_{p_1}(t)$ of the IVP \eqref{DAEsysExtReg1eta}, \eqref{RegDEeta_ini}, which is obtained in the proof of Theorem \ref{Th_GlobReg-LipschObl}, belongs to $M_1$ for all $t\in J_{max}$, where $J_{max}$ is the maximal interval of existence of the solution. The rest of the proof is similar to the proof of Theorem~\ref{Th_GlobReg-LipschObl}.

Recall that the function $\widetilde{\Pi}$ from \eqref{DAEsysExtReg1eta} is defined by ${\widetilde{\Pi}(t,x_{p_1})=\Pi(t,x_{p_1},\eta(t,x_{p_1}))}$ where the function ${\Pi(t,x_{p_1},x_{p_2})}$ has the form \eqref{Pi} and is defined and continuous for $t\in \mT$, $x_{p_i}\in D_i$, $i=1,2$. If the domain of definition of $\Pi$ can be extended (with respect to $x_{p_1}$, $x_{p_2}$) to $\mT\times D_{\Pi,1}\times D_{\Pi,2}$ where $D_{\Pi,i}$  are the sets introduced in condition \ref{RegAttractCoroll}, then $\Pi(t,x_{p_1},x_{p_2})$ is defined and continuous for all $x_{p_i}\in D_{\Pi,i}$, $i=1,2$, and $t\in \mT$.
Note that if the point $(t,x_{p_1}+x_{p_2})$ satisfies the equation  \eqref{DAEsysProjReg2}, then it satisfies the equations \eqref{DAEsysExtReg2} and \eqref{DAEsysReg2equiv}, since these equations are equivalent.
Thus, if the domain of definition of the function \eqref{eta} can be extended (with respect to $x_{p_1}$), namely, if there exists a set $D_{\eta,1}\supset M_1$ ($D_{\eta,1}\subseteq X_1$) such that $\eta(t,x_{p_1})$ is uniquely defined and continuous for all $x_{p_1}\in D_{\eta,1}$, $t\in \mT$, then $D_{\eta,1}\subseteq D_{c,1}$ and $\eta\colon \mT\times D_{\eta,1}\to D_{c,2}$, where $D_{c,i}$, $i=1,2$, are the sets introduced in condition \ref{RegAttractCoroll}. If $D_{c,1}= M_1$, then $D_{\eta,1}=M_1$. Obviously, the function  $\widetilde{\Pi}(t,x_{p_1})$ is defined and continuous for all $t\in \mT$ and $x_{p_1}\in\widetilde{M}_1$, where $\widetilde{M}_1= D_{\Pi,1}\cap D_{\eta,1}\subseteq \widetilde{D_1}$ if the function $\Pi$ depends on $x_{p_2}$ and $\widetilde{M}_1= D_{\Pi,1}=\widetilde{D_1}$ if $\Pi$ does not depend on $x_{p_2}$. Below  $\widetilde{\Pi}$ is considered as the function defined on $\mT\times \widetilde{M}_1$ (obviously, $\widetilde{M}_1\supseteq M_1$).

Further, consider the equation \eqref{DAEsysExtReg1eta}, where $\widetilde{\Pi}\in C(\mT\times \widetilde{M}_1,X_1)$, and prove the following lemma.
 \begin{lemma}\label{Lem-UltimSet}
Assume that there exists a function $W\in C(\mT\times X_1,\R)$  and for each sufficiently small number $r>0$ ($r<<1$) there exists a closed set $K_r\subset M_1$ for which $\rho(K_r,M_1^c)=r$, such that ${W(t_1,x_{p_1}^1)<W(t_2,x_{p_1}^2)}$   for every $x_{p_1}^1\in K_r$, $x_{p_1}^2\in M_1^c\cap \widetilde{M}_1$ and $t_1,t_2\!\in\! \mT$ such that $t_1\le t_2$, and, in addition, $W(t,x_{p_1})$ has the continuous partial derivatives on $\mT\times K_r^c$ and
 $$
\dot{W}_\eqref{DAEsysExtReg1eta}(t,x_{p_1})=\partial_t W(t,x_{p_1})+ \partial_{x_{p_1}} W(t,x_{p_1})\cdot \widetilde{\Pi}(t,x_{p_1})\le 0
 $$
for every $t\in \mT$, $x_{p_1}\in K_r^c\cap \widetilde{M}_1$. Then each solution $x_{p_1}(t)$ of the equation \eqref{DAEsysExtReg1eta} which satisfies the initial condition \eqref{RegDEeta_ini} can never leave $M_1$.
 \end{lemma}
  \begin{proof}
Take arbitrary fixed $r>0$ (it is sufficient to consider $r<<1$; thus, $r$ is sufficiently small), and choose a function $W\in C(\mT\times X_1,\R)$ and a closed set $K_r=\{x_{p_1}\in M_1\mid \rho(K_r,M_1^c)=r\}$ that satisfies the conditions of the lemma. Note that $\rho(x_{p_1},K_r)=\inf\limits_{k\,\in\, K_r}\|x_{p_1}-k\|<r$ for any point $x_{p_1}\in M_1$. By \cite[p.~116, Lemma~1]{LaSal-Lef} each solution $x_{p_1}(t)$ of the equation \eqref{DAEsysExtReg1eta} which at some $t_0\in \mT$ is in the set $K_r$ can never thereafter leave $M_1$. Moreover, this holds for any (sufficiently small) $r>0$, that is, for any closed set $K_r$ specified above.

Since the set $M_1\subset X_1$ is open, then it can be represented as the union  $M_1=\bigcup\limits_{r>0}K_r$ of closed sets (an infinite family of closed sets) $K_r\subset M_1$ such that $\rho(K_r,M_1^c)=r>0$ (i.e., $K_r=\{x_{p_1}\in M_1\mid \rho(K_r,M_1^c)=r>0\}$), where $r$ is sufficiently small.  Consequently, the initial value $x_{p_1,0}$ of each solution $x_{p_1}(t)$ of \eqref{DAEsysExtReg1eta} satisfying the initial condition $x_{p_1}(t_0)=x_{p_1,0}\in M_1$ belongs to one of the sets $K_r$ for which the conditions of the lemma are fulfilled. Hence, this solution cannot leave $M_1$. Thus, each solution of \eqref{DAEsysExtReg1eta} satisfying the initial condition \eqref{RegDEeta_ini} can never leave $M_1$.
 \end{proof}

It follows from condition \ref{RegAttractor} that the conditions of Lemma \ref{Lem-UltimSet} holds and, therefore, the solution $x_{p_1}(t)$ of the IVP \eqref{DAEsysExtReg1eta}, \eqref{RegDEeta_ini} belongs to $M_1$ for all $t\in J_{max}$. Note that the set $\widetilde{M}_1$ specified above denotes the same set that is denoted by $\widetilde{M}_1$ in the proof of Theorem \ref{Th_GlobReg-LipschObl}.

The rest of the proof coincides with the proof of Theorem \ref{Th_GlobReg-LipschObl}.
 \end{proof}

 \begin{proof}[The proof of Corollary \ref{Coroll-GlobReg2}]
Let $\chi(t,v):=k(t)\,U(v)$ where ${k\in C(\mT,\R)}$ and ${U\in C(0,\infty)}$ satisfies the relation $\int\limits_{{\textstyle v}_0}^{\infty}\dfrac{dv}{U(v)} =\infty$ ($v_0>0$), then the differential inequality \eqref{L1v} does not have positive solutions with finite escape time (see, e.g., \cite{LaSal-Lef}). Hence,  condition \ref{ExtensReg} of Theorem~\ref{Th_GlobReg-LipschObl}, where the function $\chi$ has the form $\chi(t,v)=k(t)\,U(v)$, is fulfilled.  Thus, all conditions of Theorem~\ref{Th_GlobReg-LipschObl} hold.
 \end{proof}

 \begin{proof}[The proof of Corollary \ref{Coroll-UstLagrReg}]
This statement follows immediately from the proof of Theorem~\ref{Th_GlobReg-LipschObl}. Indeed, since $M_1$ and $M_2$ are bounded, then $\sup\limits_{t\in J_{max}}\|x_{p_1}(t)\|<\infty$ and $\sup\limits_{t\in J_{max}}\|\eta(t,x_{p_1}(t)\|<\infty$, where ${J_{max}=[t_0,\infty)}$ is the maximal interval of existence of the solution $x(t)=x_{p_1}(t)+\eta(t,x_{p_1}(t))$, and hence  ${\sup\limits_{t\in J_{max}}\|x(t)\|<\infty}$.
 \end{proof}

 \begin{proof}[The proof of Theorem \ref{Th_GlobReg-Obl}]
We will prove that condition \ref{InvReg-Lipsch} of Theorem~\ref{Th_GlobReg-LipschObl} holds if condition \ref{InvReg} of Theorem~\ref{Th_GlobReg-Obl} holds. Then all conditions of Theorem~\ref{Th_GlobReg-LipschObl} are fulfilled.

Since $f(t,x)$ has the continuous partial derivative with respect to $x$ on $\mT\times D$, then it satisfies locally a Lipschitz condition with respect to $x$ on $\mT\times D$.
Choose arbitrary fixed $t_*\in \mT$, ${x_*=x_{p_1}^*+x_{p_2}^*}$  such that $x_{p_1}^*\in M_1$, $x_{p_2}^*\in M_2$ and ${(t_*,x_*)\in L_{t_+}}$.
Take the operator $\Phi_{t_*,x_*}$ defined by \eqref{funcPhiInvReg} as the operator $\Phi_{t_*,x_*}$ appearing in condition \ref{InvReg-Lipsch} of Theorem~\ref{Th_GlobReg-LipschObl}. Then $\Phi_{t_*,x_*}=\partial_{x_{p_2}} \widetilde{F}(t_*,x_{p_1}^*,x_{p_2}^*)$ and $\Phi_{t_*,x_*}$ has the inverse $\Phi_{t_*,x_*}^{-1}\in \mathrm{L}(Y_2,X_2)$ due to condition \ref{InvSing} of Theorem~\ref{Th_GlobReg-Obl}.
Therefore, we have $\big\|\widetilde{F}(t,x_{p_1},x_{p_2}^1)- \widetilde{F}(t,x_{p_1},x_{p_2}^2)- \partial_{x_{p_2}}\widetilde{F}(t_*,x_{p_1}^*,x^*_{p_2}) [x_{p_2}^1-x_{p_2}^2]\big\|\le \int\limits_0^1\big\|\partial_{x_{p_2}}\widetilde{F}\big(t,x_{p_1},x_{p_2}^2+\theta (x_{p_2}^1-x_{p_2}^2)\big) - \partial_{x_{p_2}}\widetilde{F}(t_*,x_{p_1}^*,x^*_{p_2})\big\| d\theta\, \|x_{p_2}^1-x_{p_2}^2\|\le q(\delta,\varepsilon)\|x_{p_2}^1-x_{p_2}^2\|$ for each $(t,x_{p_1})\in \overline{U_\delta(t_*,x_{p_1}^*)}\subset \mT\times D_1$ and each $x_{p_2}^1,x_{p_2}^2\in \overline{U_\varepsilon(x_{p_2}^*)}\subset D_2$, where
$U_\delta(t_*,x_{p_1}^*)$, $U_\varepsilon(x_{p_2}^*)$ are some open neighborhoods of $(t_*,x_{p_1}^*)$, $x_{p_2}^*$, and $q(\delta,\varepsilon)=\sup\limits_{(t,x_{p_1})\in \overline{U_\delta(t_*,x_{p_1}^*)},\, \tilde{x}_{p_2}\in \overline{U_\varepsilon(x_{p_2}^*)}} \big\|\partial_{x_{p_2}}\widetilde{F}(t,x_{p_1},\tilde{x}_{p_2}) -\partial_{x_{p_2}}\widetilde{F}(t_*,x_{p_1}^*,x^*_{p_2})\big\|\to 0$ as $\delta,\varepsilon\to 0$ since the function $\partial_{x_{p_2}}\widetilde{F}(t,x_{p_1},x_{p_2})$ is continuous at the point $(t_*,x_{p_1}^*,x^*_{p_2})$. Consequently, condition \ref{InvReg-Lipsch} of Theorem~\ref{Th_GlobReg-LipschObl} is fulfilled.
 \end{proof}

Below we present the theorem that was proved in \cite{Fil.Sing-GN} for a singular DAE and, as well as Corollary~\ref{Coroll-UstLagrReg}, gives conditions for the Lagrange stability of the regular DAE \eqref{DAE}.

 \begin{theorem}\label{Th_RegUstLagr2}
Let $f\in C(\mT\times \Rn,\Rn)$, where $\mT=[t_+,\infty)\subseteq [0,\infty)$, and let the operator pencil $\lambda A+B$ be a regular pencil of index not higher than 1. Assume that condition \ref{SoglSing} of Theorem \ref{Th_GlobReg-LipschObl}, where $M_i=X_i$, $i=1,2$, holds and  condition \ref{InvReg-Lipsch} of Theorem \ref{Th_GlobReg-LipschObl} or condition \ref{InvReg} of Theorem \ref{Th_GlobReg-Obl}, where $D=\Rn$ and $M_i=D_i=X_i$, $i=1,2$, holds. In addition, let the following conditions be satisfied:
\begin{enumerate}
\addtocounter{enumi}{2}
\item\label{LagrReg} There exists a number ${R>0}$ \,\textup{(}$R$ can be sufficiently large\textup{)}, a function ${V\in C^1\big(\mT\times M_R,\R\big)}$ positive on $\mT\times M_R$, where $M_R=\{x_{p_1}\in X_1\mid \|x_{p_1}\|>R\}$, and a function ${\chi\in C(\mT\times (0,\infty),\R)}$  such that:\; $\lim\limits_{\|x_{p_1}\|\to+\infty}V(t,x_{p_1})=+\infty$ uniformly in $t$ on $\mT$;\; for each $t\in \mT$, $x_{p_1}\in M_R$, $x_{p_2}\in X_2$ such that $(t,x_{p_1}+x_{p_2})\in L_{t_+}$ the inequality \eqref{LagrDAEReg} is satisfied;  the differential inequality \eqref{L1v} does not have unbounded positive solutions for $t\in \mT$.

\item\label{LagrAs} For all $(t,x_{p_1}+x_{p_2})\in L_{t_+}$, for which $\|x_{p_1}\|\le M<\infty$\, ($M$ is an arbitrary constant), the inequality \,$\|x_{p_2}\|\le K_M$\, or the inequality \,$\|Q_2 f(t,x_{p_1}+x_{p_2})\|\le K_M$,\, where $K_M=K(M)<\infty$ is some constant, is satisfied.
 \end{enumerate}
Then the equation \eqref{DAE} is Lagrange stable (for each initial point $(t_0,x_0)\in L_{t_+}$).
 \end{theorem}

  \subsection{Global solvability of singular (nonregular) semilinear DAEs}\label{SectGlobSing}

Below, the projectors and subspaces, described in Appendix \ref{Appendix}, as well as the definitions and constructions, given in Section \ref{Preliminaries}, are used. Recall that $D_{s_i}=S_iD$, $D_i=P_iD$, $i=1,2$ (see \eqref{Dssrr}).

 \begin{theorem}\label{Th_GlobSing-LipschObl}
Let $f\in C(\mT\times D,\Rm)$, where $D\subseteq \Rn$ is some open set and $\mT=[t_+,\infty)\subseteq [0,\infty)$, and let the operator pencil $\lambda A+B$ be a singular pencil such that its regular block $\lambda A_r+B_r$, where $A_r$, $B_r$ are defined in \eqref{srAB}, is a regular pencil of index not higher than 1. Assume that there exists an open set $M_{s1}\subseteq D_{s_1}\dot+D_1$ and sets $M_{s_2}\subseteq D_{s_2}$, $M_2\subseteq D_2$ such that the following holds:
\begin{enumerate}
\item\label{SoglSing}  For any fixed ${t\in \mT}$, ${x_{s_1}+x_{p_1}\in M_{s1}}$, $x_{s_2}\in M_{s_2}$ there exists a unique ${x_{p_2}\in M_2}$ such that ${(t,x_{s_1}+x_{s_2}+x_{p_1}+x_{p_2})\in L_{t_+}}$  \textup{(}the manifold $L_{t_+}$ has the form \eqref{L_tSing} where $t_*=t_+$\textup{)}.

\item\label{InvSing-Lipsch} A function $f(t,x)$ satisfies locally a Lipschitz condition with respect to $x$ on $\mT\times D$.\, For~any fixed ${t_*\in \mT}$, ${x_*=x_{s_1}^*+x_{s_2}^*+x_{p_1}^*+x_{p_2}^*}$ \,\textup{(}${x_{s_i}^*=S_ix_*}$, ${x_{p_i}^*=P_ix_*}$, ${i=1,2}$\textup{)}\, such that ${x_{s_1}^*+x_{p_1}^*\in M_{s1}}$, ${x_{s_2}^*\in M_{s_2}}$, ${x_{p_2}^*\in M_2}$ and  ${(t_*,x_*)\in L_{t_+}}$, there exists a neighborhood $N_\delta(t_*,x_{s_1}^*,x_{s_2}^*,x_{p_1}^*)= U_{\delta_1}(t_*)\times U_{\delta_2}(x_{s_1}^*)\times N_{\delta_3}(x_{s_2}^*)\times U_{\delta_4}(x_{p_1}^*)\subset \mT\times D_{s_1}\times D_{s_2}\times D_1$,
    an open neighborhood $U_\varepsilon(x_{p_2}^*)\subset D_2$ (the numbers $\delta, \varepsilon>0$ depend on the choice of $t_*$, $x_*$) and an invertible operator $\Phi_{t_*,x_*}\in \mathrm{L}(X_2,Y_2)$ such that for each $(t,x_{s_1},x_{s_2},x_{p_1})\in N_\delta(t_*,x_{s_1}^*,x_{s_2}^*,x_{p_1}^*)$ and each $x_{p_2}^i\in U_\varepsilon(x_{p_2}^*)$, $i=1,2$, the mapping
     \begin{multline}\label{tildePsiSing}
    \widetilde{\Psi}(t,x_{s_1},x_{s_2},x_{p_1},x_{p_2}):= Q_2f(t,x_{s_1}+x_{s_2}+x_{p_1}+x_{p_2})-   \\
    -B\big|_{X_2}x_{p_2}\colon \mT\times D_{s_1}\times D_{s_2}\times D_1\times D_2\to Y_2
     \end{multline}
   satisfies the inequality
    \begin{equation}\label{ContractiveMapPsi}
   \|\widetilde{\Psi}(t,x_{s_1},x_{s_2},x_{p_1},x_{p_2}^1)- \widetilde{\Psi}(t,x_{s_1},x_{s_2},x_{p_1},x_{p_2}^2)-\Phi_{t_*,x_*} [x_{p_2}^1-x_{p_2}^2]\|\le  q(\delta,\varepsilon)\|x_{p_2}^1-x_{p_2}^2\|,
    \end{equation}
    where $q(\delta,\varepsilon)$ is such that  $\lim\limits_{\delta,\,\varepsilon\to 0} q(\delta,\varepsilon)<\|\Phi_{t_*,x_*}^{-1}\|^{-1}$.

\item\label{SingAttractor}  If $M_{s1}\neq X_{s_1}\dot+ X_1$, then the following holds.

  The component $x_{s_1}(t)+x_{p_1}(t)=(S_1+P_1)x(t)$ of each solution $x(t)$ with the initial point $(t_0,x_0)\in L_{t_+}$, for which $(S_1+P_1)x_0\in M_{s1}$, $S_2x_0\in M_{s_2}$ and $P_2x_0\in M_2$, can never leave $M_{s1}$ \textup{(}i.e., it remains in $M_{s1}$ for all $t$ from the maximal interval of existence of the solution\textup{)}.

\item\label{ExtensSing} If $M_{s1}$ is unbounded, then the following holds.

  There exists a number ${R>0}$ \,\textup{(}$R$ can be sufficiently large\textup{)}, a function ${V\in C^1\big(\mT\times M_R,\R\big)}$ positive on $\mT\times M_R$, where $M_R=\{(x_{s_1},x_{p_1})\in X_{s_1}\times X_1\mid x_{s_1}+x_{p_1}\in M_{s1},\; \|x_{s_1}+x_{p_1}\|> R\}$, and a function ${\chi\in C(\mT\times (0,\infty),\R)}$  such that:
   \begin{enumerate}[label={\upshape(\ref{ExtensSing}.\alph*)}, ref={(\ref{ExtensSing}.\alph*)},itemsep=2pt,parsep=0pt,topsep=2pt]
  \item\label{ExtensSing1}
   $\lim\limits_{\|(x_{s_1},x_{p_1})\|\to+\infty}V(t,x_{s_1},x_{p_1})=+\infty$ uniformly in $t$ on each finite interval $[a,b)\subset \mT$;

  \item\label{ExtensSing2} for each $t\in \mT$, $(x_{s_1},x_{p_1})\in M_R$, $x_{s_2}\in M_{s_2}$, $x_{p_2}\in M_2$  such that $(t,x_{s_1}+x_{s_2}+x_{p_1}+x_{p_2})\in L_{t_+}$, the derivative \eqref{dVDAEsing} of the function $V$ along the trajectories of the equations \eqref{DAEsysExtDE1}, \eqref{DAEsysExtDE2} satisfies the inequality
   \begin{equation}\label{LagrDAEsing}
  \dot{V}_{\eqref{DAEsysExtDE1},\eqref{DAEsysExtDE2}}(t,x_{s_1},x_{p_1})\le \chi\big(t,V(t,x_{s_1},x_{p_1})\big);
   \end{equation}

  \item\label{GlobSolv} the differential inequality \eqref{L1v} does not have positive solutions with finite escape time.
   \end{enumerate}
 \end{enumerate}
Then for each initial point ${(t_0,x_0)\in L_{t_+}}$ such that $(S_1+P_1)x_0\in M_{s1}$, $S_2x_0\in M_{s_2}$ and $P_2x_0\in M_2$, the IVP \eqref{DAE}, \eqref{ini} has a unique global solution $x(t)$ for which the choice of the function $\phi_{s_2}\in C([t_0,\infty),M_{s_2})$ with the initial value  ${\phi_{s_2}(t_0)=S_2 x_0}$ uniquely defines the component $S_2x(t)=\phi_{s_2}(t)$ when ${\rank(\lambda A+B)<n}$ \,\textup{(}when ${\rank(\lambda A+B)=n}$, the component $S_2 x$ is absent\textup{)}.
 \end{theorem}

 \begin{corollary}\label{Coroll-GlobSing1}
Theorem~\ref{Th_GlobSing-LipschObl} remains valid if condition \ref{SingAttractor} is replaced by
\begin{enumerate}[label={\upshape\arabic*.}, ref={\upshape\arabic*}, itemsep=3pt,parsep=0pt,topsep=4pt,leftmargin=0.6cm]
 \addtocounter{enumi}{2}
\item\label{SingAttractCoroll} In the case when $M_{s1}\neq X_{s_1}\dot+ X_1$, the following holds.

Let ${D_{\Upsilon,s_1}\subseteq X_{s_1}}$, ${D_{\Upsilon,i}\subseteq X_i}$, $i=1,2$, be sets such that the function $\Upsilon$ of the form \eqref{Upsilon} is defined and continuous for all ${t\in \mT}$, ${x_{s_1}\in D_{\Upsilon,s_1}}$, ${x_{s_2}\in D_{s_2}}$, ${x_{p_i}\in D_{\Upsilon,i}}$, ${i=1,2}$, where $D_{\Upsilon,s_1}\supset  D_{s_1}$, $D_{\Upsilon,i}\supset D_i$ and $D_{\Upsilon,s_1}\times D_{s_2}\times D_{\Upsilon,1}\times D_{\Upsilon,2} \ne D_{s_1}\times D_{s_2}\times D_1\times D_2$, if the domain of definition of $\Upsilon$ can be extended to ${\mT\times D_{\Upsilon,s_1}\times D_{s_2}\times D_{\Upsilon,1}\times D_{\Upsilon,2}}$ in this way, and ${D_{\Upsilon,s_1}= D_{s_1}}$, ${D_{\Upsilon,i}= D_i}$ otherwise.\,
Let $D_{c,s1}\subseteq X_{s_1}\dot+X_1$, $D_{c,2}\subseteq X_2$ be sets such that for any fixed ${t\in \mT}$, ${x_{s_1}+x_{p_1}\in D_{c,s1}\supset M_{s1}}$, $x_{s_2}\in M_{s_2}$ there exists a unique $x_{p_2}\in D_{c,2}\supset M_2$ such that ${(t,x_{s_1}+x_{s_2}+x_{p_1}+x_{p_2})\in L_{t_+}}$ and $D_{c,s1}\ne M_{s1}$ or $D_{c,2}\ne M_2$ if such sets exist, and $D_{c,s1}=M_{s1}$, $D_{c,2}= M_2$ otherwise.
Further, let ${\widetilde{D}_{s1}:= D_{\Upsilon,s_1}\dot+D_{\Upsilon,1}\cap D_{c,s1}}$, ${\widetilde{D_2}:= D_{\Upsilon,2}\cap D_{c,2}}$ if the function $\Upsilon$ depends on $x_{p_2}$, and ${\widetilde{D}_{s1}:= D_{\Upsilon,s_1}\dot+D_{\Upsilon,1}}$ if $\Upsilon$ does not depend on $x_{p_2}$.

Below, $\Upsilon$ is considered as the function \eqref{Upsilon} with the domain of definition $\mT\times S_1\widetilde{D}_{s1}\times M_{s_2}\times P_1\widetilde{D}_{s1}\times \widetilde{D_2}$ if it depends on $x_{p_2}$ and $\mT\times S_1\widetilde{D}_{s1}\times M_{s_2}\times P_1\widetilde{D}_{s1}$  if it does not depend on $x_{p_2}$.

Assume that there exists a function $W\in C(\mT\times X_{s_1}\times X_1,\R)$ and for each sufficiently small number ${r>0}$  there exists a closed set $K_r=\{x_{s_1}+x_{p_1}\in M_{s1}\mid \rho(K_r,M_{s1}^c)=r\}$  \,\textup{(}${M_{s1}^c=(X_{s_1}\dot+ X_1)\setminus M_{s1}}$,\, $\rho(K_r,M_{s1}^c)=\! \inf\limits_{k\,\in\, K_r,\, m\,\in\, M_{s1}^c} \|m-k\|$\,\textup{)} such that
 $$
W(t_1,x_{s_1}^1,x_{p_1}^1)<W(t_2,x_{s_1}^2,x_{p_1}^2)
 $$
for every $x_{s_1}^1+x_{p_1}^1\in K_r$, $x_{s_1}^2+x_{p_1}^2\in M_{s1}^c\cap \widetilde{D}_{s1}$ and $t_1,t_2\in \mT$ such that $t_1\le t_2$, and, in addition, $W(t,x_{s_1},x_{p_1})$ has the continuous partial derivatives on $\mT\times S_1K_r^c\times P_1K_r^c$ \,\textup{(}${K_r^c=(X_{s_1}\dot+ X_1)\setminus K_r}$\textup{)} and the inequality
 \begin{equation}\label{SingAttractIneq}
\dot{W}_{\eqref{DAEsysExtDE1},\eqref{DAEsysExtDE2}}(t,x_{s_1},x_{p_1})=\partial_t W(t,x_{s_1},x_{p_1})+ \partial_{(x_{s_1},x_{p_1})} W(t,x_{s_1},x_{p_1})\cdot \Upsilon(t,x_{s_1},x_{s_2},x_{p_1},x_{p_2})\le 0
 \end{equation}
holds for each ${t\in \mT}$, ${x_{s_1}+x_{p_1}\in K_r^c\cap \widetilde{D}_{s1}}$, ${x_{s_2}\in M_{s_2}}$, ${x_{p_2}\in \widetilde{D_2}}$ such that ${(t,x_{s_1}+x_{s_2}+x_{p_1}+x_{p_2})\in L_{t_+}}$ \textup{(}if $\Upsilon$ does not depend on $x_{p_2}$, then \eqref{SingAttractIneq} holds for each $t\in \mT$, ${x_{s_1}+x_{p_1}\in K_r^c\cap \widetilde{D}_{s1}}$, ${x_{s_2}\in M_{s_2}}$\textup{)}.
\end{enumerate}
 \end{corollary}

 \begin{corollary}\label{Coroll-GlobSing2}
Theorem~\ref{Th_GlobSing-LipschObl} remains valid if condition \ref{ExtensSing} is replaced by
\begin{enumerate}[label={\upshape\arabic*.}, ref={\upshape\arabic*}, itemsep=3pt,parsep=0pt,topsep=4pt,leftmargin=0.6cm]
 \addtocounter{enumi}{3}
\item\label{ExtensSingCoroll} If $M_{s1}$ is unbounded, then the following holds.

  There exists a number ${R>0}$, a function ${V\in C^1\big(\mT\times M_R,\R\big)}$ positive on $\mT\times M_R$, where $M_R=\{(x_{s_1},x_{p_1})\in X_{s_1}\times X_1\mid x_{s_1}+x_{p_1}\in M_{s1},\; \|x_{s_1}+x_{p_1}\|> R\}$, and functions ${k\in C(\mT,\R)}$, ${U\in C(0,\infty)}$ such that:\; $\lim\limits_{\|(x_{s_1},x_{p_1})\|\to+\infty}V(t,x_{s_1},x_{p_1})=+\infty$ uniformly in $t$ on each finite interval $[a,b)\subset \mT$;\;
  for each ${t\in \mT}$, ${(x_{s_1},x_{p_1})\in M_R}$, ${x_{s_2}\in M_{s_2}}$, ${x_{p_2}\in M_2}$  such that $(t,x_{s_1}+x_{s_2}+x_{p_1}+x_{p_2})\in L_{t_+}$, the inequality \,${\dot{V}_{\eqref{DAEsysExtDE1},\eqref{DAEsysExtDE2}}(t,x_{s_1},x_{p_1})\le k(t)\, U\big(V(t,x_{p_1})\big)}$\, holds;\;  $\int\limits_{{\textstyle v}_0}^{\infty}\dfrac{dv}{U(v)} =\infty$\, \textup{(}$v_0>0$ is a constant\textup{)}.
\end{enumerate}
 \end{corollary}

 \begin{corollary}\label{Coroll-UstLagrSing}
If in the conditions of Theorem~\ref{Th_GlobSing-LipschObl} the sets $M_{s1}$, $M_{s_2}$ and $M_2$ are bounded, then the equation \eqref{DAE} is Lagrange stable for the initial points $(t_0,x_0)\in L_{t_+}$ for which $(S_1+P_1)x_0\in M_{s1}$, $S_2x_0\in M_{s_2}$ and $P_2x_0\in M_2$.
 \end{corollary}

 \begin{theorem}\label{Th_GlobSing-Obl}
Theorem~\ref{Th_GlobSing-LipschObl} remains valid if condition \ref{InvSing-Lipsch} is replaced by
\begin{enumerate}[label={\upshape\arabic*.}, ref={\upshape\arabic*}, itemsep=3pt,parsep=0pt,topsep=4pt,leftmargin=0.6cm]
\addtocounter{enumi}{1}
\item\label{InvSing} A function $f(t,x)$ has the continuous partial derivative with respect to $x$ on $\mT\times D$.\;  For~any fixed $t_*\in \mT$, ${x_*=x_{s_1}^*+x_{s_2}^*+x_{p_1}^*+x_{p_2}^*}$ such that $x_{s_1}^*+x_{p_1}^*\in M_{s1}$, $x_{s_2}^*\in M_{s_2}$, $x_{p_2}^*\in M_2$ and ${(t_*,x_*)\in L_{t_+}}$, the operator
     \begin{equation}\label{funcPhiInvSing}
    \Phi_{t_*,x_*}:=\left[\partial_x (Q_2f)(t_*,x_*)- B\right] P_2\colon X_2\to Y_2
     \end{equation}
    has the inverse $\Phi_{t_*,x_*}^{-1}\in \mathrm{L}(Y_2,X_2)$.
\end{enumerate}
 \end{theorem}
   \begin{remark}[cf. {\cite[Remark 3.2]{Fil.Sing-GN}}]
The operator  $\Phi_{t_*,x_*}$ \eqref{funcPhiInvSing} (as well as \eqref{funcPhiInvReg}) is defined as an operator from $X_2$ into $Y_2$, however, in general, the operator defined by the formula from \eqref{funcPhiInvSing} is an operator from $\Rn$ into $\Rm$ with the rang $Y_2$, i.e., $\widehat{\Phi}_{t_*,x_*}:=\left[\partial_x (Q_2f)(t_*,x_*)- B\right] P_2\in \mathrm{L}(\Rn,\Rm)$ and $\widehat{\Phi}_{t_*,x_*}\Rn= Y_2$ \,($t_*$, $x_*$ are fixed). Since $\Phi_{t_*,x_*}=\widehat{\Phi}_{t_*,x_*}\big|_{X_2}$ and it is assumed that $\Phi_{t_*,x_*}$ is invertible, then $\widehat{\Phi}_{t_*,x_*}\Rn=\widehat{\Phi}_{t_*,x_*}X_2=Y_2$ \,(${X_s\dot+X_1=\Ker(\widehat{\Phi}_{t_*,x_*})}$).
The operator $\widehat{\Phi}_{t_*,x_*}$ has the semi-inverse $\widehat{\Phi}_{t_*,x_*}^{(-1)}$, i.e., the operator $\widehat{\Phi}_{t_*,x_*}^{(-1)}\in \mathrm{L}(\Rm,\Rn)$ such that $\widehat{\Phi}_{t_*,x_*}^{(-1)}\Rm= \widehat{\Phi}_{t_*,x_*}^{(-1)}Y_2=X_2$ and $\Phi_{t_*,x_*}^{-1}=\widehat{\Phi}_{t_*,x_*}^{(-1)}\big|_{Y_2}$, which is defined by the relations \;$\widehat{\Phi}_{t_*,x_*}^{(-1)} \widehat{\Phi}_{t_*,x_*}= P_2$, $\widehat{\Phi}_{t_*,x_*} \widehat{\Phi}_{t_*,x_*}^{(-1)}= Q_2$ and $\widehat{\Phi}_{t_*,x_*}^{(-1)}= P_2\widehat{\Phi}_{t_*,x_*}^{(-1)}$.
 \end{remark}

 \begin{proof}[The proof of Theorems \ref{Th_GlobSing-LipschObl}, \ref{Th_GlobSing-Obl} and Corollaries \ref{Coroll-GlobSing1}, \ref{Coroll-GlobSing2}, \ref{Coroll-UstLagrSing}]

The proofs of Theorems \ref{Th_GlobSing-LipschObl}, \ref{Th_GlobSing-Obl} and Corollaries \ref{Coroll-GlobSing1}, \ref{Coroll-GlobSing2}, \ref{Coroll-UstLagrSing} are carried out in a similar way as the proofs of Theorems \ref{Th_GlobReg-LipschObl}, \ref{Th_GlobReg-Obl} and Corollaries \ref{Coroll-GlobReg1}, \ref{Coroll-GlobReg2}, \ref{Coroll-UstLagrReg}, respectively.

In addition, the proofs of Theorems \ref{Th_GlobSing-LipschObl}, \ref{Th_GlobSing-Obl} for $D=\Rn$, $M_{s1}=X_{s_1}\dot+X_1$ and $M_2=X_2$ in their conditions are carried out in the same way as the proofs of the theorems \cite[Theorems 3.4, 3.2]{Fil.Sing-GN} (note that $D_{s_2}$ in the theorems \cite[Theorems 3.4, 3.2]{Fil.Sing-GN} denotes some set in $X_{s_2}$, not the subset $D_{s_2}=S_2D$ of $D$ as in Theorems \ref{Th_GlobSing-LipschObl}, \ref{Th_GlobSing-Obl}).
 \end{proof}

Note that the theorem \cite[Theorem 4.1]{Fil.Sing-GN}, as well as Corollary~\ref{Coroll-UstLagrSing}, gives conditions for the Lagrange stability of the singular DAE \eqref{DAE}, but in the case when $D=\Rn$ and $M_{s1}=X_{s_1}\dot+X_1$ (conditions \ref{SoglSing} and \ref{InvSing-Lipsch} of Theorem \ref{Th_GlobSing-LipschObl}, where $D=\Rn$, $M_{s1}=X_{s_1}\dot+X_1$ and $M_2=X_2$, are similar to those contained in \cite[Theorem 4.1]{Fil.Sing-GN}); thus, Corollary~\ref{Coroll-UstLagrSing} gives more general conditions.

 \section{The blow-up of solutions}\label{SectLagrUnst}

 \subsection{The blow-up of solutions (Lagrange instability) of regular semilinear DAEs}\label{SectLagrUnstReg}

 \begin{theorem}\label{Th_RegNeLagr}
Let $f\in C(\mT\times D,\Rn)$, where $D\subseteq \Rn$ is some open set and $\mT=[t_+,\infty)\subseteq [0,\infty)$, and let the operator pencil $\lambda A+B$ be a regular pencil of index not higher than 1. Assume that there exists an open (unbounded) set $M_1\subseteq D_1$ and a set $M_2\subseteq D_2$ such that condition \ref{SoglReg} of Theorem \ref{Th_GlobReg-LipschObl}, condition \ref{InvReg-Lipsch} of Theorem \ref{Th_GlobReg-LipschObl} (or condition \ref{InvReg} of Theorem \ref{Th_GlobReg-Obl})
and  condition \ref{RegAttractor} of Theorem \ref{Th_GlobReg-LipschObl} (or condition \ref{RegAttractCoroll} of Corollary \ref{Coroll-GlobReg1}) hold and the following condition holds:
\begin{enumerate}[label={\upshape\arabic*.}, ref={\upshape\arabic*}, itemsep=3pt,parsep=0pt,topsep=4pt,leftmargin=0.6cm]
\addtocounter{enumi}{3}
\item\label{NeUstLagrReg}
   There exists a function $V\in C^1\big(\mT\times M_1,\R\big)$ positive on $\mT\times M_1$ and a function ${\chi\in C(\mT\times (0,\infty),\R)}$ such that:
   \begin{enumerate}[label={\upshape(\ref{NeUstLagrReg}.\alph*)}, ref={(\ref{NeUstLagrReg}.\alph*)},itemsep=2pt,parsep=0pt,topsep=2pt]
  \item for each $t\in \mT$, $x_{p_1}\in M_1$, $x_{p_2}\in M_2$  such that $(t,x_{p_1}+x_{p_2})\in L_{t_+}$, the derivative \eqref{dVDAEreg} of the function $V$ along the trajectories of the equation \eqref{DAEsysExtReg1} satisfies the inequality
   \begin{equation*}
    \dot{V}_\eqref{DAEsysExtReg1}(t,x_{p_1})\ge \chi\big(t,V(t,x_{p_1})\big);
   \end{equation*}

  \item\label{NeustLagrIneq}
  the differential inequality
   \begin{equation}\label{L2v}
     \dot{v}\ge \chi(t,v)\qquad (t\in \mT)
   \end{equation}
    does not have global positive solutions.
  \end{enumerate}
\end{enumerate}
Then for each initial point $(t_0,x_0)\in L_{t_+}$ for which $P_ix_0\in M_i$, $i=1,2$, the IVP \eqref{DAE}, \eqref{ini} has a unique solution $x(t)$  and this solution has a finite escape time (i.e., is blow-up in finite time).
 \end{theorem}
  \begin{corollary}\label{Coroll-RegNeLagr}
Theorem~\ref{Th_RegNeLagr} remains valid if condition \ref{NeUstLagrReg} is replaced by
\begin{enumerate}[label={\upshape\arabic*.}, ref={\upshape\arabic*}, itemsep=3pt,parsep=0pt,topsep=4pt,leftmargin=0.6cm]
 \addtocounter{enumi}{3}
\item\label{NeUstLagrRegCoroll}
  There exists a function $V\in C^1\big(\mT\times M_1,\R\big)$ positive on $\mT\times M_1$ and functions ${k\in C(\mT,\R)}$, ${U\in C(0,\infty)}$ such that:\; for each ${t\in \mT}$, ${x_{p_1}\in M_1}$, ${x_{p_2}\in M_2}$  such that $(t,x_{p_1}+x_{p_2})\in L_{t_+}$ the inequality \;${\dot{V}_\eqref{DAEsysExtReg1}(t,x_{p_1})\ge k(t)\, U\big(V(t,x_{p_1})\big)}$\, holds;\; ${\int\limits_{k_0}^{\infty}k(t)dt=\infty}$ and ${\int\limits_{{\textstyle v}_0}^{\infty}\dfrac{dv}{U(v)}<\infty}$  (\,${k_0,v_0>0}$ are constants).
\end{enumerate}
  \end{corollary}

 \begin{proof}[The proof of Theorem~\ref{Th_RegNeLagr}]
Just as in the proof of Theorem~\ref{Th_GlobReg-LipschObl}, or Theorem~\ref{Th_GlobReg-Obl}, or Corollary \ref{Coroll-GlobReg1} (depending on which conditions hold), it is proved that there exists a unique solution $x_{p_1}(t)$ of the IVP \eqref{DAEsysExtReg1eta}, \eqref{RegDEeta_ini} on the maximal interval of existence $J_{max}$ and either $J_{max}=[t_0,\infty)$, or $J_{max}=[t_0,\beta)$ where $\beta<\infty$ and $\lim\limits_{t\to\beta-0}\|x_{p_1}(t)\|=\infty$.
In addition, it is proved that the IVP \eqref{DAE}, \eqref{ini} has the unique solution $x(t)=x_{p_1}(t)+\eta(t,x_{p_1}(t))$ on the same maximal interval of existence $J_{max}$. This holds for each initial point $(t_0,x_0)\in L_{t_+}$ for which $P_ix_0\in M_i$, $i=1,2$.
Further, from condition~\ref{NeUstLagrReg} it follows that there exists a function $V\in C^1\big(\mT\times M_1,\R\big)$ positive on $\mT\times M_1$ and a function $\chi\in C(\mT\times (0,\infty),\R)$ such that $\dot{V}_{\eqref{DAEsysExtReg1eta}}(t,x_{p_1})\ge \chi\big(t,V(t,x_{p_1})\big)$ for each $t\in \mT$, $x_{p_1}\in M_1$ and the inequality \eqref{L2v} does not have global positive solutions.
Then, using \cite[Chapter~IV, Theorem~XIV]{LaSal-Lef},  we obtain that the solution $x_{p_1}(t)$ has a finite escape time and hence $J_{max}=[t_0,\beta)$, $\beta<\infty$ ($\lim\limits_{t\to \beta-0}\|x_{p_1}(t)\|= \infty$). Consequently, the solution $x(t)$ of the IVP  \eqref{DAE}, \eqref{ini} has a finite escape time.
 \end{proof}

  \begin{proof}[The proof of Corollary \ref{Coroll-RegNeLagr}]
Let $\chi(t,v):=k(t)\,U(v)$ where ${k\in C(\mT,\R)}$, ${U\in C(0,\infty)}$ satisfies the relations ${\int\limits_{k_0}^{\infty}k(t)dt=\infty}$, ${\int\limits_{{\textstyle v}_0}^{\infty}\dfrac{dv}{U(v)}<\infty}$  (\,${k_0,v_0>0}$), then the differential inequality \eqref{L2v} does not have global positive solutions (see, e.g., \cite{LaSal-Lef}). Hence,  condition \ref{NeUstLagrReg} of Theorem~\ref{Th_RegNeLagr}, where the function $\chi$ has the form $\chi(t,v)=k(t)\,U(v)$, holds, and all conditions of Theorem~\ref{Th_RegNeLagr} are fulfilled.
 \end{proof}

 \subsection{The blow-up of solutions (Lagrange instability) of singular semilinear DAEs}\label{SectLagrUnstSing}

 \begin{theorem}\label{Th_SingNeLagr}
Let $f\in C(\mT\times D,\Rm)$, where $D\subseteq \Rn$ is some open set and $\mT=[t_+,\infty)\subseteq [0,\infty)$, and let the operator pencil $\lambda A+B$ be a singular pencil such that its regular block $\lambda A_r+B_r$, where $A_r$, $B_r$ are defined in \eqref{srAB}, is a regular pencil of index not higher than 1. Assume that there exists an open (unbounded) set $M_{s1}\subseteq D_{s_1}\dot+D_1$ and sets $M_{s_2}\subseteq D_{s_2}$, $M_2\subseteq D_2$ such that condition \ref{SoglSing} of Theorem \ref{Th_GlobSing-LipschObl}, condition \ref{InvSing-Lipsch} of Theorem \ref{Th_GlobSing-LipschObl} (or condition \ref{InvSing} of Theorem \ref{Th_GlobSing-Obl})  and  condition \ref{SingAttractor} of Theorem \ref{Th_GlobSing-LipschObl} (or condition \ref{SingAttractCoroll} of Corollary \ref{Coroll-GlobSing1}) hold and the following condition holds:
\begin{enumerate}[label={\upshape\arabic*.}, ref={\upshape\arabic*}, itemsep=3pt,parsep=0pt,topsep=4pt,leftmargin=0.6cm]
\addtocounter{enumi}{3}
\item\label{NeUstLagrSing}
   There exists a function ${V\in C^1\big(\mT\times \widehat{M}_{s1},\R\big)}$ positive on $\mT\times \widehat{M}_{s1}$, where $\widehat{M}_{s1}=\{(x_{s_1},x_{p_1})\in X_{s_1}\times X_1 \mid x_{s_1}+x_{p_1}\in M_{s1}\}$, and a function ${\chi\in C(\mT\times (0,\infty),\R)}$ such that:
   \begin{enumerate}[label={\upshape(\ref{NeUstLagrSing}.\alph*)}, ref={(\ref{NeUstLagrSing}.\alph*)},itemsep=2pt,parsep=0pt,topsep=2pt]
  \item for each $t\in \mT$, $(x_{s_1},x_{p_1})\in \widehat{M}_{s1}$, $x_{s_2}\in M_{s_2}$, $x_{p_2}\in M_2$  such that $(t,x_{s_1}+x_{s_2}+x_{p_1}+x_{p_2})\in L_{t_+}$, the derivative \eqref{dVDAEsing} of the function $V$ along the trajectories of the equations \eqref{DAEsysExtDE1}, \eqref{DAEsysExtDE2} satisfies the inequality
   \begin{equation}\label{NeLagrDAESing}
    \dot{V}_{\eqref{DAEsysExtDE1},\eqref{DAEsysExtDE2}}(t,x_{s_1},x_{p_1})\ge \chi\big(t,V(t,x_{s_1},x_{p_1})\big);
   \end{equation}

  \item the differential inequality \eqref{L2v} does not have global positive solutions.
  \end{enumerate}
\end{enumerate}
Then for each initial point $(t_0,x_0)\in L_{t_+}$, for which $(S_1+P_1)x_0\in M_{s1}$, $S_2x_0\in M_{s_2}$ and $P_2x_0\in M_2$, the IVP \eqref{DAE}, \eqref{ini} has a unique solution $x(t)$ for which the choice of the function $\phi_{s_2}\in C([t_0,\infty),M_{s_2})$ with the initial value  ${\phi_{s_2}(t_0)=S_2 x_0}$ uniquely defines the component $S_2x(t)=\phi_{s_2}(t)$ when ${\rank(\lambda A+B)<n}$ \,\textup{(}when ${\rank(\lambda A+B)=n}$, the component $S_2 x$ is absent\textup{)}, and this solution has a finite escape time (i.e., is blow-up in finite time).
 \end{theorem}

 \begin{corollary}\label{Coroll-SingNeLagr}
Theorem~\ref{Th_SingNeLagr} remains valid if condition \ref{NeUstLagrSing} is replaced by
\begin{enumerate}[label={\upshape\arabic*.}, ref={\upshape\arabic*}, itemsep=3pt,parsep=0pt,topsep=4pt,leftmargin=0.6cm]
 \addtocounter{enumi}{3}
\item\label{NeUstLagrSingCoroll}
  There exists a function ${V\in C^1\big(\mT\times \widehat{M}_{s1},\R\big)}$ positive on $\mT\times \widehat{M}_{s1}$, where $\widehat{M}_{s1}=\{(x_{s_1},x_{p_1})\in X_{s_1}\times X_1 \mid x_{s_1}+x_{p_1}\in M_{s1}\}$,  and functions ${k\in C(\mT,\R)}$, ${U\in C(0,\infty)}$ such that:\; for each $t\in \mT$, $(x_{s_1},x_{p_1})\in \widehat{M}_{s1}$, $x_{s_2}\in M_{s_2}$, $x_{p_2}\in M_2$ such that $(t,x_{s_1}+x_{s_2}+x_{p_1}+x_{p_2})\in L_{t_+}$  the inequality \;$\dot{V}_{\eqref{DAEsysExtDE1},\eqref{DAEsysExtDE2}}(t,x_{s_1},x_{p_1})\ge k(t)\, U\big(V(t,x_{s_1},x_{p_1})\big)$\, holds;\; ${\int\limits_{k_0}^{\infty}k(t)dt=\infty}$ and ${\int\limits_{{\textstyle v}_0}^{\infty}\dfrac{dv}{U(v)}<\infty}$  (\,${k_0,v_0>0}$ are constants).
\end{enumerate}
 \end{corollary}

 \begin{proof}[The proof of Theorem \ref{Th_SingNeLagr} and Corollary \ref{Coroll-SingNeLagr}]
The proofs of Theorem \ref{Th_SingNeLagr} and Corollary \ref{Coroll-SingNeLagr} are carried out in a similar way as the proofs of Theorem \ref{Th_RegNeLagr} and Corollary \ref{Coroll-RegNeLagr}, respectively.  In addition, the proof of Theorem \ref{Th_SingNeLagr} for $D=\Rn$, $M_{s1}=X_{s_1}\dot+X_1$ and $M_2=X_2$ in its conditions is carried out in a similar way as the proof of the theorem \cite[Theorem 5.1]{Fil.Sing-GN} (note that $D_{s_2}$ in \cite[Theorem 5.1]{Fil.Sing-GN} denotes some set in $X_{s_2}$, not the subset $D_{s_2}=S_2D$ of $D$ as in Theorem \ref{Th_SingNeLagr}).
 \end{proof}

 \section{The criterion of global solvability}\label{CritGlobSolv}

 \subsection{The criterion of global solvability of regular semilinear DAEs}\label{CritGlobSolvReg}

  \begin{theorem}\label{Th_CritGlobReg}
Let $f\in C(\mT\times D,\Rn)$, where $D\subseteq \Rn$ is some open set and $\mT=[t_+,\infty)\subseteq [0,\infty)$, and let $\lambda A+B$ be a regular pencil of index not higher than 1. Let there exist an open set $M_1\subseteq D_1$ and a set $M_2\subseteq D_2$ such that conditions \ref{SoglReg}, \ref{InvReg-Lipsch} and \ref{RegAttractor} of Theorem \ref{Th_GlobReg-LipschObl} hold.

Then the IVP \eqref{DAE}, \eqref{ini} has a unique solution $x(t)$ for each initial point ${(t_0,x_0)\in L_{t_+}}$ such that $P_ix_0\in M_i$, $i=1,2$, which is global (i.e., exists on $[t_0,\infty)$) if condition \ref{ExtensReg} of Theorem \ref{Th_GlobReg-LipschObl} is satisfied and has a finite escape time (i.e., exists on $[t_0,\beta)$ where $\beta<\infty$ and $\lim\limits_{t\to\beta-0}\|x_{p_1}(t)\|=\infty$) if condition \ref{NeUstLagrReg} of Theorem \ref{Th_RegNeLagr} is satisfied.
 \end{theorem}
 \begin{corollary}\label{Coroll_CritGlobReg}
Theorem \ref{Th_CritGlobReg} remains valid if:
\begin{itemize}
[itemsep=1pt,parsep=0pt,topsep=1pt,listparindent=0cm,leftmargin=0.6cm]
\item  condition \ref{InvReg-Lipsch} of Theorem~\ref{Th_GlobReg-LipschObl} is replaced by condition \ref{InvReg} of Theorem~\ref{Th_GlobReg-Obl};
\item  condition \ref{RegAttractor} of Theorem~\ref{Th_GlobReg-LipschObl} is replaced by condition \ref{RegAttractCoroll} of Corollary~\ref{Coroll-GlobReg1};
\item  condition \ref{ExtensReg} of Theorem~\ref{Th_GlobReg-LipschObl} is replaced by condition \ref{ExtensRegCoroll} of Corollary~\ref{Coroll-GlobReg2},
\item condition \ref{NeUstLagrReg} of Theorem~\ref{Th_RegNeLagr} is replaced by condition \ref{NeUstLagrRegCoroll} of Corollary~\ref{Coroll-RegNeLagr}.
\end{itemize}
 \end{corollary}

 \begin{proof} Both Theorem \ref{Th_CritGlobReg} and Corollary \ref{Coroll_CritGlobReg} follow directly from the theorems and corollaries proved above.
 \end{proof}

 \subsection{The criterion of global solvability of singular semilinear DAEs}\label{CritGlobSolvSing}

  \begin{theorem}\label{Th_CritGlobSing}
Let $f\in C(\mT\times D,\Rm)$, where $D\subseteq \Rn$ is some open set and $\mT=[t_+,\infty)\subseteq [0,\infty)$, and let the operator pencil $\lambda A+B$ be a singular pencil such that its regular block $\lambda A_r+B_r$, where $A_r$, $B_r$ are defined in \eqref{srAB}, is a regular pencil of index not higher than 1. Let there exist an open set $M_{s1}\subseteq D_{s_1}\dot+D_1$ and sets $M_{s_2}\subseteq D_{s_2}$, $M_2\subseteq D_2$ such that conditions \ref{SoglSing}, \ref{InvSing-Lipsch} and \ref{SingAttractor} of Theorem \ref{Th_GlobSing-LipschObl} hold.

Then for each initial point ${(t_0,x_0)\in L_{t_+}}$ such that $(S_1+P_1)x_0\in M_{s1}$, $S_2x_0\in M_{s_2}$ and $P_2x_0\in M_2$, the IVP \eqref{DAE}, \eqref{ini} has a unique solution $x(t)$ for which the choice of the function $\phi_{s_2}\in C([t_0,\infty),M_{s_2})$ with the initial value  ${\phi_{s_2}(t_0)=S_2 x_0}$ uniquely defines the component $S_2x(t)=\phi_{s_2}(t)$ when ${\rank(\lambda A+B)<n}$ \,\textup{(}when ${\rank(\lambda A+B)=n}$, the component $S_2 x$ is absent\textup{)}, and this solution is global if condition \ref{ExtensSing} of Theorem \ref{Th_GlobSing-LipschObl} holds and has a finite escape time if condition \ref{NeUstLagrSing} of Theorem \ref{Th_SingNeLagr} holds.
 \end{theorem}
 \begin{corollary}\label{Coroll_CritGlobSing}
Theorem \ref{Th_CritGlobSing} remains valid if:
\begin{itemize}
[itemsep=1pt,parsep=0pt,topsep=1pt,listparindent=0cm,leftmargin=0.6cm]
\item  condition \ref{InvSing-Lipsch} of Theorem~\ref{Th_GlobSing-LipschObl} is replaced by condition \ref{InvSing} of Theorem~\ref{Th_GlobSing-Obl};
\item  condition \ref{SingAttractor} of Theorem~\ref{Th_GlobSing-LipschObl} is replaced by condition \ref{SingAttractCoroll} of Corollary~\ref{Coroll-GlobSing1};
\item  condition \ref{ExtensSing} of Theorem~\ref{Th_GlobSing-LipschObl} is replaced by condition \ref{ExtensSingCoroll} of Corollary~\ref{Coroll-GlobSing2},
\item condition \ref{NeUstLagrSing} of Theorem~\ref{Th_SingNeLagr} is replaced by condition \ref{NeUstLagrSingCoroll} of Corollary~\ref{Coroll-SingNeLagr}.
\end{itemize}
 \end{corollary}

 \begin{proof} Both Theorem \ref{Th_CritGlobSing} and Corollary \ref{Coroll_CritGlobSing} follow directly from the theorems and corollaries proved above.
 \end{proof}

  \section{Examples}\label{Sect-Example}

 \subsection{Example 1 (a regular DAE)}\label{Example1}

Consider the system of differential and algebraic equations
 \begin{align}
\dot{x}_1 &=x_1^2,    \label{Ex1DE} \\
x_1+x_2 &=\varphi(t), \label{Ex1AE}
 \end{align}
where $\varphi\in C(\mT,\R)$, $\mT= [0,\infty)$, and $x_i=x_i(t)$, $i=1,2$, are real functions. The system \eqref{Ex1DE}, \eqref{Ex1AE} can be represented in the form of the DAE $\frac{d}{dt}[Ax]+Bx=f(t,x)$, i.e., \eqref{DAE}, where
\begin{equation}\label{Ex1DAE} 
x=\begin{pmatrix} x_1 \\  x_2 \end{pmatrix},\quad
A=\begin{pmatrix} 1 & 0 \\ 0 & 0 \end{pmatrix},\quad
B=\begin{pmatrix} 0 & 0 \\ 1 & 1 \end{pmatrix},\quad
f(t,x)=\begin{pmatrix} x_1^2 \\ \varphi(t) \end{pmatrix},
\end{equation}
$f\in C(\mT\times D,\R^2)$, $D=\R^2$.
It can be readily verified that the characteristic pencil $\lambda A+B$ is a regular pencil of index 1. The initial condition is given as $x(t_0)=x_0$, i.e., \eqref{ini}. It is clear that the initial values $t_0$, $x_0=(x_{1,0},x_{2,0})^\T$ must satisfy the equation \eqref{Ex1AE}, that is, $x_{1,0}+x_{2,0}=\varphi(t_0)$. The consistency condition $(t_0,x_0)\in L_0$ means the same (as shown below).

 \smallskip
\emph{A solution of the IVP for the system \eqref{Ex1DE}, \eqref{Ex1AE} with the initial condition \eqref{ini} has the form $x_1(t)=\dfrac{1}{x_{1,0}^{-1}+t_0-t}$,\, $x_2(t)=\varphi(t)-x_1(t)$ when $x_{1,0}\ne 0$, and $x_1(t)\equiv 0$,\, $x_2(t)=\varphi(t)$ when $x_{1,0}= 0$} (accordingly, a solution of the DAE \eqref{DAE}, \eqref{Ex1DAE} is $x(t)=(x_1(t),x_2(t))^\T$ with the specified components $x_1(t)$, $x_2(t)$). \emph{Obviously, the solution is global} (i.e., exists on the interval $[t_0,\infty)$) \emph{if $x_{1,0}<0$, and it has the finite escape time $t_0+x_{1,0}^{-1}$} (i.e., exists on the finite interval $[t_0,T)$, where $T=t_0+x_{1,0}^{-1}$, and $\lim\limits_{t\to T-0} \|x(t)\|=+\infty$)\, \emph{if $x_{1,0}>0$. Let us prove this, using the theorems obtained above.} First we prove that the solution is global if $x_{1,0}<0$, then we prove that the solution has the finite escape time if $x_{1,0}>0$; in both cases, $x_{2,0}\in \R$ is such that \eqref{Ex1AE} is satisfied, i.e., $x_{2,0}=\varphi(t_0)-x_{1,0}$.

 \smallskip
First, construct the pairs $P_i\colon \R^2\to X_i$, $i=1,2$, and $Q_j\colon \R^2\to Y_j$, $j=1,2$, of mutually complementary projectors satisfying \eqref{ProjPropInd1} and the subspaces $X_i$, $Y_i$, $i=1,2$, from the corresponding direct decompositions \eqref{rr}, where $X_r=Y_r=\R^2$, as indicated in Appendix \ref{Appendix} (see Remarks \ref{Rem-RegCase}, \ref{Rem-RegPenc}). The projection matrices (which for brevity we will also call projectors) corresponding to the mentioned projectors with respect to the standard bases in $\R^2$ (a basis in $\R^n$ is standard if the $i$th coordinate of the basis vector $e_j$, $j=1,...,n$, is equal to $\delta_{ij}$) have the form:
$$
P_1=\begin{pmatrix} 1 & 0 \\ -1 & 0 \end{pmatrix},\quad
P_2=\begin{pmatrix} 0 & 0 \\ 1 & 1 \end{pmatrix},\quad
Q_1=\begin{pmatrix} 1 & 0 \\ 0 & 0 \end{pmatrix},\quad
Q_2=\begin{pmatrix} 0 & 0 \\ 0 & 1 \end{pmatrix}.
$$
The mentioned subspaces have the form
$$
X_1=\lin\{e_1\},\quad X_2=\lin\{e_2\},\qquad Y_1=\lin\{q_1\},\quad Y_2=\lin\{q_2\},
$$
where $e_1=(1,-1)^\T$, $e_2=(0,1)^\T$, $q_1=(1,0)^\T$ and $q_2=(0,1)^\T$. Hence, $\{e_1,\, e_2\}$ is the basis of $\R^2 = X_1\dot +X_2$ and $\{q_1,\, q_2\}$ is the basis of $\R^2 =Y_1\dot +Y_2$. Note that $D=\R^2$ and $D_i=X_i=P_i\R^2$, $i=1,2$ (see \eqref{Drr}).
The components of $x\in \R^2$ represented as $x= x_{p_1}+x_{p_2}$ (see \eqref{xrr}), where $x_{p_i}\in X_i$, $i=1,2$,  have the form
 $$
x_{p_1}=P_1 x=x_1 e_1,\quad x_{p_2}=P_2 x=(x_1+x_2) e_2,
 $$
where $e_1$, $e_2$ are the basis vectors defined above. If we make the change of
variables
$$x_1=z,\quad x_1+x_2=u,
$$
then $x_{p_1}=z\, e_1$,\; $x_{p_2}=u\, e_2$ and the system \eqref{Ex1DE}, \eqref{Ex1AE} takes the following form:\; ${\dot{z}=z^2}$,\; ${u=\varphi(t)}$.

Notice that $\Es A_1=A$, $\Es B_1=0$, $\Es B_2=B$ and $\Es A_1^{(-1)}=\left(\begin{smallmatrix} 1 & 0 \\ -1 & 0 \end{smallmatrix}\right)$   are the matrices  corresponding to the operators (with respect to the standard bases in $\R^2$) defined by \eqref{ABrrExtend}, \eqref{InvA1}.

The equation $Q_2[f(t,x)-Bx]=0$ defining the manifold $L_0$ (i.e., \eqref{L_treg} where $t_*=0$) is equivalent to \eqref{Ex1AE} or $u=\varphi(t)$ ($u=x_1+x_2$). Consequently, condition \ref{SoglReg} of Theorem \ref{Th_GlobReg-LipschObl} is satisfied for any set $M_1\subseteq X_1$ and $M_2=X_2$.

Choose
 \begin{equation}\label{SetsGlobalSolvEx1}
M_1=\{x_{p_1}=x_1 (1,-1)^\T\mid x_1\in (-\infty,0)\},\;\; M_2=X_2=\{x_{p_2}=(x_1+x_2) (0,1)^\T\mid x_1+x_2\in \R\},
 \end{equation}
where $x_1=z$, $x_1+x_2=u$ if the new variables are used. Hence, $x_{p_1}\in M_1$ iff $x_1\in (-\infty,0)$, and, in general, for any $x=(x_1,x_2)^\T=x_{p_1}+x_{p_2}\in M_1\dot+ M_2$ (i.e., $x_{p_1}\in M_1$ and $x_{p_2}\in M_2$), the components $x_1$, $x_2$ are such that $x_1<0$, $x_2\in \R$.
Let us verify  whether all conditions of Theorem \ref{Th_GlobReg-LipschObl} are satisfied.

The function $f(t,x)$ has the continuous partial derivative with respect to $x$ on $\mT\times D$ (recall that $\mT= [0,\infty)$ and $D=\R^2$) and therefore will use condition \ref{InvReg} of Theorem \ref{Th_GlobReg-Obl}. The operator \eqref{funcPhiInvReg} takes the form  $\Phi_{t_*,x_*}=-B\big|_{X_2}=-B_2\colon X_2\to Y_2$ for any fixed $t_*\in \mT$, $x_*\in D$, and hence the matrix $\Phi_{t_*,x_*}=-1$ corresponds to the operator  $\Phi_{t_*,x_*}$  with respect to the bases $e_2$ and $q_2$ in $X_2$ and $Y_2$, respectively. Since the operator $\Phi_{t_*,x_*}$ has the inverse $\Phi_{t_*,x_*}^{-1}=-B_2^{-1}\in \mathrm{L}(Y_2,X_2)$ for any fixed $t_*\in \mT$, $x_*\in D$, then condition \ref{InvReg} of Theorem \ref{Th_GlobReg-Obl} is fulfilled.

Let us prove that the component $x_{p_1}(t)$ of each solution $x(t)$ of the DAE \eqref{DAE}, \eqref{Ex1DAE} with the initial point $(t_0,x_0)\in L_0$ for which $P_ix_0\in M_i$, $i=1,2$, can never leave $M_1$. Note that any $t_0\in [0,\infty)$, $x_0=(x_{1,0},x_{2,0})^\T$ such that $x_{1,0}<0$ and $x_{2,0}=\varphi(t_0)-x_{1,0}\in \R$ are consistent initial values (i.e., $(t_0,x_0)\in L_0$) for which $P_ix_0\in M_i$, $i=1,2$.
Since the DAE \eqref{DAE}, \eqref{Ex1DAE} is equivalent to the system of the equations \eqref{DAEsysExtReg1}, \eqref{DAEsysExtReg2} which are  equivalent to the equations  \eqref{Ex1DE}, \eqref{Ex1AE}, respectively, and $x_{p_1}\in M_1$ iff $x_1\in (-\infty,0)$, then we need to prove that a solution $x_1=\widehat{x}_1(t)$ of \eqref{Ex1DE} with the initial values $t_0\in [0,\infty)$, $\widehat{x}_1(t_0)=x_{1,0}\in (-\infty,0)$ can never leave the set $\widehat{M}_1:=\{x_1\in (-\infty,0)\}$. Suppose that the solution $\widehat{x}_1(t)$ can leave $\widehat{M}_1$, then it should cross the boundary $\partial \widehat{M}_1=\{x_1=0\}$ of $\widehat{M}_1$. Consequently, there exists $t_1>t_0$ such that $\widehat{x}_1(t_1)=0$. It is obviously that the IVP for the equation \eqref{Ex1DE} with the initial condition $x_1(t_0)=0$ has the unique solution $x_1=\widetilde{x}_1(t)\equiv 0$ on $[t_0,\infty)$. Since the solutions $\widehat{x}_1(t)$ and $\widetilde{x}_1(t)$ coincide at the point $t_1$, then by uniqueness they coincide on $[t_0,t_1]$ and hence
$\widehat{x}_1(t)=0$ on $[t_0,t_1]$, which is impossible, since $\widehat{x}_1(t_0)=x_{1,0}<0$. Thus, the solution $\widehat{x}_1(t)$ cannot leave $\widehat{M}_1$. It follows from the above that condition \ref{RegAttractor} of Theorem \ref{Th_GlobReg-LipschObl} is satisfied.

Define the function $V(t,x_{p_1})\equiv V(x_{p_1}):=0.5x_{p_1}^\T x_{p_1}=x_1^2$. This function is positive for $x_{p_1}\ne 0$ and, obviously, satisfies condition \ref{ExtensReg1} of Theorem \ref{Th_GlobReg-LipschObl}. For the DAE \eqref{DAE} with \eqref{Ex1DAE}, the equation \eqref{DAEsysExtReg1} takes the form
$$
{\dot{x}_{p_1}=\Pi(t,x_{p_1},x_{p_2})},\quad \text{where}\quad   \Pi(t,x_{p_1},x_{p_2})\equiv \Pi(x_{p_1})=x_1^2\,(1,-1)^\T,
$$
and the derivative of the function $V$ along the trajectories of \eqref{DAEsysExtReg1} satisfies the inequality
 \begin{equation}\label{LagrDAERegEx1}
\dot{V}_\eqref{DAEsysExtReg1}(x_{p_1})=2x_{p_1}^\T \Pi(x_{p_1})= 4x_1^3<x_1^2=V(x_{p_1})
 \end{equation}
for each $x_{p_1}\in M_R$ where $M_R=\{x_{p_1}\in M_1\mid \|x_{p_1}\|> R\}=\{x_{p_1}=x_1 (1,-1)^\T\mid x_1< -R\,\}$, $R>0$ is some number. Therefore, the differential inequality \eqref{L1v} takes the form $\dot{v}\le v$. It is easy to verify that this inequality does not have positive solutions with finite escape time. For example, the inequality ${\dot{v}\le v}$ can be written as ${\dot{v}\le k(t)\,U(v)}$, where $k(t)\equiv 1$ and $U(v)=v$, and since $\int\limits_{{\textstyle v}_0}^{\infty}\dfrac{dv}{U(v)} =\infty$ ($v_0>0$), then this inequality does not have positive solutions with finite escape time. Hence,  condition \ref{ExtensReg} of Theorem \ref{Th_GlobReg-LipschObl} holds. Notice that condition \ref{ExtensRegCoroll} of Corollary \ref{Coroll-GlobReg2}, where the functions $k$, $U$ and $V$ are the same as defined above, is also satisfied.

Thus, \emph{by Theorem \ref{Th_GlobReg-LipschObl} (or Theorem \ref{Th_CritGlobReg}), where condition \ref{InvReg-Lipsch} is replaced by condition \ref{InvReg} of Theorem \ref{Th_GlobReg-Obl}, there exists a unique global solution of the IVP \eqref{DAE}, \eqref{Ex1DAE}, \eqref{ini} for each initial point $(t_0,x_0)$, where $x_0=(x_{1,0},x_{2,0})^\T$, such that $t_0\in [0,\infty)$, $x_{1,0}<0$, $x_{2,0}\in \R$ and $x_{2,0}=\varphi(t_0)-x_{1,0}$ (i.e., $(t_0,x_0)$ satisfies \eqref{Ex1AE}).}

\medskip
Now, choose
$$
M_1=\{x_{p_1}=x_1 (1,-1)^\T\mid x_1\in (0,\infty)\}
$$
and the same $M_2=X_2$ as in \eqref{SetsGlobalSolvEx1} (where $x_1=z$, $x_1+x_2=u$ if the new variables are used). Hence, $x_{p_1}\in M_1$ iff $x_1\in (0,\infty)$.

Thus, consistent initial values $t_0$, $x_0=(x_{1,0},x_{2,0})^\T$, for which $P_1x_0\in M_1$, have $x_{1,0}>0$. Let us check whether the conditions of Theorem \ref{Th_RegNeLagr} on the blow-up of solutions are satisfied.

Recall that Theorem \ref{Th_RegNeLagr} contains the same conditions as Theorem~\ref{Th_GlobReg-LipschObl} (or \ref{Th_GlobReg-Obl}), except for condition~\ref{NeUstLagrReg}, but it is necessary to check whether the conditions are satisfied for the new set $M_1$.

It follows directly from the above that condition \ref{SoglReg} of Theorem \ref{Th_GlobReg-LipschObl} and condition \ref{InvReg} of Theorem \ref{Th_GlobReg-Obl} are fulfilled.

We will prove that condition \ref{RegAttractor} of Theorem \ref{Th_GlobReg-LipschObl} holds. Since $x_{p_1}\in M_1$ iff $x_1\in (0,\infty)$, then, by the same arguments as above,  we need to prove that a solution $x_1=x_1(t)$ of \eqref{Ex1DE} with the initial values $t_0\in [0,\infty)$, $x_1(t_0)=x_{1,0}\in (0,\infty)$ can never leave the set $\widehat{M}_1:=\{x_1\in (0,\infty)\}$. Indeed, since $\dot{x}_1=0$ along the boundary $\partial \widehat{M}_1=\{x_1=0\}$ of $\widehat{M}_1$ and $\dot{x}_1>0$ inside $\widehat{M}_1$, then each solution of \eqref{Ex1DE} which at the initial moment $t_0\in [0,\infty)$ is in  $\widehat{M}_1$ can never thereafter leave it.

Further, take the same function $V$ as above, i.e., $V(t,x_{p_1})\equiv V(x_{p_1}):=0.5x_{p_1}^\T x_{p_1}=x_1^2$. Then $\dot{V}_\eqref{DAEsysExtReg1}(x_{p_1})=2x_{p_1}^\T \Pi(x_{p_1})= 4x_1^3$ (see \eqref{LagrDAERegEx1}) and, hence, it satisfies the inequality
 \begin{equation*}
\dot{V}_\eqref{DAEsysExtReg1}(x_{p_1})=4V^{\frac{3}{2}}(x_{p_1})
 \end{equation*}
for each $x_{p_1}\in M_1$ ($x_1>0$). Therefore, the differential inequality \eqref{L2v} takes the form $\dot{v}=4 v^{3/2}$. It is easy to verify that this inequality does not have global positive solutions. For example, it can be written as ${\dot{v}\le k(t)\,U(v)}$, where $k(t)\equiv 4$ and $U(v)=v^{3/2}$, and since ${\int\limits_{k_0}^{\infty}k(t)dt=\infty}$ and $\int\limits_{{\textstyle v}_0}^{\infty}\dfrac{dv}{U(v)}= 2v_0^{-\frac{1}{2}} <\infty$  (\,${k_0,v_0>0}$ are constants), then this inequality does not have global positive solutions. Hence,  condition \ref{NeUstLagrReg} of Theorem \ref{Th_RegNeLagr} holds.  Notice that condition \ref{NeUstLagrRegCoroll} of Corollary \ref{Coroll-RegNeLagr}, where the functions $k$, $U$ and $V$ are the same as defined above, is also fulfilled.

Thus, \emph{by Theorem \ref{Th_RegNeLagr} (or Theorem \ref{Th_CritGlobReg}), for each initial point $(t_0,x_0)$, where $x_0=(x_{1,0},x_{2,0})^\T$, such that $t_0\in [0,\infty)$, $x_{1,0}>0$, $x_{2,0}\in \R$ and $x_{2,0}=\varphi(t_0)-x_{1,0}$,  there exists a unique solution of the IVP \eqref{DAE}, \eqref{Ex1DAE}, \eqref{ini} and this solution has a finite escape time (is blow-up in finite time).}

 \subsection{Example 2 (a regular DAE)}\label{Example2}

Consider the system of differential and algebraic equations
 \begin{align}
\dot{x}_1 &=2x_2,  \label{Ex2DE} \\
x_2 &=x_1x_2-a^2  \label{Ex2AE}
 \end{align}
where $a\ne 0$ ($a\in\R$), $t,t_0\in \mT= [0,\infty)$,  and $x_i=x_i(t)$, $i=1,2$, are real functions.  The system \eqref{Ex2DE}, \eqref{Ex2AE} can be represented as the DAE \eqref{DAE}, i.e., $\frac{d}{dt}[Ax]+Bx=f(t,x)$, where
\begin{equation}\label{Ex2DAE}
x=\begin{pmatrix} x_1 \\  x_2 \end{pmatrix},\quad
A=\begin{pmatrix} 1 & 0 \\ 0 & 0 \end{pmatrix},\quad
B=\begin{pmatrix} 0 & 0 \\ 0 & 1 \end{pmatrix},\quad
f(t,x)\equiv f(x)=\begin{pmatrix} 2x_2 \\ x_1x_2-a^2  \end{pmatrix},
\end{equation} 
$f\in C^\infty(D,\R^2)$, $D=\R^2$. It is easy to verify that the characteristic pencil $\lambda A+B$ is a regular pencil of index 1. For the DAE \eqref{DAE}, \eqref{Ex2DAE}, we use the initial condition \eqref{ini}, i.e., $x(t_0)=x_0$, where $x_0=(x_{1,0},x_{2,0})^\T$. Obviously, $x_{1,0}$ and $x_{2,0}$ must satisfy \eqref{Ex2AE}.

Construct the projectors \eqref{ProjR} (where $n=m=2$) and the subspaces from the direct decompositions \eqref{rr} (where $X_r=Y_r=\R^2$). With respect to the standard bases in $\R^2$, the projection matrices
 \begin{equation}\label{Ex2Project}
P_1=\begin{pmatrix} 1 & 0 \\ 0 & 0 \end{pmatrix},\quad
P_2=\begin{pmatrix} 0 & 0 \\ 0 & 1 \end{pmatrix},\quad
Q_1=\begin{pmatrix} 1 & 0 \\ 0 & 0 \end{pmatrix},\quad
Q_2=\begin{pmatrix} 0 & 0 \\ 0 & 1 \end{pmatrix}
 \end{equation}
correspond to the projectors $P_i\colon \R^2\to X_i$, $Q_i\colon \R^2\to Y_i$ ($i=1,2$) from \eqref{ProjR}. Therefore, the matrices $\Es A_1=A$, $\Es B_1=0$, $\Es B_2=B$ and $A_1^{(-1)}=\left(\begin{smallmatrix} 1 & 0 \\ 0 & 0 \end{smallmatrix}\right)$ correspond to the operators, defined by \eqref{ABrrExtend} and \eqref{InvA1}, with respect to the standard bases in $\R^2$. The subspaces from \eqref{rr} have the form\;
$X_1=\lin\{e_1\}$,\; $X_2=\lin\{e_2\}$,\; $Y_1=\lin\{q_1\}$,\; $Y_2=\lin\{q_2\}$,\;
where  $e_1=(1,0)^\T$, $e_2=(0,1)^\T$, $q_1=(1,0)^\T$ and $q_2=(0,1)^\T$. Thus, $\{e_1,\, e_2\}$ is the basis of $\R^2 = X_1\dot +X_2$ and $\{q_1,\, q_2\}$ is the basis of $\R^2 =Y_1\dot +Y_2$; obviously, the bases coincide with the standard basis of $\R^2$. With respect to the direct decomposition $\R^2 = X_1\dot +X_2$, a vector $x\in \R^2$ can be uniquely represented in the form \eqref{xrr} where
 \begin{equation}\label{Ex2-xrr}
x_{p_1}=P_1 x=x_1(1,0)^\T,\quad x_{p_2}=P_2 x=x_2(0,1)^\T.
 \end{equation}
Note that $x_{p_1}=x_1$ with respect to the chosen basis $e_1$ in $X_1$ and $x_{p_2}=x_2$ with respect to the chosen basis $e_2$ in $X_2$.
The sets $D_i$ ($i=1,2$) from the decomposition \eqref{Drr}, where $D=\R^2$, have the form
 \begin{equation}\label{D_1D_2}
D_1=X_1=P_1\R^2=\{x_{p_1}=(x_1,0)^\T\mid x_1\in \R\},\;\; D_2=X_2=P_2\R^2=\{x_{p_2}=(0,x_2)^\T\mid x_2\in \R\}.
 \end{equation}

 \smallskip
Further, using the theorems from Section \ref{SectGlobReg},  we will prove that the DAE \eqref{DAE}, \eqref{Ex2DAE}  with the initial condition \eqref{ini} has a unique global solution for any initial values $t_0\in \mT$, $x_0=(x_{1,0},x_{2,0})^\T\in \R^2$ such that \eqref{Ex2AE} is satisfied and $P_ix_0\in M_i$, $i=1,2$, where $M_i$ will be specified below.

The consistency condition $(t,x)\in L_0$ holds if $(t,x)$ satisfies the equation \eqref{Ex2AE}, which is equivalent to the algebraic equation  $Q_2[f(t,x)-Bx]=0$ defining the manifold $L_0$ and can be rewritten as
\begin{equation}\label{Ex2AEequiv}
x_2 =a^2(x_1-1)^{-1}.
\end{equation}
Thus, for any fixed ${t\in \mT}$, $x_1\in \R\setminus \{1\}$ there exists a unique $x_2\in \R\setminus \{0\}$ such that the equation \eqref{Ex2AEequiv} is satisfied.
Choose
 \begin{equation}\label{SetsGlobalSolvEx2}
M_1=\{x_{p_1}=(x_1,0)^\T\in X_1\mid x_1\in \R\setminus \{1\}\},\quad
M_2=\{x_{p_2}=(0,x_2)^\T\in X_2\mid x_2\in \R\setminus \{0\}\}.
 \end{equation}
Then condition \ref{SoglReg} of Theorem \ref{Th_GlobReg-LipschObl} holds.
Note that $x_{p_1}\in M_1$ iff $x_1\in \R\setminus \{1\}$ and $x_{p_2}\in M_2$ iff $x_2\in \R\setminus \{0\}$.

The matrix corresponding to the operator \eqref{funcPhiInvReg}, where $t_*$, $x_*=(x_1^*,x_2^*)^\T$ are fixed,  has the form $\Phi_{t_*,x_*}=x_1^*-1$ with respect to the basis $e_2$ of $X_2$ and the basis $q_2$ of $Y_2$. Hence, the operator $\Phi_{t_*,x_*}$ has the inverse $\Phi_{t_*,x_*}^{-1}\in \mathrm{L}(Y_2,X_2)$ for any fixed $t_*\in \mT$, $x_{p_1}^*=(x_1^*,0)^\T\in M_1$ (i.e., $x_1^*\ne 1$) and $x_{p_2}^*=(0,x_2^*)^\T\in M_2$. Consequently, condition \ref{InvReg} of Theorem \ref{Th_GlobReg-Obl} holds.

The function $\Pi$ defined in \eqref{Pi} has the form\; $\Pi(t,x_{p_1},x_{p_2})\equiv \Pi(x_{p_1},x_{p_2})= (2x_2,0)^\T$,\;
and the equation \eqref{DAEsysExtReg1}, which is equivalent to \eqref{Ex2DE}, takes the form $\dot{x}_{p_1}=\Pi(x_{p_1},x_{p_2})$. Therefore, the sets $D_{\Pi,i}$, $i=1,2$, specified in condition \ref{RegAttractCoroll} of Corollary \ref{Coroll-GlobReg1}, have the form ${D_{\Pi,i}=X_i}$. It is clear that $D_{c,i}=M_i$, $i=1,2$, and hence ${\widetilde{D_i}= M_i}$, $i=1,2$, where $D_{c,i}$ and $\widetilde{D_i}$ are specified in the same condition. Define the function
$$
W(t,x_{p_1})\equiv W(x_{p_1}):=-(x_1-1,0)\, (x_1-1,0)^\T=-(x_1-1)^2
$$
and a family of closed sets
$$
K_r=\{x_{p_1}=(x_1,0)^\T\in M_1\mid x_1\in (-\infty,1-r]\cup [1+r,+\infty)\,\},\quad {0<r<<1},
$$ 
i.e., $r>0$ is sufficiently small. Further, we will verify that condition \ref{RegAttractCoroll} of Corollary \ref{Coroll-GlobReg1} is satisfied.  Since $M_1^c\cap \widetilde{D_1}=M_1^c\cap M_1=\emptyset$, then there is no need to check that the inequality $W(x_{p_1}^1)<W(x_{p_1}^2)$ holds for ${x_{p_1}^1\in K_r}$ and $x_{p_1}^2\in M_1^c\cap \widetilde{D_1}$, but generally it holds for ${x_{p_1}^1\in K_r}$ and $x_{p_1}^2\in M_1^c=\{x_{p_1}=(x_1,0)^\T\mid x_1=1\}$.
Since $\partial_{x_{p_1}} W(x_{p_1})=-2 (x_1-1,0)$ then  $\dot{W}_\eqref{DAEsysExtReg1}(x_{p_1})=-4(x_1-1)x_2$. Recall that the consistency condition $(t,x)\in L_0$ is satisfied if $(t,x)$ satisfies the equation \eqref{Ex2AE} or \eqref{Ex2AEequiv}. Hence, for each $t\in \mT$, $x_{p_1}\in K_r^c\cap \widetilde{D_1}=\{x_{p_1}=(x_1,0)^\T\in X_1\mid x_1\in (1-r,1)\cup (1,1+r)\,\}$, $x_{p_2}\in \widetilde{D_2}=M_2$ such that $t$, $x_1$, $x_2$ satisfy \eqref{Ex2AEequiv} (i.e., $(t,x_{p_1}+x_{p_2})\in L_0$), the inequality \eqref{RegAttractIneq}, which takes the form $\dot{W}_\eqref{DAEsysExtReg1}(x_{p_1})=-4(x_1-1)a^2(x_1-1)^{-1}=-4a^2\le 0$, is fulfilled. Consequently, condition \ref{RegAttractCoroll} of Corollary~\ref{Coroll-GlobReg1} holds.

Define the function ${V(t,x_{p_1})\equiv V(x_{p_1}):=x_{p_1}^\T x_{p_1}=x_1^2}$. Obviously, the function $V$ satisfies condition \ref{ExtensReg1} of Theorem \ref{Th_GlobReg-LipschObl}, and $\dot{V}_\eqref{DAEsysExtReg1}(x_{p_1})=4x_1x_2$   (as mentioned above, \eqref{DAEsysExtReg1} has the form $\dot{x}_{p_1}=\Pi(x_{p_1},x_{p_2})$ where $\Pi(x_{p_1},x_{p_2})= (2x_2,0)^\T$).  For each $t\in \mT$, $x_{p_1}\in M_R=\{x_{p_1}=(x_1,0)^\T\in M_1\mid |x_1|>R\,\}$, where ${R>>1}$ (i.e., $R>0$ is sufficiently large), $x_{p_2}\in M_2$ such that $t$, $x_1$, $x_2$ satisfy \eqref{Ex2AEequiv}, the inequality $\dot{V}_\eqref{DAEsysExtReg1}(x_{p_1})=4x_1a^2(x_1-1)^{-1}\le 4a^2$ holds. Therefore, the inequality \eqref{L1v} takes the form $\dot{v}\le 4a^2$, and it is easy to verify that this inequality does not have positive solutions with finite escape time. Hence, condition \ref{ExtensReg} of Theorem \ref{Th_GlobReg-LipschObl} is fulfilled.

 \smallskip
Thus, \emph{by Theorem \ref{Th_GlobReg-LipschObl} (or Theorem \ref{Th_CritGlobReg}), where conditions \ref{InvReg-Lipsch} and \ref{RegAttractor} are replaced by conditions \ref{InvReg} of Theorem \ref{Th_GlobReg-Obl} and  \ref{RegAttractCoroll} of Corollary \ref{Coroll-GlobReg1} respectively, there exists a unique global solution of the IVP \eqref{DAE}, \eqref{Ex2DAE}, \eqref{ini} for each initial point $(t_0,x_0)$, where $x_0=(x_{1,0},x_{2,0})^\T$, such that $t_0\in [0,\infty)$, $x_{1,0}\in \R\setminus \{1\}$, $x_{2,0}\in \R\setminus \{0\}$} (\emph{i.e., $P_ix_0\in M_i$, $i=1,2$}),  \emph{and $x_{2,0}=x_{1,0}x_{2,0}-a^2$} (\emph{i.e., $(t_0,x_0)\in L_0$}).

 \smallskip
Indeed, \emph{a solution of the DAE \eqref{DAE}, \eqref{Ex2DAE}, which satisfies the initial condition \eqref{ini} where $x_0=(x_{1,0},x_{2,0})^\T$, has the form $x(t)=(x_1(t),x_2(t))^\T$ where $x_1(t)=\sgn (x_{1,0}-1)\, \big(4a^2(t-t_0)+(x_{1,0}-1)^2\big)^{1/2}+1$ and
$x_2(t)=\sgn (x_{1,0}-1)a^2\big(4a^2(t-t_0)+(x_{1,0}-1)^2\big)^{-1/2}$, and therefore it exists and is global for $t_0\ge 0$, $x_{1,0}\ne 1$ and $x_{2,0}=a^2(x_{1,0}-1)^{-1}\ne 0$}.

 \subsection{Example 3 (a regular DAE)}\label{Example3}

In this section, we consider the system of differential and algebraic equations
 \begin{align}
\dot{x}_1+b\, x_1 &=0,  \label{Ex3DE}  \\
x_1^2+x_2^2 &=1,    \label{Ex3AE}
 \end{align}
where $t\in \mT\!= [0,\infty)$, $b\in\R$, and $x_i=x_i(t)$, $i=1,2$, are real functions. It can be written as
 \begin{align}
\dot{x}_1 + b\, x_1 &=0,   \\
x_2 &=x_1^2+x_2^2+x_2-1
 \end{align}
and be represented in the form of the DAE $\frac{d}{dt}[Ax]+Bx=f(t,x)$, i.e., \eqref{DAE}, where
\begin{equation}\label{Ex3DAE}
x=\begin{pmatrix} x_1 \\  x_2 \end{pmatrix},\quad
A=\begin{pmatrix} 1 & 0 \\ 0 & 0 \end{pmatrix},\quad
B=\begin{pmatrix} b & 0 \\ 0 & 1 \end{pmatrix},\quad
f(t,x)=\begin{pmatrix} 0 \\ x_1^2+x_2^2+x_2-1 \end{pmatrix},
\end{equation}
$f\in C^\infty(\mT\times D,\R^2)$, $D=\R^2$, and $A$,~$B$ are such that $\lambda A+B$ is a regular pencil of index 1. Consider the initial condition $x_1(t_0)=x_{1,0}$, $x_2(t_0)=x_{2,0}$, where $t_0\in \mT$, which for the DAE takes the form \eqref{ini}, where $x_0=(x_{1,0},x_{2,0})^\T$. Obviously, $x_{1,0}$, $x_{2,0}$ must satisfy \eqref{Ex3AE}.

As for the examples discussed above, we first construct the projectors \eqref{ProjR} (where $n=m=2$) and the subspaces from the direct decompositions \eqref{rr} (where $X_r=Y_r=\R^2$). With respect to the standard bases in $\R^2$, the projection matrices of the form \eqref{Ex2Project}  correspond to the projectors $P_i\colon \R^2\to X_i$, $Q_i\colon \R^2\to Y_i$ ($i=1,2$) from \eqref{ProjR}. 
Therefore, the matrices $\Es A_1=\Es A_1^{(-1)}=\left(\begin{smallmatrix} 1 & 0 \\ 0 & 0 \end{smallmatrix}\right)$, $\Es B_1=\left(\begin{smallmatrix} b & 0 \\ 0 & 0 \end{smallmatrix}\right)$ and $\Es B_2=\left(\begin{smallmatrix} 0 & 0 \\ 0 & 1 \end{smallmatrix}\right)$ correspond to the operators, defined by \eqref{ABrrExtend} and \eqref{InvA1}, with respect to the standard bases in $\R^2$. The subspaces from \eqref{rr} have the form $X_1=Y_1=\lin\{(1,0)^\T\}$, $X_2=Y_2=\lin\{(0,1)^\T\}$. Denote by $\{e_1=(1,0)^\T,\, e_2=(0,1)^\T\}$ the basis of $\R^2 = X_1\dot +X_2$ and by $\{q_1=(1,0)^\T,\, q_2=(0,1)^\T\}$ the basis of $\R^2 =Y_1\dot +Y_2$. With respect to the direct decomposition $\R^2 = X_1\dot +X_2$, a vector $x\in \R^2$ can be uniquely represented in the form \eqref{xrr} where $x_{p_1}$, $x_{p_2}$ have the form \eqref{Ex2-xrr}, i.e.,
 $$
x_{p_1}=P_1 x=x_1(1,0)^\T,\quad x_{p_2}=P_2 x=x_2(0,1)^\T.
 $$
The subspaces $D_i$, $i=1,2$, from the decomposition \eqref{Drr}, where $D=\R^2$, have the form  \eqref{D_1D_2}.

Introduce the sets
 \begin{equation}\label{SetsGlobalSolvEx3}
M_1=\{x_{p_1}=(x_1,0)^\T\in X_1\mid x_1\in (-1,1)\},\;\;
M_2=\{x_{p_2}=(0,x_2)^\T\in X_2\mid x_2\in (0,1]\}.
 \end{equation}
Then $x_{p_1}\in M_1$ iff $x_1\in (-1,1)$ and $x_{p_2}\in M_2$ iff $x_2\in (0,1]$.
In addition, introduce the set
\begin{equation}\label{Set2GlobalSolvEx3}
\widehat{M_2}=\{x_{p_2}=(0,x_2)^\T\in X_2\mid x_2\in [-1,0)\}.
\end{equation}

The equation $Q_2[f(t,x)-Bx]=0$ defining $L_0$ is equivalent to  \eqref{Ex3AE}. Hence, the consistent initial values $t_0$, $x_0=(x_{1,0},x_{2,0})^\T$ (i.e., $(t_0,x_0)\in L_0$), for which $P_1x_0\in M_1$ and $P_2x_0\in M_2$ (respectively, $P_2x_0\in \widehat{M_2}$), are such that $x_{1,0}\in (-1,1)$, $x_{2,0}\in (0,1]$ (respectively, $x_{2,0}\in [-1,0)$\,) and $x_{1,0}^2+x_{2,0}^2 =1$.

 \smallskip
Further, we will verify  whether the conditions for the global solvability of the IVP \eqref{DAE}, \eqref{Ex3DAE}, \eqref{ini}, which are given in Section \ref{SectGlobReg}, are satisfied. 

Condition \ref{SoglReg} of Theorem \ref{Th_GlobReg-LipschObl} holds since for any fixed  ${t\in \mT}$, ${x_{p_1}\in M_1}$ (i.e., $x_1\in (-1,1)$\,) there exists a unique ${x_{p_2}\in M_2}$ (i.e., $x_2\in (0,1]$\,) such that ${(t,x_{p_1}+x_{p_2})\in L_0}$ (i.e., the equation \eqref{Ex3AE} is satisfied). Note that the same condition will hold if one replace the set $M_2$ by $\widehat{M_2}$.

Condition \ref{InvReg} of Theorem \ref{Th_GlobReg-Obl} is fulfilled since for any fixed $t_*\in \mT$ and any fixed $x_*=(x_1^*,x_2^*)^\T$ such that $x_1^*\in (-1,1)$ and  $x_2^*\in (0,1]$ there exists $\Phi_{t_*,x_*}^{-1}=(2x_2^*)^{-1}$, where $\Phi_{t_*,x_*}=2x_2^*$ is the matrix corresponding to the operator \eqref{funcPhiInvReg} with respect to the bases $e_2$ and $q_2$ in $X_2$ and $Y_2$, respectively. Note that the same condition holds if $x_2^*$ belongs to $[-1,0)$ instead of $(0,1]$.

Let us verify whether condition \ref{RegAttractCoroll} of Corollary \ref{Coroll-GlobReg1} is fulfilled (then condition \ref{RegAttractor} of Theorem \ref{Th_GlobReg-LipschObl} holds).
The function defined in \eqref{Pi} has the form\;
${\Pi(t,x_{p_1},x_{p_2})\equiv \Pi(x_{p_1})=(-b\,x_1,0)^\T=-b\, x_{p_1}}$\; 
and the equation \eqref{DAEsysExtReg1} takes the form ${\dot{x}_{p_1}=\Pi(x_{p_1})}$.  Since $\Pi$ does not depend on $x_{p_2}$, then ${\widetilde{D_1}= D_{\Pi,1}=X_1}$, where $\widetilde{D_1}$ and $D_{\Pi,1}$ are the sets specified in condition \ref{RegAttractCoroll} of Corollary \ref{Coroll-GlobReg1}.
Define the function \;${W(t,x_{p_1})\equiv W(x_{p_1}):=x_{p_1}^\T x_{p_1}=x_1^2}$\; and a family of closed sets
\;${K_r=\{x_{p_1}=(x_1,0)^\T\in M_1\mid x_1\in [-1+r,1-r]\,\}}$,\;  ${0<r<<1}$. 
Then  ${W(x_{p_1}^1)<W(x_{p_1}^2)}$ for every ${x_{p_1}^1\in K_r}$ and  ${x_{p_1}^2\in M_1^c\cap \widetilde{D_1}=M_1^c=\{x_{p_1}=(x_1,0)^\T\in X_1\mid x_1\in (-\infty,-1]\cup[1,+\infty)\,\}}$. In addition, if ${b\ge 0}$, then ${\dot{W}_\eqref{DAEsysExtReg1}(x_{p_1})=-2 b x_1^2\le 0}$ for every $x_{p_1}\in K_r^c\cap \widetilde{D_1}=K_r^c=\{x_{p_1}=(x_1,0)^\T\in X_1\mid x_1\in (-\infty,-1+r)\cup (1-r,+\infty)\,\}$. Hence, condition \ref{RegAttractCoroll} of Corollary \ref{Coroll-GlobReg1} holds (both for ${x_{p_2}\in M_2}$ and for ${x_{p_2}\in \widehat{M_2}}$\,)\, if ${b\ge 0}$.

Note that since the set $M_1$ is bounded, then condition \ref{ExtensReg} of \ref{Th_GlobReg-LipschObl} does not need to be checked.

 \smallskip
Thus, \emph{if $b\ge 0$, then by Theorem \ref{Th_GlobReg-LipschObl} (or Theorem \ref{Th_CritGlobReg}), where conditions \ref{InvReg-Lipsch} and \ref{RegAttractor} are replaced by conditions \ref{InvReg} of Theorem \ref{Th_GlobReg-Obl} and  \ref{RegAttractCoroll} of Corollary \ref{Coroll-GlobReg1} respectively, there exists a unique global solution of the IVP \eqref{DAE}, \eqref{Ex3DAE}, \eqref{ini} for each initial point $(t_0,x_0)$, where $x_0=(x_{1,0},x_{2,0})^\T$, for which $t_0\in [0,\infty)$, $x_{1,0}\in (-1,1)$ and $x_{2,0}\in (0,1]$} (i.e., $P_1x_0\in M_1$ and $P_2x_0\in M_2$ where $M_1$, $M_2$ are defined in \eqref{SetsGlobalSolvEx3}), \emph{and in addition, $x_{1,0}^2+x_{2,0}^2 =1$} (i.e., $(t_0,x_0)\in L_0$). \emph{The same statement is valid for $x_{2,0}$ belonging to $[-1,0)$} (i.e., for $P_2x_0\in \widehat{M_2}$ where $\widehat{M_2}$ is defined in \eqref{Set2GlobalSolvEx3}) \emph{instead of $(0,1]$}.

In addition, since the sets $M_1$, $M_2$ are bounded, then \emph{by Corollary \ref{Coroll-UstLagrReg} the DAE \eqref{DAE}, \eqref{Ex3DAE} is Lagrange stable for the initial points specified above} (see the definition of the Lagrange stability of the DAE in Section \ref{Preliminaries}).

 \smallskip
Indeed, \emph{a solution of the IVP \eqref{DAE}, \eqref{Ex3DAE}, \eqref{ini} has the form $x(t)=(x_1(t),x_2(t))^\T$ where $x_1(t)=x_{1,0}\, \ee^{-b(t-t_0)}$,\;   $x_2(t)=\sqrt{1-x_{1,0}^2\, \ee^{-2b(t-t_0)}}$ for the initial value $x_0=(x_{1,0},x_{2,0})^\T$ such that $x_{1,0}\in (-1,1)$ and $x_{2,0}\in (0,1]$, and $x_2(t)=-\sqrt{1-x_{1,0}^2\, \ee^{-2b(t-t_0)}}$ for the initial value $x_0$ such that $x_{1,0}\in (-1,1)$ and $x_{2,0}\in [-1,0)$. Obviously, the solution exists and is global if $b\ge 0$ and, in addition, it is bounded for all $t\ge t_0$, namely, $x_1(t)\in (-1,1)$ and $x_2(t)\in (0,1]$ or $x_2(t)\in [-1,0)$ (depending on the initial values) for all $t\ge t_0$. This coincides with the conclusions obtained using the theorems}.

 \subsection{Example 4 (a singular (nonregular) DAE)}\label{Example4}

Consider the system of differential and algebraic equations
 \begin{align}
\frac{d}{dt}(x_1-x_3)+x_1-x_2-x_3 &= f_1(t,x),  \label{genSing1}\\
x_1+x_2-x_3 &= f_2(t,x),  \label{genSing2}\\
2x_2 &=f_3(t,x). \label{genSing3}
 \end{align}
where the functions $f_i\in C(\mT\times D,\R)$, $i=1,2,3$, have the continuous partial derivatives $\partial_x f_i$ on $\mT\times D$, $\mT= [0,\infty)$, $D=\R^3$, and $x_i=x_i(t)$, $i=1,2,3$, are real functions.  The system \eqref{genSing1}--\eqref{genSing3} can be represented in the form of the DAE $\frac{d}{dt}[Ax]+Bx=f(t,x)$, i.e., \eqref{DAE}, where
 \begin{equation}\label{pencilGener}
x=\begin{pmatrix} x_1 \\ x_2 \\ x_3 \end{pmatrix},\;
A=\begin{pmatrix} 1 & 0 & -1 \\
                  0 & 0 & 0 \\
                  0 & 0 & 0 \end{pmatrix},\;\;
B=\begin{pmatrix} 1 & -1 & -1 \\
                  1 & 1  & -1 \\
                  0 & 2 & 0 \end{pmatrix},\;\;
f(t,x)=\begin{pmatrix} f_1(t,x) \\ f_2(t,x) \\ f_3(t,x) \end{pmatrix}.
 \end{equation} 
It can be readily verified that the characteristic pencil $\lambda A+B$ is singular and ${\rank(\lambda A+B)=2}$. The total defect of the pencil equals $d(\lambda A+B)=2$ and the defects of $\lambda A+B$ and $\lambda A^\T+B^\T$ equal $1$ (see definition \ref{DefectPencil} in Appendix \ref{Appendix}).
The initial condition is given as $x(t_0)=x_0$, i.e., \eqref{ini}. Obviously, the initial values $t_0$, $x_0=(x_{1,0},x_{2,0},x_{3,0})^\T$ must satisfy the equations \eqref{genSing2}, \eqref{genSing3}.

The global solvability of the DAE \eqref{DAE}, \eqref{pencilGener} (accordingly, the system \eqref{genSing1}--\eqref{genSing3}) has been studied in \cite[Section~10]{Fil.Sing-GN} by using the theorems \cite[Theorems 3.2 and 7.1]{Fil.Sing-GN}. The same conditions as when applying the theorem \cite[Theorem 3.2]{Fil.Sing-GN} will be obtained if we apply Theorem \ref{Th_GlobSing-LipschObl}, where condition \ref{InvSing-Lipsch} is replaced by condition \ref{InvSing} of Theorem \ref{Th_GlobSing-Obl} and $D=\R^3$, and take $M_{s1}=X_{s_1}\dot+X_1$ and $M_2=X_2$.

 \smallskip
In this section, we will consider a certain particular case of the system \eqref{genSing1}--\eqref{genSing3} in which a solution is blow-up.

The projection matrices corresponding to the projectors (see \eqref{ProjRS}, \eqref{ProjS} and \eqref{ProjR} in Appendix \ref{Appendix}) onto the  ``singular'' subspaces, i.e., $S_i\colon \R^3\to X_{s_i}$, $F_i\colon\R^3\to Y_{s_i}$, $i=1,2$, $S=S_1+S_2\colon\R^3\to X_s$ and $F=F_1+F_2\colon \R^3\to Y_s$, and onto the ``regular'' subspaces, i.e., $P_i\colon \R^3\to X_i$, $Q_i\colon\R^3\to Y_i$, $i=1,2$, $P=P_1+P_2\colon \R^3\to X_r$ and $Q=Q_1+Q_2\colon \R^3\to Y_r$,   are presented in \cite[Section~10]{Fil.Sing-GN}. Also, in \cite[Section~10]{Fil.Sing-GN}, the matrices
correspond to the operators $\EuScript A_r$, $\EuScript B_r$, $\EuScript A_{gen}$, $\EuScript B_{gen}$, $\EuScript B_{und}$, $\EuScript B_{ov}$, $\EuScript A_{gen}^{(-1)}$ defined by \eqref{AsrBsrExtend}, \eqref{ABssExtend}, \eqref{InvAgen} are given.
The subspaces from the decompositions \eqref{Xssrr} and \eqref{Yssrr}, where $n=3$ and $m=3$, have the following form (see \cite[Section~10]{Fil.Sing-GN}):
\begin{gather*}
X_s = X_{s_1}\dot +X_{s_2}=\lin\{s_i\}_{i=1}^2,\; {X_{s_1}=\lin\{s_1\}},\; X_{s_2}=\lin\{s_2\},\; X_r=X_2= \lin\{p\},\; X_1=\{0\},  \\ 
Y_s = Y_{s_1}\dot +  Y_{s_2}= \lin\{l_i\}_{i=1}^2,\; {Y_{s_1} = \lin\{l_1\}},\; Y_{s_2} = \lin\{l_2\},\; Y_r=Y_2= \lin\{q\},\; Y_1=\{0\}, 
\end{gather*}
 \begin{equation}\label{BasisVectors}
s_1 =\begin{pmatrix} 1 \\ 0 \\ 0 \end{pmatrix},\,
s_2 =\begin{pmatrix} 1 \\ 0 \\ 1 \end{pmatrix},\,
p =\begin{pmatrix} 0 \\ 1 \\ 0 \end{pmatrix},\;
l_1 =\begin{pmatrix} 1 \\ 0 \\ 0 \end{pmatrix},\,
l_2 =\begin{pmatrix} 0 \\ 1 \\ 0 \end{pmatrix},\;
q =\begin{pmatrix} -1/2 \\ 1/2 \\ 1 \end{pmatrix}.
 \end{equation}
Since $D=\R^3$, then $D_{s_i}= X_{s_i}$ and $D_i= X_i$, $i=1,2$, in the decomposition \eqref{Dssrr}.

The components of $x\in \R^3$ represented as $x=x_{s_1}+x_{s_2}+x_{p_1}+x_{p_2}$ (see \eqref{xsr}) have the form
\begin{equation*}
x_{s_1} = S_1 x=(x_1-x_3)(1,0,0)^\T,\; x_{s_2}=S_2x=x_3(1,0,1)^\T,\;
x_{p_1}=P_1 x=0,\; x_{p_2}=P_2 x=x_2(0,1,0)^\T.
\end{equation*}
(notice that $x_1-x_3$, $x_3$, $x_2$ are the coordinates of the vector $x=(x_1,x_2,x_3)^\T$ with respect to the basis $s_1$, $s_2$, $p$ in $\R^3$, i.e., $x=(x_1-x_3)s_1+x_3 s_2+ x_2 p$, where $s_1$, $s_2$, $p$ are defined in \eqref{BasisVectors}).
If we make the change of variables
\begin{equation}\label{NewVariable}
x_1-x_3=w,\quad x_3=\xi,\quad x_2=u,
\end{equation}
then
 $$
x_{s_1}=w\, (1,0,0)^\T,\quad x_{s_2}=\xi\, (1,0,1)^\T,\quad x_{p_1}=0,\quad x_{p_2}=u\, (0,1,0)^\T,
 $$
and the system \eqref{genSing1}--\eqref{genSing3} can be represented in the form (the same form as in \cite[(10.13)--(10.15)]{Fil.Sing-GN})
 \begin{align}
\dot{w} &=-w+ \widetilde{f}_1(t,w,\xi,u)+ 0.5\widetilde{f}_3(t,w,\xi,u), \label{genSing1New}\\
u &=0.5\widetilde{f}_3(t,w,\xi,u),  \label{genSing3New}\\
w &= \widetilde{f}_2(t,w,\xi,u)-0.5\widetilde{f}_3(t,w,\xi,u), \label{genSing2New}
 \end{align}
where
 \begin{equation}\label{tilde_f}
\widetilde{f}(t,w,\xi,u):=f(t,w+\xi,u,\xi)=f(t,x_1,x_2,x_3)=f(t,x).
 \end{equation}
Note that the system \eqref{genSing1New}--\eqref{genSing2New} is equivalent to the system \eqref{DAEsysExtDE1}--\eqref{DAEsysExtAE2} where the equation \eqref{DAEsysExtDE2} disappears since $Q_1=0$ and $P_1=0$.
The equations $Q_2[f(t,x)-Bx]=0$, $F_2[Bx-f(t,x)]=0$ (as well as  \eqref{DAEsysExtAE1}, \eqref{DAEsysExtAE2}) defining the manifold $L_0$ (i.e., \eqref{L_tSing} where $t_*=0$) are equivalent to \eqref{genSing2}, \eqref{genSing3} (or \eqref{genSing2New}, \eqref{genSing3New}).

 \smallskip
\emph{Consider the particular case of the system \eqref{genSing1}--\eqref{genSing3} when}
 \begin{equation}\label{genSing-f}
\begin{split}
& f_1(t,x)=(x_1-x_3-1)^3+x_1-x_3+x_2,\\
& f_2(t,x)=x_1-x_3+x_2,\\
& f_3(t,x)=-(x_2^3+3x_2^2+x_2+1)-(t+1)x_3^2.
\end{split}
 \end{equation}
For the new variables \eqref{NewVariable},  $\widetilde{f}_i(t,w,\xi,u):=f_i(t,w+\xi,u,\xi)$, $i=1,2,3$, (see \eqref{tilde_f})  take the form
\begin{equation}\label{genSing-tilde_f}
\widetilde{f}_1(t,w,\xi,u)\!=\!(w-1)^3+w+u,\;
\widetilde{f}_2(t,w,\xi,u)\!=\!w+u,\;
\widetilde{f}_3(t,w,\xi,u)\!=\!-(u^3+3u^2+u+1)-(t+1)\xi^2.
\end{equation}
Then the equations \eqref{genSing1New} and \eqref{genSing3New} take the form
\begin{align}
\dot{w} &=(w-1)^3, \label{genSing1NewCase}\\
u &=-0.5\big[(u^3+3u^2+u+1)-(t+1)\xi^2\big],  \label{genSing3NewCase}
 \end{align}
and the equation \eqref{genSing2New} becomes the identity ($w=w$) when we substitute \eqref{genSing3New} into it.

From \eqref{genSing3NewCase} we obtain $u=-(t+1)^{1/3}\xi^{2/3}-1$. Therefore, for any fixed  $t\in \mT$, $w, \xi\in\R$ there exists a unique $u\in\R$ such that \eqref{genSing3NewCase} is fulfilled and, hence, \eqref{genSing2New}, \eqref{genSing3New} with $\widetilde{f}_i$, $i=1,2,3$, of the form \eqref{genSing-tilde_f} are satisfied. Consequently, condition \ref{SoglSing} of Theorem \ref{Th_GlobSing-LipschObl} holds for any sets $M_{s1}\subseteq X_{s_1}$, $M_{s_2}\subseteq X_{s_2}$ and $M_2\subseteq X_2$. Note that $M_{s1}\subseteq X_{s_1}$ since $X_1=\{0\}$.

For $f_i$ (or $\widetilde{f}_i$), $i=1,2,3$, of the form \eqref{genSing-f}, the matrix corresponding to the operator \eqref{funcPhiInvSing}, where $t_*$ and $x_*=(x_1^*,x_2^*,x_3^*)^\T=(w_*+\xi_*,u_*,\xi_*)^\T$ are fixed,  has the form $\Phi_{t_*,x_*}=\partial_{x_2}f_3(t_*,x_*)-2=-3(x_2^*+1)^2$ with respect to the basis $p$ of $X_2$ and the basis $q$ of $Y_2$. Hence, the operator $\Phi_{t_*,x_*}$ has the inverse $\Phi_{t_*,x_*}^{-1}\in \mathrm{L}(Y_2,X_2)$ if $x_2^*\ne -1$ or $u_*\ne -1$ (note that $\partial_{x_2}f_3(t,x)=\partial_u \widetilde{f}_3(t,w,\xi,u)$). Consequently, condition \ref{InvSing} of Theorem \ref{Th_GlobSing-Obl} holds if $x_2^*=u_*\ne -1$.

Choose
 \begin{equation}\label{SetsGlobalSolvEx4}
\begin{split}
& M_{s1}=\{x_{s_1}=(x_1-x_3)(1,0,0)^\T\mid x_1-x_3\in (1,+\infty)\},\\
& M_{s_2}=\{x_{s_2}=x_3(1,0,1)^\T\mid x_3\in \R\setminus\{0\}\},\quad
  M_2=\{x_{p_2}=x_2(0,1,0)^\T\mid x_2\in \R\setminus\{-1\})\},
\end{split}
 \end{equation}
where $x_1-x_3=w$, $x_3=\xi$ and $x_2=u$ if the new variables \eqref{NewVariable} are used.

Let us prove that the component $x_{s_1}(t)=S_1x(t)$ (recall that $x_{p_1}=0$) of each solution $x(t)$ of the DAE \eqref{DAE}, \eqref{pencilGener}, \eqref{genSing-f} with the initial point $(t_0,x_0)\in L_0$, for which $S_1x_0\in M_{s1}$, $S_2x_0\in M_{s_2}$ and $P_2x_0\in M_2$, can never leave $M_{s1}$. Notice that any $t_0\in [0,\infty)$, $x_0=(x_{1,0},x_{2,0},x_{3,0})^\T=(w_0+\xi_0,u_0,\xi_0)^\T$ such that $w_0\in (1,+\infty)$, $u_0\in \R\setminus\{-1\}$, $\xi_0\in \R\setminus\{0\}$ and $u_0=-(t_0+1)^{1/3}\xi_0^{2/3}-1$ are consistent initial values for which $S_1x_0\in M_{s1}$, $S_2x_0\in M_{s_2}$ and $P_2x_0\in M_2$.
Since the DAE \eqref{DAE}, \eqref{pencilGener}, \eqref{genSing-f} is equivalent to the system \eqref{genSing1NewCase}, \eqref{genSing3NewCase} and $x_{s_1}=w(1,0,0)^\T\in M_{s1}$ if and only if $w\in (1,+\infty)$, then we need to prove that a solution $w=w(t)$ of \eqref{genSing1NewCase} with the initial values $t_0\in [0,\infty)$, $w(t_0)=w_0\in (1,+\infty)$ can never leave the set $\widetilde{M}_{s1}:=\{w\in (1,+\infty)\}$. Indeed, since $\dot{w}=0$ along the boundary $\partial \widetilde{M}_{s1}=\{w=1\}$ of $\widetilde{M}_{s1}$ and $\dot{w}>0$ inside $\widetilde{M}_{s1}$, then each solution of \eqref{genSing1NewCase} which at the initial moment $t_0$ is in $\widetilde{M}_{s1}$ can never thereafter leave it. Hence, condition \ref{SingAttractor} of Theorem \ref{Th_GlobSing-LipschObl} holds.

Define the function $V(t,x_{s_1},x_{p_1})\equiv V(x_{s_1}):=(w-1,0,0)(w-1,0,0)^\T=(w-1)^2$. Obviously, the function $V$ is positive for  $x_{s_1}=(w,0,0)^\T\in M_{s1}$, i.e., $w\in (1,+\infty)$. Note that  the equation \eqref{DAEsysExtDE2} disappears and the equation \eqref{DAEsysExtDE1} is equivalent to \eqref{genSing1New}. Namely, \eqref{DAEsysExtDE1} takes the form $\dot{x}_{s_1}=\big(\dot{w},0,0\big)^\T=\big(-w+ \widetilde{f}_1(t,w,\xi,u)+ 0.5\widetilde{f}_3(t,w,\xi,u),0,0\big)^\T$, where $\widetilde{f}_i$, $i=1,2,3$, have the from \eqref{genSing-tilde_f}, and it is equivalent to \eqref{genSing1NewCase} for $(t,w,\xi,u)$ satisfying \eqref{genSing1New}, \eqref{genSing3New}. Hence, $\dot{V}_{\eqref{DAEsysExtDE1},\eqref{DAEsysExtDE2}}(x_{s_1})= \dot{V}_{\eqref{DAEsysExtDE1}}(x_{s_1})=2(w-1)^4=2 V^2(x_{s_1})$ for each $t\in \mT$, $x_{s_1}\in M_{s1}$, $x_{s_2}\in M_{s_2}$, $x_{p_2}\in M_2$  such that $(t,x_{s_1}+x_{s_2}+x_{p_2})\in L_0$. Therefore, the inequality \eqref{L2v} takes the form $\dot{v}=2v^2$, and it is easy to verify that this inequality does not have global positive solutions. Hence, condition \ref{NeUstLagrSing} of Theorem \ref{Th_SingNeLagr} is fulfilled.

Thus, \emph{by Theorem \ref{Th_SingNeLagr} (or Theorem \ref{Th_CritGlobSing} and Corollary \ref{Coroll_CritGlobSing}), for each initial point $(t_0,x_0)\in [0,\infty)\times \R^3$, where $x_0=(x_{1,0},x_{2,0},x_{3,0})^\T$, such that $x_{1,0}-x_{3,0}>1$} (i.e., $S_1x_0\in M_{s1}$), $x_{2,0}\ne -1$ (i.e.,  $P_2x_0\in M_2$), $x_{3,0}\ne 0$ (i.e., $S_2x_0\in M_{s_2}$), \emph{and $x_{2,0}=-(t_0+1)^{1/3}x_{3,0}^{2/3}-1$} (i.e., $(t_0,x_0)\in L_0$), \emph{the IVP \eqref{DAE}, \eqref{ini}, where $A$, $B$ and $f(t,x)$ are of the form \eqref{pencilGener}, \eqref{genSing-f},  has a unique solution $x(t)$ with the component $S_2x(t)=\varphi_{s_2}(t)=x_3(t)\, (1,0,1)^\T$, where $\varphi_{s_2}\in C([t_0,\infty),M_{s_2})$} (accordingly, $x_3\in C([t_0,\infty),\R\setminus\{0\})$\,)  \emph{is some function with the initial value  $\varphi_{s_2}(t_0)=x_{3,0}$, and this solution has a finite escape time (is blow-up in finite time).}

 \smallskip
Indeed, \emph{for each initial values $t_0\in [0,\infty)$, $x_0=(x_{1,0},x_{2,0},x_{3,0})^\T\in \R^3$ such that $x_{1,0}-x_{3,0}>1$, $x_{2,0}\ne -1$, $x_{3,0}\ne 0$ and\; $x_{2,0}=-(t_0+1)^{1/3}x_{3,0}^{2/3}-1$,  the IVP \eqref{DAE}, \eqref{pencilGener}, \eqref{genSing-f}, \eqref{ini} has the solution $x(t)=(x_1(t),x_2(t),x_3(t))^\T$ where\; $x_1(t)=\dfrac{1}{\sqrt{-2(t-t_0)+(x_{1,0}-x_{3,0}-1)^{-2}}}+1+x_3(t)$,\;    $x_2(t)=-(t+1)^{\frac{1}{3}}x_3^{\frac{2}{3}}(t)-1$\;
and\; $x_3\in C([t_0,\infty),\R\setminus\{0\})$ is some function such that\; $x_3(t_0)=x_{3,0}$.} For example, take $t_0=0$, $x_{3,0}=1$ and $x_3(t)=t+1$, then the initial values $t_0=0$, $x_0=(x_{1,0},-2,1)^\T$, where $x_{1,0}>2$, satisfy all the conditions specified above. Further, take $x_{1,0}=3$, then the components of the solution $x(t)$ have the form\;
${x_1(t)=\dfrac{1}{\sqrt{-2t+1}}+t+2}$,\; ${x_2(t)=-(t+2)}$,\; ${x_3(t)=t+1}$.\;
Obviously, it exists on $[0,\frac{1}{2})$ and $\|x(t)\|\to +\infty$ as $t\to \frac{1}{2}$; hence, it has the finite escape time $0.5$. \emph{In general, the solution of the IVP \eqref{DAE}, \eqref{pencilGener}, \eqref{genSing-f}, \eqref{ini} has the finite escape time, namely, it exists on $[t_0,T)$, where $T=t_0+0.5(x_{1,0}-x_{3,0}-1)^{-2}$, and $\lim\limits_{t\to T-0} \|x(t)\|=+\infty$.}

\appendix

 \section{Appendix. The block form of a singular pencil of operators and the method for constructing the associated direct decompositions of spaces and projectors}\label{Appendix}

Below, the results from \cite{Fil.UMJ,Fil.KNU2019} are given. The detailed description and proof of these results can be found in \cite{Fil.KNU2019}.

Let $A$, $B$ be linear operators mapping $\Rn$ into $\Rm$; by $A$, $B$ we also denote $m\times n$-matrices corresponding to the operators $A$, $B$  (with respect to some bases in $\Rn$, $\Rm$).  Consider the operator pencil $\lambda A+B$, where $\lambda$ is a complex parameter. Note that instead of the real operators $A,\, B$ we can consider the complex operators  $A,\, B\colon \Cn \to \Cm$, for which Statement~\ref{STABssr} (see below), as well as Remark \ref{Rem-RegPenc}, remains true, but when constructing direct decompositions of the form $\eqref{ssr}$ for the complex spaces $\Cn$ and $\Cm$ and the associated projectors (the method of construction is given below), it is necessary to replace transposition by Hermitian conjugation everywhere.

Recall the following classic definition:  A linear space $L$ is decomposed into the \emph{direct sum} $L =L_1\dot + L_2$ of the subspaces $L_1\subseteq L$ and $L_2\subseteq L$ if $L_1\cap L_2 =\{0\}$ and $L_1+L_2=\{x_1+x_2 \mid x_1\in L_1, x_2\in L_2\}=L$, or, equivalently, if every $x\in L$ can be uniquely represented in the form $x=x_1+x_2$ where $x_i\in L_i$, $i=1,2$ (see, e.g., \cite[p.~309]{Faddeev}). The representation $L=L_1\dot + L_2$ is called a \emph{direct decomposition} of the space $L$.

Since the direct product (Cartesian product)  $L_1\times L_2$ is the direct sum of the spaces $L_1\times \{0\}$ and $\{0\}\times L_2$, where $0$ from $L_2$ and $L_1$ respectively, then it can be identified with the direct sum $L_1\dot + L_2$ by identifying $L_1\times \{0\}$ with $L_1$ and $\{0\}\times L_2$ with $L_2$.

\emph{Thus, when indicating the block structures of operators which are presented below, we identify direct sums and the corresponding direct products of subspaces for convenience of notation.}

 \begin{statement}[see \cite{Fil.UMJ,Fil.KNU2019}]\label{STABssr}
For operators $A,\, B\colon\Rn\to\Rm$, which form a singular pencil $\lambda A+B$ \;\textup{(}$\lambda\in \mC$ is a parameter\textup{)}, there exist the decompositions of the spaces $\Rn$, $\Rm$ into the direct sums of subspaces
 \begin{equation}\label{ssr}
\Rn=X_s\dot +X_r=X_{s_1}\dot+X_{s_2}\dot+X_r,\quad  \Rm=Y_s\dot+Y_r=Y_{s_1}\dot+Y_{s_2}\dot+Y_r
 \end{equation} 
such that with respect to the decompositions $\Rn=X_s\dot +X_r$, $\Rm=Y_s\dot+Y_r$ the operators $A$, $B$ have the block structure
 \begin{equation}\label{srAB}
A=\begin{pmatrix} 
   A_s & 0   \\
   0   & A_r \end{pmatrix},\;
B=\begin{pmatrix} 
   B_s & 0   \\
   0   & B_r \end{pmatrix}\colon X_s\dot +X_r\to Y_s\dot +Y_r\quad  (X_s\times X_r\to Y_s\times Y_r),
 \end{equation}
where $A_s=A\big|_{X_s},\, B_s=B\big|_{X_s}\colon X_s\to Y_s$ and $A_r=A\big|_{X_r},\, B_r=B\big|_{X_r}\colon X_r\to Y_r$, that is, the pair of ``singular'' subspaces $\{X_s,Y_s\}$ and the pair of ``regular'' subspaces  $\{X_r,Y_r\}$  are invariant under the operators  $A$, $B$  \textup{(}i.e., $AX_s \subseteq Y_s$, $BX_s \subseteq Y_s$ and $AX_r \subseteq Y_r$, $BX_r \subseteq Y_r$\textup{)}, and with respect to the decompositions
\begin{equation}\label{sGener}
X_s = X_{s_1}\dot +  X_{s_2},\quad Y_s = Y_{s_1}\dot +  Y_{s_2}
\end{equation}
the ``singular'' blocks $A_s$, $B_s$ have the block structure
 \begin{equation}\label{sab}
A_s = \begin{pmatrix} A_{gen} & 0 \\ 0 & 0
     \end{pmatrix},\;
B_s = \begin{pmatrix} B_{gen} & B_{und} \\ B_{ov} & 0
     \end{pmatrix}\colon X_{s_1}\dot + X_{s_2} \to Y_{s_1}\dot + Y_{s_2}\quad (X_{s_1}\times X_{s_2} \to Y_{s_1}\times Y_{s_2}),
 \end{equation}
where the operator ${A_{gen}\colon X_{s_1}\to Y_{s_1}}$ has the inverse ${A_{gen}^{-1}\in \mathrm{L}(Y_{s_1},X_{s_1})}$ \textup{(}if ${X_{s_1}\ne \{0\}}$\textup{)}, ${B_{gen}\colon X_{s_1}\to Y_{s_1}}$, ${B_{und}\colon X_{s_2}\to Y_{s_1}}$ and ${B_{ov}\colon X_{s_1}\to Y_{s_2}}$.\;
If ${\rank(\lambda A+B)=m<n}$, then the structure of the singular blocks takes the form
  \begin{equation}\label{sab1}
A_s = \begin{pmatrix} A_{gen} & 0 \end{pmatrix},\;
B_s = \begin{pmatrix} B_{gen} & B_{und} \end{pmatrix}\colon X_{s_1}\dot +  X_{s_2}\to Y_s\quad (X_{s_1}\times  X_{s_2}\to Y_s)
 \end{equation}
and\; $Y_{s_1}=Y_s$, $Y_{s_2}= \{0\}$\; in \eqref{ssr} and, accordingly, in \eqref{sGener}.\; If ${\rank(\lambda A+B)=n<m}$, then the structure of the singular blocks takes the form
 \begin{equation}\label{sab2}
A_s = \begin{pmatrix} A_{gen} \\ 0
  \end{pmatrix},\;
B_s = \begin{pmatrix} B_{gen} \\ B_{ov}
 \end{pmatrix}\colon X_s\to Y_{s_1}\dot+ Y_{s_2}\quad (X_s\to Y_{s_1}\times Y_{s_2})
\end{equation}
and $X_{s_1}=X_s$, $X_{s_2}=\{0\}$ in  \eqref{ssr} and, accordingly, in \eqref{sGener}.

The direct decompositions \eqref{ssr} generate the pair $S$, $P$, the pair $F$, $Q$, the pair $S_1$, $S_2$ and the pair $F_1$, $F_2$  of the mutually complementary projectors
\begin{align}
& S \colon \Rn \to X_s,\; P \colon \Rn\to X_r,
&F \colon \Rm \to Y_s,\; Q \colon \Rm \to Y_r, \label{ProjRS}  \\
& S_i\colon \Rn \to X_{s_i},
&F_i\colon \Rm \to Y_{s_i},\quad i=1, 2, \label{ProjS}
\end{align}
\textup{(}i.e., ${S+P=I_{\Rn}}$, ${S^2=S}$, $P^2=P$, $SP=PS=0$;\,  $F+Q=I_{\Rm}$, $F^2=F$, $Q^2=Q$, $FQ=QF=0$;\,  $S_1+S_2=S$, $S_iS_j=\delta_{ij}S_i$;\,  $F_1+F_2=F$, $F_iF_j=\delta_{ij}F_i$\,\textup{)}\,
where \,$F_1=F$, $F_2=0$\, if $\rank(\lambda A+B)=m<n$,\, and \,$S_1=S$, $S_2=0$\, if $\rank(\lambda A+B)=n<m$. These projectors have the properties
\begin{align}
& FA=AS,\quad FB=BS,\qquad  QA =AP,\quad QB=BP, \label{ProjRSInvar} \\
& {A S_2 = 0},\quad {F_2 A = 0},\quad {F_2 B S_2=0}. \label{properProjS}
\end{align}

The converse assertion that there exist the pairs of mutually complementary projectors \eqref{ProjRS}, \eqref{ProjS} satisfying \eqref{ProjRSInvar}, \eqref{properProjS}, which generate the direct decompositions \eqref{ssr}, is also true.
 \end{statement}
 \begin{remark}
The spaces $\Rn$, $\Rm$ have the direct decompositions \eqref{ssr} and, accordingly, the singular subspaces $X_s$, $Y_s$ have the direct decompositions \eqref{sGener} in the general case when $\rank(\lambda A+B) < n$ and $\rank(\lambda A+B) < m$. Also, notice that $F_1 A=AS_1=FA$.
 \end{remark}

The extensions of the operators from the block representations \eqref{srAB}, \eqref{sab}, \eqref{sab1}, \eqref{sab2} to $\Rn$ and the corresponding semi-inverse operators, which have been introduced in \cite{Fil.KNU2019}, are used in this paper and described below.

First, introduce the following extensions of the operators $A_s$, $A_r$, $B_s$, $B_r$ from \eqref{srAB} to $\Rn$:
\begin{equation}\label{AsrBsrExtend}
 \EuScript A_s = F A,\quad \EuScript A_r= Q A,\quad \EuScript B_s= F B,\quad \EuScript B_r= Q B.
\end{equation}
Then the operators  $\EuScript A_s,\, \EuScript B_s,\, \EuScript A_r,\, \EuScript B_r\in \mathrm{L}(\Rn,\Rm)$ act so that ${\EuScript A_s,\, \EuScript B_s\colon \Rn\to Y_s}$,  ${\EuScript A_r,\, \EuScript B_r\colon \Rn\to Y_r}$,\; ${X_r\subset \Ker(\EuScript A_s)}$, ${X_r\subset \Ker(\EuScript B_s)}$,\; $X_s\subset \Ker(\EuScript A_r)$, ${X_s\subset \Ker(\EuScript B_r)}$, and
\begin{equation}\label{AsrBsr}
\EuScript A_s\big|_{X_s}=A_s,\quad \EuScript A_r\big|_{X_r}=A_r,\quad \EuScript B_s\big|_{X_s}=B_s,\quad \EuScript B_r\big|_{X_r}=B_r.
\end{equation}

Further, introduce extensions of the operators (blocks) from \eqref{sab} to $\Rn$ as follows:
 \begin{equation}\label{ABssExtend}
\EuScript A_{gen}= F_1 A,\quad \EuScript B_{gen}=F_1 B S_1,\quad
\EuScript B_{und}=F_1 B S_2,\quad \EuScript B_{ov}=F_2 B S_1.
 \end{equation}  
Then $\EuScript A_{gen},\, \EuScript B_{gen},\, \EuScript B_{und},\, \EuScript B_{ov}\in \mathrm{L}(\Rn,\Rm)$ act so that $\EuScript A_{gen}\Rn = \EuScript A_{gen}X_{s_1}= Y_{s_1}$ ($X_{s_2}\dot +X_r= \Ker(\EuScript A_{gen})$),\, ${\EuScript B_{gen}\colon \Rn\to Y_{s_1}}$, $X_{s_2}\dot +X_r\subset \Ker(\EuScript B_{gen})$,\, ${\EuScript B_{und}\colon \Rn\to Y_{s_1}}$, $X_{s_1}\dot + X_r\subset \Ker(\EuScript B_{und})$,\, and ${\EuScript B_{ov}\colon \Rn\to Y_{s_2}}$, $X_{s_2}\dot +X_r\subset \Ker(\EuScript B_{ov})$, and
\begin{equation}\label{ABss}
 \EuScript A_{gen}\big|_{X_{s_1}}=A_{gen},\;\;
 \EuScript B_{gen}\big|_{X_{s_1}}=B_{gen},\;\;
 \EuScript B_{und}\big|_{X_{s_2}}=B_{und},\;\;
 \EuScript B_{ov}\big|_{X_{s_1}}=B_{ov}.
\end{equation}

Extensions of the operators (blocks) from \eqref{sab1} to $\Rn$ are introduced as follows:
\begin{equation}\label{ABsab1Extend} 
\EuScript A_{gen}=AS_1,\quad \EuScript B_{gen}=BS_1,\quad \EuScript B_{und}=BS_2.
\end{equation} 
Then the operators ${\EuScript A_{gen},\, \EuScript B_{gen},\, \EuScript B_{und}\in \mathrm{L}(\Rn,\Rm)}$ act so that ${\EuScript A_{gen}\Rn =\EuScript A_{gen}X_{s_1}=Y_s}$ ($X_{s_2}\dot + X_r= \Ker(\EuScript A_{gen})$),\, $\EuScript B_{gen}\colon \Rn\to Y_s$, $X_{s_2}\dot + X_r\subset \Ker(\EuScript B_{gen})$, and ${\EuScript B_{und}\colon \Rn\to Y_s}$,\, $X_{s_1}\dot +X_r\subset \Ker(\EuScript B_{und})$, and
\begin{equation}\label{ABsab1}
\EuScript A_{gen}\big|_{X_{s_1}}=A_{gen},\quad \EuScript B_{gen}\big|_{X_{s_1}}=B_{gen},\quad \EuScript B_{und}\big|_{X_{s_2}}=B_{und}.
\end{equation}

Extensions of the operators (blocks) from \eqref{sab2} to $\Rn$ are introduced as follows:
\begin{equation}\label{ABsab2Extend}
\EuScript A_{gen}=F_1A,\quad \EuScript B_{gen}=F_1B,\quad \EuScript B_{ov}=F_2B.
\end{equation}
Then the operators  $\EuScript A_{gen}, \EuScript B_{gen}, \EuScript B_{ov}\!\in\! \mathrm{L}(\Rn,\Rm)$ act so that $\EuScript A_{gen}\Rn \!=\!\EuScript A_{gen}X_s\!=\! Y_{s_1}$ ($X_r\!=\! \Ker(\EuScript A_{gen})$), $\EuScript B_{gen}\colon \Rn\to Y_{s_1}$, $X_r\!\subset\! \Ker(\EuScript B_{gen})$, and $\EuScript B_{ov}\colon \Rn\to Y_{s_2}$, $X_r\!\subset\! \Ker(\EuScript B_{ov})$, and
\begin{equation}\label{ABsab2}
\EuScript A_{gen}\big|_{X_s}=A_{gen},\quad \EuScript B_{gen}\big|_{X_s}=B_{gen},\quad \EuScript B_{ov}\big|_{X_s}=B_{ov}.
\end{equation}

 \begin{remark}[{\cite{Fil.KNU2019}}]\label{Rem_InvAgen}
The operator $\EuScript A_{gen}^{(-1)}\in \mathrm{L}(\Rm,\Rn)$ defined by the relations
\begin{equation}\label{InvAgen}
\EuScript A_{gen}^{(-1)} \EuScript A_{gen}=S_1,\quad  \EuScript A_{gen}\, \EuScript A_{gen}^{(-1)}=F_1,\quad \EuScript A_{gen}^{(-1)}=S_1 \EuScript A_{gen}^{(-1)},
\end{equation}
where ${F_1 = F}$ if ${\rank(\lambda A+B)= m<n}$ and ${S_1=S}$ if ${\rank(\lambda A+B)= n<m}$, is the \mbox{\emph{semi-inverse}} operator of $\EuScript A_{gen}$ \,(see the definition, e.g., in \cite{Faddeev}), i.e., $\EuScript A_{gen}^{(-1)}\Rm =\EuScript A_{gen}^{(-1)}Y_{s_1}=X_{s_1}$ \,($Y_{s_2}\dot +Y_r= \Ker(\EuScript A_{gen}^{(-1)})$) and $A_{gen}^{-1}= \EuScript A_{gen}^{(-1)}\big|_{Y_{s_1}}$.
The operator $\EuScript A_{gen}^{(-1)}\in \mathrm{L}(\Rm,\Rn)$ is the extension of the operator $A_{gen}^{-1}$ to $\Rm$.

Note that $\EuScript A_{gen}^{(-1)} F_1=S_1 \EuScript A_{gen}^{(-1)}=\EuScript A_{gen}^{(-1)}$ (where $F_1 = F$ if $\rank(\lambda A+B)= m<n$ and $S_1=S$ if $\rank(\lambda A+B)= n<m$).
 \end{remark}

With respect to the direct decompositions $\Rn=X_s\dot +X_r$, $\Rm=Y_s\dot+Y_r$, defined in \eqref{ssr}, the singular pencil $\lambda A+B$ of the operators $A, B\colon \Rn \to \Rm$ takes the block form
 \begin{equation}\label{penc}
\lambda A+B =  \begin{pmatrix}
 \lambda A_s+B_s & 0 \\
 0 &  \lambda A_r+B_r \end{pmatrix},\quad
  A_s,\, B_s\colon X_s\to Y_s,\quad A_r,\, B_r\colon X_r\to Y_r,
 \end{equation}
where the regular block $\lambda A_r+ B_r$ is a regular pencil and the singular block $\lambda A_s+ B_s$ is a purely singular pencil, i.e., it is impossible to separate out a regular block in this pencil. The pencil $\lambda A+B=\lambda A_s+ B_s$ is a purely singular pencil if the regular block $\lambda A_r+ B_r$ is absent, i.e., $X_r=\{0\}$ and $Y_r=\{0\}$.

According to \cite{GantmaherII}, one can always choose bases in $\Rn$ and $\Rm$ so that the matrix pencil $\lambda A+B$, which corresponds to the singular pencil of the operators $A, B\colon \Rn\to \Rm$ with respect to these bases,  has the canonical form that consists of singular and regular blocks of certain forms and is usually called the Weierstrass-Kronecker canonical form. Therefore, there exist the direct decompositions of the spaces $\Rn$, $\Rm$ such that the singular operator pencil $\lambda A+B$ has the form \eqref{penc}.
In the present paper, we use the special block form of the singular operator pencil (the corresponding matrix pencil is different from the canonical form presented in \cite{GantmaherII}), which is defined by the block structures of the operators $A$, $B$ and their singular blocks $A_s$, $B_s$, presented in Statement \ref{STABssr}. Below, the method for constructing the subspaces from the decompositions \eqref{ssr} and the associated projectors \eqref{ProjRS}, \eqref{ProjS}, which reduce the pencil to the mentioned block form, are presented. This method has been briefly described in \cite[Section~3]{Fil.UMJ} and described in detail in \cite[Section~3]{Fil.KNU2019}.

Previously, consider the kernel (the null-space) $\Ker (\lambda A+B)=\{x(\lambda) \mid (\lambda A+B)\,x(\lambda)\equiv 0\}$ and the range $\mathcal R(\lambda A+B)={\{y(\lambda) \mid \exists\, x:(\lambda A+B)\,x = y(\lambda)\}}$  of $\lambda A+B$. The dimensions of the kernel and the range of $\lambda A+B$ are equal to the dimensions of the kernel and the range of the complex extension $\lambda \hat{A}+\hat{B}$, respectively. Let $\lambda \in \mC$ be a fixed number. Since $\dim \Ker (\lambda \hat{A}+\hat{B})= \dim{\Cn} - \dim \mathcal R(\lambda \hat{A}+\hat{B})= n - \rank(\lambda \hat{A}+\hat{B})$, then $\dim \Ker (\lambda A+B)= n - \rank(\lambda A+B)$. Notice that $\rank(\lambda A+B)$ and $\dim \Ker (\lambda A+B) = n - \rank(\lambda A+B)$ are constants.

\paragraph{Construction of the direct decompositions of spaces and the associated projectors, which reduce the pencil to the block form.}\;

 \smallskip
\textbf{The case $\rank(\lambda A+B)=m<n$.}\; In the case when $\rank(\lambda A+B)=m<n$, there exist the direct decompositions \eqref{ssr}, where $Y_{s_1}=Y_s$ and $Y_{s_2}=\{0\}$ (i.e., $\Rn=X_s\dot +X_r$, $X_s=X_{s_1}\dot+X_{s_2}$ and $\Rm=Y_s\dot+Y_r$), with respect to which $A$, $B$ and their singular blocks $A_s$, $B_s$ have the structures \eqref{srAB} and \eqref{sab1}, respectively.
These direct decompositions generate the pairs of mutually complementary projectors \eqref{ProjRS} and \eqref{ProjS}, where $F_1=F$ and $F_2=0$, which satisfy the properties \eqref{ProjRSInvar}, \eqref{properProjS}  (more precisely, the decomposition $X_s=X_{s_1}\dot +X_{s_2}$ of the space  $X_s=S \Rn$  generates the pair of mutually complementary projectors, the extensions of which to $\Rn$ are the projectors $S_i$, $i=1,2$, defined in \eqref{ProjS}, such that $AS_2=0$).

First, we construct the singular subspaces $X_s$, $Y_s$, $X_{s_1}$ and $X_{s_2}$. To do this, we find the maximum number of linearly independent solutions $x_1(\lambda),\;  x_2(\lambda),\ldots,\; x_N(\lambda)$ of the equation
 \begin{equation}\label{har1}
(\lambda A+B)x=0.
 \end{equation}
It is sufficient to consider the solutions being polynomials in $\lambda$, that is,
\begin{equation}\label{har1x_j}
{x_j(\lambda )= \sum\limits_{i=0}^{k_j} (-1)^i \lambda^i x_{ji}},\quad {x_{ji}\ne 0},\;  {i=0,\ldots,k_j},\; {j=1,\ldots,N},
\end{equation}
where ${k_j\ge 0}$ is the degree of $x_j(\lambda )$. Clearly, the columns $x_1(\lambda),\ldots,\, x_N(\lambda)$ are linearly independent if the rank of the matrix formed from these columns is equal to~$N$.
Since the set of columns  $\{x_1(\lambda),\ldots,\, x_N(\lambda)\}$ forms a basis of $\Ker (\lambda A+B)$, then
$$
N =\dim\Ker (\lambda A+B)=n-\rank(\lambda A+B).
$$
Substituting $x_j(\lambda )$ in \eqref{har1} and equating coefficients of powers of $\lambda$ to zero, we obtain a set of equalities (cf. \cite[p.~29]{GantmaherII})
\begin{equation*}
Bx_{j 0}=0,\; B x_{j 1} = Ax_{j 0},\,\ldots\,,\; Bx_{j k_j}= Ax_{j k_j-1},\; Ax_{j k_j}=0.
\end{equation*}
Obviously, if $k_j= 0$,\, i.e., $x_j(\lambda )\equiv x_{j0}$ is a constant solution, then $Bx_{j0}=0$ and $Ax_{j 0}=Ax_{j k_j}=0$. 
Among all the solutions of the equation \eqref{har1}, we can select a collection of linearly independent solutions
$\{\hat{x}_j(\lambda )\}_{j=1}^{N}$ of the form \eqref{har1x_j}, i.e.,\; ${\hat{x}_j(\lambda)= \sum\limits_{i=0}^{k_j} (-1)^i \lambda^i \hat{x}_{ji}}$\; where ${\hat{x}_{ji}\ne 0}$,\; ${i=0,\ldots,k_j}$,\; ${j=1,\ldots,N}$,
with the lowest possible degrees $k_1,\, k_2,\, \ldots,\, k_N$ (this collection is not uniquely determined, but any two such collections of solutions have identical sets of degrees to within permutations; obviously, ${\sum\limits_{j=1}^{N} k_j \le m}$, ${\sum\limits_{j=1}^{N} k_j +N \le n}$, and we can select the collection such that ${k_1\le k_2\le \ldots\le k_N}$).
Then the vector sets $\{\hat{x}_{ji}\}_{j=1, i=0}^{N, k_j}$ and $\{B\hat{x}_{ji}\}_{j=1, i=1}^{N, k_j}=\{A\hat{x}_{ji}\}_{j=1, i=0}^{N, k_j-1}$ are linearly independent and are bases of the spaces
$$
X_s = \lin\{\hat{x}_{ji}\}_{ j=1, i=0}^{N, k_j},\quad Y_s = \lin\{B\hat{x}_{ji}\}_{j=1, i=1}^{N, k_j} =\lin\{A\hat{x}_{ji}\}_{j=1, i=0}^{N, k_j-1}.
$$
If all $k_j=0$\, ($j=1,...,N$), then $Y_s = \{0\}$, $X_{s_2}=X_s$,  $X_{s_1}=\{0\}$ and $A_s=0$, $B_s=0$, and the block form of the pencil has the form $\lambda A+B = \begin{pmatrix} 0 & \lambda A_r+B_r \end{pmatrix}$. 
Note that if we take an arbitrary maximal collection of linearly independent solutions $\{x_j(\lambda)\}_{j=1}^{N}$ of the equation \eqref{har1}, then the linear spans of the sets $\{x_{ji}\}_{j=1, i=0}^{N, k_j}$ and $\{Bx_{ji}\}_{j=1, i=1}^{N, k_j}$ also form the spaces  $X_s$ and $Y_s$ respectively, however, these sets may contain linearly dependent vectors. 
Further, one can construct the spaces (i.e., the bases of the spaces) $X_{s_1}$, $X_{s_2}$ and $X_s$ (the bases of $X_s$ consists of the union of the bases of $X_{s_1}$, $X_{s_2}$\,), using the set $\{x_{ji}\}_{j=1, i=0}^{N,\, k_j}$, and the space $Y_s$, using the set $\{Bx_{ji}\}_{j=1, i=1}^{N, k_j}$, so that the operators $A_s$, $B_s$ have the block structure \eqref{sab1}.
Obviously, $X_{s_2}=\Ker (A_s)$\, ($\dim X_{s_2}=N$), and $X_{s_1}$ is a direct complement of $X_{s_2}$ in $X_s$ \,($\dim X_{s_1}=\rank(A_s)=\rank(\EuScript A_s)$).
For the collection of linearly independent solutions $\{\hat{x}_j(\lambda)\}_{j=1}^{N}$ with the lowest possible degrees $k_j$, ${j=1,\ldots,N}$, these subspaces have the form
$$
X_{s_2}=\lin\{\hat{x}_{j k_j}\}_{j=1}^N,\quad X_{s_1}=\lin\{\hat{x}_{ji}\}_{j=1, i=0}^{N,k_j-1}\qquad \Big(\dim X_{s_1}=\sum\limits_{j=1}^N k_j\Big).
$$

Using the form of the obtained spaces $X_s$ and $Y_s$, one can construct the regular spaces $X_r$ and $Y_r$ from the decompositions \eqref{ssr} (where $Y_{s_1}=Y_s$, $Y_{s_2}=\{0\}$).  The form of the spaces $X_s$, $Y_s$, $X_r$, $Y_r$, $X_{s_1}$, and $X_{s_2}$ is used to construct the corresponding projectors with the above properties.

 \medskip
\textbf{The case $\rank(\lambda A+B)=n<m$.}\; In the case when $\rank(\lambda A+B)=n<m$, there exist the direct decompositions \eqref{ssr}, where $X_{s_1}=X_s$ and $X_{s_2}=\{0\}$ (i.e., $\Rn=X_s\dot +X_r$ and $\Rm=Y_s\dot+Y_r$, $Y_s=Y_{s_1}\dot +Y_{s_2}$), with respect to which $A$, $B$ and their singular blocks $A_s$, $B_s$ have the structures \eqref{srAB} and \eqref{sab2}, respectively.
These decompositions generate the pairs of mutually complementary projectors \eqref{ProjRS} and \eqref{ProjS}, where $S_1=S$ and $S_2=0$, which satisfy the properties \eqref{ProjRSInvar}, \eqref{properProjS}  (more precisely, the decomposition $Y_s=Y_{s_1}\dot +Y_{s_2}$ of the space $Y_s=F\, \Rm$  generates the pair of mutually complementary projectors, the extensions of which to $\Rm$ are the projectors $F_i$, $i=1,2$, defined in \eqref{ProjS}, such that $F_2 A=0$).

First, we construct the singular subspaces $X_s$, $Y_s$, $Y_{s_1}$ and $Y_{s_2}$. To do this, we find the maximum number of linearly independent solutions $y_1(\lambda),\ldots,\, y_M(\lambda)$ of the equation
\begin{equation}\label{har2}
  (\lambda A^\T+B^\T)y=0,
\end{equation}
where $\lambda A^\T+B^\T$ is the transposed pencil and
$$
M =\dim \Ker (\lambda A^\T+B^\T)= m-\rank(\lambda A+B).
$$
As above, it is sufficient to consider solutions of the form\;
\begin{equation}\label{har2y_j}
y_j(\lambda)= \sum\limits_{l=0}^{m_j} (-1)^l \lambda^l y_{jl},\quad {y_{jl} \ne 0},\quad {l=0,\ldots,m_j},\quad {j=1,\ldots,M},
\end{equation}
where ${m_j\ge 0}$.  Substituting $y_j(\lambda )$ in \eqref{har2} and equating the coefficients of the powers of $\lambda$ to zero, we obtain the set of equalities
$$
B^\T y_{j 0}=0,\; B^\T y_{j 1} = A^\T y_{j 0},\,\ldots\,,\; B^\T y_{j m_j}= A^\T y_{j m_j-1},\; A^\T y_{j m_j}=0.
$$
If $m_j= 0$, then $B^\T y_{j0}= A^\T y_{j0}= A^\T y_{j m_j}=0$.  Further, we construct the singular subspaces $\hat{X}_s=\hat{X}_{s_1}\dot +\hat{X}_{s_2}$, $\hat{Y}_s$, $\hat{X}_{s_2}=\Ker(A_s^\T)$ and $\hat{X}_{s_1}$ for the pencil $\lambda A^\T+B^\T$ in the same way as in the previous case (i.e., as for $\lambda A+B$ in the case when $\rank(\lambda A+B)=m<n$). For the collection of linearly independent solutions $\{y_j(\lambda)\}_{j=1}^M$ with the lowest possible degrees  $m_1,\, \ldots,\, m_M$, these subspaces have the form\; $\hat{X}_{s_2}=\lin\{y_{j m_j}\}_{j=1}^M$,\;  $\hat{X}_{s_1}=\lin\{y_{jl}\}_{j=1, l=0}^{M, m_j-1}$, $\hat{X}_s=\lin\{y_{jl}\}_{j=1, l=0}^{M, m_j}$ and
$\hat{Y}_s=\lin\{B^\T y_{jl}\}_{j=1, l=1}^{M, m_j}=\lin\{A^\T y_{jl}\}_{j=1, l=0}^{M, m_j-1}$. 
If all $m_j=0$, then $\hat{Y}_s= \{0\}$, $\hat{X}_{s_2}=\hat{X}_s$, and the singular blocks \eqref{sab2} are such that $A_s^\T=0$, $B_s^\T=0$.
Since $A^\T,\, B^\T\colon (\Rm)'\to (\Rn)'$, then the singular spaces $\hat{X}_{s_i}$, $i=1,2$, $\hat{X}_s$, $\hat{Y}_s$, constructed for $\lambda A^\T+B^\T$, coincide with the conjugate spaces of $Y_{s_i}$, $i=1,2$, $Y_s$, $X_s$   from the decompositions \eqref{ssr}, where $X_{s_1}=X_s$ and $X_{s_2}=\{0\}$, i.e., $\hat{X}_{s_i}=Y'_{s_i}$, $i=1,2$, $\hat{X}_s=Y'_s$ and $\hat{Y}_s=X'_s$.  Clearly, the conjugate space $(\R^k)'$ can be replaced by $\R^k$,  however, their bases may not coincide and therefore we reserve the notation $(\R^k)'$ here.
If $\hat{Y}_s= \{0\}$ and $\hat{X}_{s_2}=\hat{X}_s$, then the block form of the pencil takes the form  $\lambda A+B=\begin{pmatrix} 0 \\ \lambda A_r+B_r \end{pmatrix}$ and $X_s=\{0\}$, $Y_{s_2}=Y_s$, $Y_{s_1}=\{0\}$.
Using the form of the singular subspaces $\hat{X}_{s_i}$, $i=1,2$, $\hat{X}_s$, $\hat{Y}_s$, one can construct the regular subspaces $\hat{X}_r=Y'_r$, $\hat{Y}_r=X'_r$  from the direct decompositions $(\Rm )' =\hat{X}_s\dot +\hat{X}_r= Y'_s\dot + Y'_r$ and $(\Rn )'=\hat{Y}_s\dot +\hat{Y}_r=X'_s\dot + X'_r$  and the pair $\hat{S}_1$, $\hat{S}_2$, the pair $\hat{S}$, $\hat{P}$ and the pair $\hat{F}$, $\hat{Q}$ of the mutually complementary projectors
$\hat{S}_i\colon (\Rm )'\to \hat{X}_{s_i}$, ${i=1,2}$, $\hat{S}=\hat{S}_1+\hat{S}_2 \colon (\Rm )' \to \hat{X}_s$, $\hat{P}\colon (\Rm)' \to \hat{X}_r$,
$\hat{F}\colon (\Rn)' \to \hat{Y}_s$ and $\hat{Q}\colon (\Rn)'\to \hat{Y}_r$
such that $A^\T \hat{S}_2=0$, $\hat{S}^\T A = A\hat{F}^\T$, $\hat{S}^\T B =B\hat{F}^\T$, $\hat{P}^\T A =A\hat{Q}^\T$ and $\hat{P}^\T B=B\hat{Q}^\T$.  It follows from the properties of $\hat{S}_i$ that the transposed (adjoint) projectors $F_i=\hat{S}_i^\T\colon \Rm \to Y_{s_i}$, $i{=1, 2}$, ($F_1+F_2=F$) are mutually complementary projectors such that $F_2 A = 0$ and generate the direct decomposition $Y_s=Y_{s_1}\dot+ Y_{s_2}$ with respect to which $A_s$, $B_s$ have the block structure \eqref{sab2}.
Also, it follows from the properties of $\hat{S}$, $\hat{P}$ and $\hat{F}$, $\hat{Q}$ that the transposed (adjoint) projectors $S=\hat{F}^\T$, $P=\hat{Q}^\T$ and $F=\hat{S}^\T$, $Q=\hat{P}^\T$ are the two pairs of mutually complementary projectors \eqref{ProjRS} satisfying \eqref{ProjRSInvar} and they
generate the direct decompositions $\Rn =X_s\dot + X_r$, $\Rm = Y_s\dot + Y_r$ such that $A$, $B$ have the block structure \eqref{srAB} (accordingly,  the pair $\{X_s,Y_s\}$ and the pair $\{X_r,Y_r\}$  are invariant under the operators  $A$, $B$). Further, using the form of the projectors (or the conjugate subspaces), one can construct the subspaces from the direct decompositions \eqref{ssr}, where $X_{s_1}=X_s$ and $X_{s_2}=\{0\}$.

 \medskip
\textbf{The case $\rank(\lambda A+B) < n, m$ (the general case).}\;
In the general case, when $\rank(\lambda A+B)<n$ and $<m$, there exist the direct decompositions \eqref{ssr} with respect to which $A$, $B$ and their singular blocks $A_s$, $B_s$ have the structures \eqref{srAB} and \eqref{sab}, respectively. These decompositions generate the pairs of mutually complementary projectors \eqref{ProjRS}, \eqref{ProjS}, which satisfy \eqref{ProjRSInvar}, \eqref{properProjS}  (more precisely, the decompositions \eqref{sGener} generate the two pairs of mutually complementary projectors, the extensions of which to $\Rn$, $\Rm$ are the projectors $S_i$, $F_i$ $i=1,2$, defined in \eqref{ProjS}, such that \eqref{properProjS} holds).

In this case, for the construction of the singular spaces and the corresponding projectors it is necessary to find $N=n-\rank(\lambda A+B)$ linearly independent solutions of the equation \eqref{har1} and $M=m-\rank(\lambda A+B)$ linearly independent solutions of the equation \eqref{har2} (we seek the solutions of the form \eqref{har1x_j} and \eqref{har2y_j} with the lowest possible degrees, as was done above). 
Further, analyzing the solutions of \eqref{har1} and \eqref{har2},  we construct the singular spaces $X_s$, $Y_s$, $X_{s_1}$, $X_{s_2}$, $Y_{s_1}$, $Y_{s_2}$ and the associated projectors with regard for their properties enabling to obtain the block structure \eqref{sab}.
The basis of $X_{s_2}=\Ker(A_s)$ consists of the vectors obtained by analyzing the solutions of \eqref{har1} and is constructed in the same way as the basis of $X_{s_2}$ in the case when $\rank(\lambda A+B) = m<n$;\,
the basis of $Y_{s_2}=(\Ker (A_s^\T))'$ consists of the vectors obtained by analyzing the solutions of \eqref{har2} and is constructed in the same way as a basis of $Y_{s_2}$ in the case when $\rank(\lambda A+B) = n<m$.
The basis of $X_{s_1}$ consists of the vectors that are found in the same way as the basis vectors of $X_{s_1}$ in the case when $\rank(\lambda A+B) = m<n$ and the vectors that are found in the same way as the basis vectors of $X_s$ in the case when $\rank(\lambda A+B) =n<m$.
The basis of $Y_{s_1}$ consists of the vectors that are found in the same way as the basis vectors of $Y_{s_1}$ in the case when $\rank(\lambda A+B) = n<m$ and the vectors that are found in the same way as the basis vectors of $Y_s$ in the case when $\rank(\lambda A+B) = m<n$.
The bases of the spaces $X_s$ and $Y_s$ which are the direct sums \eqref{sGener} consist of the union of bases of their summands.
Based on the form of the singular spaces $X_s$ and $Y_s$, we construct the regular spaces $X_r$ and $Y_r$ from the direct decompositions \eqref{ssr}.

\begin{definition}[{\cite{Fil.KNU2019, Fil.UMJ}}]\label{DefectPencil}
The total maximum number $d(\lambda A+B) = n+m-2\, \rank(\lambda A+B) =\dim \Ker (\lambda A+B) +\dim \Ker (\lambda A^\T+B^\T)$  of linearly independent solutions of the equation \eqref{har1} and linearly independent solutions of the equation \eqref{har2} is called the \emph{total defect of the pencil} $\lambda A+B$\, (in  \cite{Fil.UMJ} it was called a defect). The total maximum number of linearly independent solutions of the equation \eqref{har1}, i.e., $\dim \Ker (\lambda A+B)$, is called the \emph{defect of the pencil} $\lambda A+B$.
\end{definition}
For the pencil of the rank $\rank(\lambda A+B)= m<n$, the total defect is $d(\lambda A+B)= \dim \Ker (\lambda A+B)= n-\rank(\lambda A+B)$.  If the pencil has the rank $\rank(\lambda A+B)= n<m$, its total defect is $d(\lambda A+B) = \dim \Ker (\lambda A^\T+B^\T) = m-\rank(\lambda A+B)$. The total defects of the original and transposed pencils coincide, and if $n=m$, their defects also coincide. The defect and the total defect of a regular pencil are equal to zero.

\paragraph{The block structure of a regular pencil of index not higher than 1, and the method for constructing the associated direct decompositions of spaces and projectors.}
Consider a regular pencil $\lambda A_r+ B_r$ of operators $A_r,\, B_r\colon X_r\to Y_r$ acting in finite-dimensional spaces ($\dim X_r =\dim Y_r$).
We assume that the point $\mu=0$ is either a pole of order 1 of the resolvent $(A_r+ \mu B_r)^{-1}$ or a regular point of the pencil $A_r+ \mu B_r$ (i.e., either $\lambda=\infty$ is a removable singular point of the resolvent  $(\lambda A_r+B_r)^{-1}$, or the operator $A_r$ is invertible). These conditions are equivalent to the following: there exist constants $C_1,\, C_2 >0$ such that
 \begin{equation}\label{index1}
\left\|(\lambda A_r+B_r)^{-1}\right\|\le C_1,\quad   |\lambda|\ge C_2.
 \end{equation}

If the point $\mu=0$ is a pole of the resolvent $(A_r+\mu B_r)^{-1}$, then the order $\nu$ ($\nu\in {\mathbb N}$) of the pole is called the \emph{index of the regular pencil $\lambda A_r+B_r$}, and in the case if $\mu=0$ is a regular point of the pencil $A_r+\mu B_r$, the \emph{index of the regular pencil $\lambda A_r+B_r$} is $\nu=0$ \cite{Fil.KNU2019}.
This definition is equivalent to the following: The maximum length of the chain of root vectors of the pencil $A_r+\mu B_r$ at the point $\mu=0$ is called the \emph{index} of the regular pencil $\lambda A_r+B_r$, where $A_r$, $B_r$ are square matrices \cite[Section~6.2]{Vlasenko1}.

Thus, if $\lambda A_r+B_r$ is a regular pencil and \eqref{index1} holds, then $\lambda A_r+B_r$ is a regular pencil of \emph{index not higher than 1}, i.e., index~0 or~1. If $A_r$ is invertible, then $\lambda A_r+B_r$ is a regular pencil of \emph{index~0}, and if $A_r$ is noninvertible and \eqref{index1} holds, then $\lambda A_r+ B_r$ is a regular pencil of \emph{index 1}. Note that if $A_r=0$ and there exists $B_r^{-1}$, then $\lambda A_r+ B_r\equiv B_r$ can be considered as a regular pencil of index 1.

 \begin{remark}\label{Rem-RegPenc}
Assume that the regular block $\lambda A_r+ B_r$ from \eqref{penc} is a regular pencil of index not higher than 1. Then there exists the pair $\Tilde P_i\colon X_r\to X_i$, $i=1,2$, and the pair $\Tilde Q_j\colon Y_r \to Y_j$, $j=1,2$, of mutually complementary projectors which generate the decompositions of the regular spaces $X_r$, $Y_r$ into the direct sum of subspaces
 \begin{equation}\label{rr}
X_r = X_1 \dot +X_2,\quad Y_r = Y_1 \dot +Y_2
 \end{equation}
such that the pairs of subspaces $X_1$, $Y_1$ and $X_2$, $Y_2$ are invariant under $A_r$, $B_r$ ($A_r,\, B_r \colon  X_i\to Y_i$, $i=1,2$), i.e.,
\begin{equation}\label{ProjRInv}
\Tilde Q_j A_r=A_r \Tilde P_j,\quad \Tilde Q_j B_r=B_r \Tilde P_j,
\end{equation}
$X_2=\Ker(A_r)$ (\,$A_r \Tilde P_2=0$\,), $Y_1=\mathcal R(A_r)=A_r X_r$, and the restricted operators $A_i=A_r\big|_{X_i}\colon X_i\to Y_i$, $B_i=B_r\big|_{X_i}\colon X_i\to Y_i$, $i=1,2$, are such that  ${A_2=0}$ and there exist $A_1^{-1} \in \mathrm{L}(Y_1,X_1)$\, (if $X_1\not=\{0\}$) and $B_2^{-1} \in \mathrm{L}(Y_2,X_2)$\, (if $X_2\not=\{0\}$). For a regular pencil of operators, the pairs of projectors with the specified properties were introduced in \cite{Rut}.  
If the index of the pencil is higher than 1, then $\Ker (A_r)\subsetneqq X_2$ and, therefore, $A_2\ne 0$.
With respect to the direct decompositions \eqref{rr} the operators $A_r$, $B_r$ have the block structure
\begin{equation}\label{rab}
  A_r = \begin{pmatrix}  A_1 & 0 \\ 0 & 0
  \end{pmatrix},\;
  B_r = \begin{pmatrix}  B_1 & 0 \\ 0 & B_2
  \end{pmatrix}\colon X_1\dot +X_2 \to Y_1\dot +Y_2\quad (X_1\times X_2 \to Y_1\times Y_2),
\end{equation}
where $A_1$ and $B_2$ are invertible (if $X_1\not=\{0\}$ and $X_2\not=\{0\}$ respectively).

The projectors $\Tilde P_i$, $\Tilde Q_i$, $i=1,2$, can be calculated by the formulas from \cite[Lemma 3.2]{Rut}, or by the formulas \cite[(28), p.~44]{Fil.KNU2019}:
 \begin{equation}\label{ProjRes}
\Tilde P_1 \!=\! \mathop{Res }\limits_{\mu =0}\bigg(\frac{(A_r+ \mu B_r)^{-1} A_r}{\mu} \bigg),\;
\Tilde P_2 \!=\!I_{X_r}-\Tilde P_1,\;\;
\Tilde Q_1 \!=\!\mathop{Res }\limits_{\mu =0}\bigg(\frac{A_r(A_r+ \mu B_r)^{-1}}{\mu} \bigg),\;
\Tilde Q_2 \!=\!I_{Y_r} - \Tilde Q_1,
 \end{equation}
or simply by constructing operators $\Tilde P_i$, $\Tilde Q_i$, $i=1,2$, such that  $\Tilde P_2^2=\Tilde P_2$, $A_r \Tilde P_2 = 0$, $\Tilde P_1= I_{X_r}-\Tilde P_2$, $\Tilde Q_1^2=\Tilde Q_1$ (or $\Tilde Q_2^2=\Tilde Q_2$), $\Tilde Q_2+\Tilde Q_1= I_{Y_r}$ and \eqref{ProjRInv} is fulfilled \,(note that if an operator $\Tilde P_2$ is a projector, then $\Tilde P_1= I_{X_r}-\Tilde P_2$ is also a projector and the projectors $\Tilde P_1$, $\Tilde P_2$ are mutually complementary) \cite[Remark~3, p.~44--45]{Fil.KNU2019}.

In addition, the direct decompositions of spaces \eqref{rr} generate the projectors $\Tilde P_i$, $\Tilde Q_i$, $i=1,2$. To obtain the subspaces from the direct decompositions \eqref{rr}, one can construct the subspaces $X_2=\Ker(A_r)$, $Y_1=\mathcal R(A_r)$ and then construct direct complements $X_1$ and $Y_2$ of the spaces $X_2$ and $Y_1$, respectively, such that the pairs $X_1$, $Y_1$ and $X_2$, $Y_2$ are invariant under $B_r$  (obviously, these pairs are invariant under $A_r$). It is clear that $Y_2=B_rX_2$ and $X_1=(\lambda A_r+B_r)^{-1}Y_1$, $|\lambda|\ge C_2$ ($C_2$ is a constant defined in \eqref{index1}) \cite{Vlasenko1}.
  \end{remark}

Introduce the extensions $P_i$, $Q_i$ of the projectors $\Tilde P_i$, $\Tilde Q_i$, defined in Remark~\ref{Rem-RegPenc}, to $\Rn$, $\Rm$, respectively, so that $X_i=P_i\Rn$, $Y_i=Q_i\Rm$, $i=1,2$ (where $X_i$, $Y_i$ from \eqref{rr}) \cite{Fil.KNU2019}. The extended projectors
\begin{equation}\label{ProjR}
P_i\colon \Rn \to X_i,\quad Q_i\colon \Rm \to Y_i,\quad i=1,2,
\end{equation}
have the properties of the original ones: the pairs $P_1$, $P_2$ and $Q_1$, $Q_2$ are two pairs of mutually complementary projectors (i.e., $P_i P_j =\delta _{ij} P_i$, $P_1+P_2=P$, $Q_i Q_j =\delta _{ij} Q_i$, $Q_1+Q_2=Q$) and
\begin{equation}\label{ProjPropInd1}
Q_i A=A P_i,\quad Q_i B=B P_i,\quad A P_2=0.
\end{equation}
Also, the properties of the operators
\begin{equation}\label{A_iB_i}
A_i=A\big|_{X_i}=A_r\big|_{X_i},\quad B_i = B\big|_{X_i}=B_r\big|_{X_i},\quad i=1,2,
\end{equation}
which are specified in Remark~\ref{Rem-RegPenc}, are retained.  Introduce their extensions to $\Rn$ as follows:
\begin{equation}\label{ABrrExtend}
 \Es A_i=Q_i A,\quad \Es B_i=Q_i B,\quad i=1,2.
\end{equation}
Then
 \begin{equation}\label{ABrr}
\Es A_i\big|_{X_i}=A_i,\quad \Es B_i\big|_{X_i}=B_i,\quad  i=1,2,
 \end{equation}
and the operators $\Es A_i,\, \Es B_i\in \mathrm{L}(\Rn,\Rm)$, $i=1,2$, act so that $\Es A_1\Rn =\Es A_1 X_1= Y_1$ \,($X_2\dot +X_s= \Ker(\Es A_1)$), $\Es A_2=0$,  $\Es B_1\colon \Rn\to Y_1$, ${X_2\dot +X_s\subset \Ker(\Es B_1)}$, and  $\Es B_2\Rn =\Es B_2 X_2= Y_2$ \,(${X_1\dot +X_s= \Ker(\Es B_2)}$).

The semi-inverse operator $\Es A_1^{(-1)}$ of $\Es A_1$, i.e., the operator $\Es A_1^{(-1)}\in \mathrm{L}(\Rm,\Rn)$ such that $\Es A_1^{(-1)}\Rm =\Es A_1^{(-1)} Y_1 =X_1$ \,($\Ker(\Es A_1^{(-1)})=Y_2\dot + Y_s$) and $A_1^{-1}=\Es A_1^{(-1)}\big|_{Y_1}$, can be defined by the relations
\begin{equation}\label{InvA1}
\Es A_1^{(-1)} \Es A_1=P_1,\quad \Es A_1 \Es A_1^{(-1)}=Q_1,\quad \Es A_1^{(-1)}=P_1 \Es A_1^{(-1)}.
\end{equation}

Similarly, the semi-inverse operator $\Es B_2^{(-1)}\in \mathrm{L}(\Rm,\Rn)$ of $\Es B_2$, i.e., the operator $\Es B_2^{(-1)}\in \mathrm{L}(\Rm,\Rn)$ such that  $\Es B_2^{(-1)}\Rm =\Es B_2^{(-1)}Y_2=X_2$ \,($\Ker(\Es B_2^{(-1)})=Y_1\dot + Y_s$) and $B_2^{-1}=\Es B_2^{(-1)}\big|_{Y_2}$,  can be defined by the relations
\begin{equation}\label{InvB1}
\Es B_2^{(-1)}\Es B_2=P_2,\quad \Es B_2\Es B_2^{(-1)}=Q_2,\quad \Es B_2^{(-1)}=P_2 \Es B_2^{(-1)}.
\end{equation}

Note that $\Es A_1^{(-1)}Q_1=P_1 \Es A_1^{(-1)}=\Es A_1^{(-1)}$ and $\Es B_2^{(-1)}Q_2=P_2 \Es B_2^{(-1)}=\Es B_2^{(-1)}$.

 \begin{remark}\label{Rem-Decomp_x_y}
The decompositions \eqref{ssr} and \eqref{rr} together give the direct decomposition
 \begin{equation}\label{Xssrr}
\Rn=X_s\dot +X_r =X_{s_1}\dot +  X_{s_2}\dot +  X_1 \dot +X_2
 \end{equation}
with respect to which any element $x\in \Rn$ can be \emph{uniquely represented} (the uniqueness of the representation follows from the definition of a direct sum of subspaces) in the form
 \begin{equation}\label{xsr}
x= x_s+x_r= x_{s_1}+x_{s_2}+x_{p_1}+x_{p_2}\qquad (x_s=x_{s_1}+x_{s_2},\quad x_r=x_{p_1}+x_{p_2}),
 \end{equation}
where $x_s=Sx\in X_s$, \,$x_r=Px\in X_r$, \,$x_{s_i}=S_i x\in X_{s_i}$, \,$x_{p_i}=P_i x\in X_i$, $i=1,2$.

Similarly, the decompositions \eqref{ssr} and \eqref{rr} together give the direct decomposition
\begin{equation}\label{Yssrr}
 \Rm=Y_s\dot +Y_r =Y_{s_1}\dot +  Y_{s_2}\dot + Y_1 \dot +Y_2,
\end{equation}
with respect to which any element $y\in \Rm$ can be uniquely represented as
 \begin{equation}\label{ysr}
y= y_s+y_r= y_{s_1}+y_{s_2}+y_{p_1}+y_{p_2}\qquad (y_s=y_{s_1}+y_{s_2},\quad y_r=y_{p_1}+y_{p_2}),
 \end{equation}
where $y_s=Fy\in Y_s$, $y_r=Qy\in Y_r$, $y_{s_i}=F_i\, y\in Y_{s_i}$ and $y_{p_i}=Q_i\, y\in Y_i$, $i=1,2$.
 \end{remark}

 \begin{remark}\label{Rem-RegCase}
In the case when the pencil $\lambda A+B$ is regular, i.e.,  $n=m=\rank(\lambda A+B)$, we have $A=A_r,\, B=B_r\in \mathrm{L}(\Rn)$, $X_r=Y_r=\Rn$,  $P_j=\Tilde P_j\colon \Rn\to X_j$ and $Q_j=\Tilde Q_j\colon \Rn\to Y_j$, $j=1,2$, (thus, there is no need to introduce extensions of the projectors $\Tilde P_j$, $\Tilde Q_j$), and, accordingly, $X_j=P_j\Rn$, $Y_j=Q_j\Rn$ ($j=1,2$) in the direct decompositions \eqref{rr}. Since we assume that $\lambda A+B=\lambda A_r+B_r$ is a regular pencil of index not higher than 1, then the operators $A$, $B$ have the block structure \eqref{rab}. The extended operators $\Es A_j,\, \Es B_j\in \mathrm{L}(\Rn)$ (see \eqref{ABrrExtend}) and the semi-inverse  $\Es A_1^{(-1)},\,\Es B_2^{(-1)}\in \mathrm{L}(\Rn)$  are introduced in the same way as above. The methods for constructing the subspaces from the direct decompositions \eqref{rr} and the projectors $P_j=\Tilde P_j$ and $Q_j=\Tilde Q_j$ ($j=1,2$) are specified in Remark~\ref{Rem-RegPenc}.

The representations \eqref{xsr} and \eqref{ysr} take the form
 \begin{equation}\label{xrr}
x= x_{p_1}+x_{p_2},\qquad x_{p_i}=P_i x\in X_i,\quad i=1,2.
 \end{equation}
 \begin{equation}\label{yrr}
y=y_{p_1}+y_{p_2},\qquad y_{p_i}=Q_i\, y\in Y_i,\quad i=1,2.
 \end{equation}
Recall that the representations \eqref{xrr} and \eqref{yrr} are uniquely determined for each $x\in \Rn$ and each $y\in \Rn$, respectively.

Also, consider the invertible operator $G\in \mathrm{L}(\Rn)$ ($G^{-1}\in \mathrm{L}(\Rn)$) \cite[Sections~3.3.1, 3.3.2]{Vlasenko1}
\begin{equation}\label{Proj.G}
G=A+Q_2B =A +BP_2,
\end{equation}
where $P_2$, $Q_2$ are the projectors introduced above, which acts so that $GX_j =Y_j$, $j=1,2$.
 \end{remark}

 \begin{remark}[{cf. \cite[Section 2.3]{Fil.Sing-GN}}]\label{Rem-Correspond}
The space $\Rn=X_{s_1}\dot+ X_{s_2}\dot+X_1 \dot+X_2$ is isomorphic to the space $X_{s_1}\times X_{s_2}\times X_1 \times X_2$, and there exists a one-to-one correspondence between an element $x=x_{s_1}+x_{s_2}+x_{p_1}+x_{p_2}$ and $x=(x_{s_1},x_{s_2},x_{p_1},x_{p_2})$, where $x_{s_i}\in X_{s_i}$ and $x_{p_i}\in X_i$, $i=1,2$. To establish this correspondence, we identify an ordered collection $(x_{s_1},x_{s_2},x_{p_1},x_{p_2})\in X_{s_1}\times X_{s_2}\times X_1 \times X_2$ (which is assumed to be a column vector) with the corresponding element $x_{s_1}+x_{s_2}+x_{p_1}+x_{p_2}\in X_{s_1}\dot+ X_{s_2}\dot+X_1 \dot+X_2$.
We define a norm in the space $X_{s_1}\times X_{s_2}\times X_1 \times X_2$ so that the norms of elements $x_{s_1}\in X_{s_1}$,  $x_{s_2}\in X_{s_2}$, $x_{p_1}\in X_1$ and $x_{p_2}\in X_2$  coincide with the norms of the elements (corresponding ordered collections)  $(x_{s_1},0,0,0)$, $(0,x_{s_2},0,0)$, $(0,0,x_{p_1},0)$ and $(0,0,0,x_{p_2})$ from $X_{s_1}\times X_{s_2}\times X_1 \times X_2$, respectively.  Also, norms in $X_{s_1}\dot+ X_{s_2}\dot+X_1 \dot+X_2$ and $X_{s_1}\times X_{s_2}\times X_1\times X_2$ are defined so that they coincide for any element $x\in \Rn$, i.e., for any $x=x_{s_1}+x_{s_2}+x_{p_1}+x_{p_2}$ and $x=(x_{s_1},x_{s_2},x_{p_1},x_{p_2})$, where $x_{s_i}\in X_{s_i}$, $x_{p_i}\in X_i$, $i=1,2$, are the same.
Thus, the correspondence between the representations of $x\in \Rn$ in the form $x=x_{s_1}+x_{s_2}+x_{p_1}+x_{p_2}$ (i.e., \eqref{xsr}) and in the form $x=(x_{s_1},x_{s_2},x_{p_1},x_{p_2})$, where $x_{s_i}\in X_{s_i}$ and $x_{p_i}\in X_i$, $i=1,2$, is established. Since we identified these representations, we can set $x=x_{s_1}+x_{s_2}+x_{p_1}+x_{p_2}=(x_{s_1},x_{s_2},x_{p_1},x_{p_2})$.

In a similar way, an ordered collection (a column vector) $y=(y_{s_1},y_{s_2},y_{p_1},y_{p_2})\in Y_{s_1}\times Y_{s_2}\times Y_1 \times Y_2$ can be identified with the corresponding element  $y=y_{s_1}+y_{s_2}+y_{p_1}+y_{p_2}\in \Rm=Y_{s_1}\dot+ Y_{s_2}\dot+Y_1 \dot+Y_2$.
 \end{remark}

{\small
\paragraph{Supports.} This work was supported by the Alexander von Humboldt Foundation.

}


\vspace*{-3mm}

 \begin{thebibliography}{}
\bibitem{BGHHST}  P. Benner,   S. Grundel,  C. Himpe,  C.Huck,   T. Streubel,  C. Tischendorf,  Gas Network Benchmark Models.  In: S. Campbell,  A. Ilchmann,  V. Mehrmann,  T. Reis (eds.), Applications of Differential-Algebraic Equations: Examples and Benchmarks. Differential-Algebraic Equations Forum, pp. 171--197.  Springer, Cham (2018). 


\bibitem{ChistyakovCh11} V.F. Chistyakov,  E.V. Chistyakova,  Application of the least squares method to solving linear differential-algebraic equations. Numer. Analys. Appl. \textbf{6}, 77–90 (2013). 


\bibitem{Faddeev} D.K. Faddeev,  Lectures on algebra. Nauka, Moscow, 1984. [in Russian]

\bibitem{Fil.Sing-GN} Filipkovska M. Qualitative analysis of nonregular differential-algebraic equations and the dynamics of gas networks.
     J. of Math. Phys., Anal., Geom.  \textbf{19}(4), 719--765  (2023). \url{https://doi.org/10.15407/mag19.04.719}

\bibitem{Fil.UMJ}  M.S. Filipkovska, Lagrange stability and instability of irregular semilinear differential-algebraic equations and applications. Ukrainian Math. J. \textbf{70}(6), 947--979 (2018).  \url{https://doi.org/10.1007/s11253-018-1544-6}

\bibitem{Fil.KNU2019} M.S. Filipkovska (Filipkovskaya),  A block form of a singular pencil of operators and a method of obtaining it. Visnyk of V.N. Karazin Kharkiv National University. Ser. ``Mathematics, Applied Mathematics and Mechanics''  \textbf{89},  33--58 (2019). [in Russian] \url{https://doi.org/10.26565/2221-5646-2019-89-04}


\bibitem{Fil.MPhAG} M.S. Filipkovska,  Lagrange stability of semilinear differential-algebraic equations and application to nonlinear electrical circuits. J. of Math. Phys., Anal., Geom. \textbf{14}(2), 169--196 (2018). \url{https://doi.org/10.15407/mag14.02.169}



\bibitem{Fil.GSA} M. Filipkovska (Filipkovskaya), Existence, boundedness and stability of solutions of time-varying semilinear differential-algebraic equations. Global and Stochastic Analysis, \textbf{7}(2), 169--195 (2020).

\bibitem{GantmaherII}  F.R. Gantmacher,  The theory of matrices, Vol. II. American Mathematical Society, Providence, Rhode Island, 2000.

\bibitem{Hartman} P. Hartman, Ordinary differential equations. John Wiley \& Sons, New York, 1964.
   

\bibitem{KSSTW22}  T. Kreimeier,  H. Sauter,  S.T. Streubel,  C. Tischendorf, A. Walther, Solving Least-Squares Collocated Differential Algebraic Equations by Successive Abs-Linear Minimization -- A Case Study on Gas Network Simulation, Humboldt-Universit{\"a}t zu Berlin, 2022 [Preprint].

\bibitem{Kunkel_Mehrmann} P. Kunkel,  V. Mehrmann, Differential-Algebraic Equations: Analysis and Numerical Solution. European Mathematical Society, Zurich, 2006.

\bibitem{LaSal-Lef}   J. La Salle,  S. Lefschetz, Stability by Liapunov's direct method with applications. Academic Press, New York, 1961.

\bibitem{Schwartz1}  L. Schwartz, Analyse Math\'{e}matique, I. Hermann, Paris, 1967.  [in French]

\bibitem{Schwartz2} L. Schwartz, Analyse Math\'{e}matique, II, Hermann, Paris, 1967. [in French]

\bibitem{Rut}   A.G. Rutkas, Cauchy problem for the equation $Ax'(t) +Bx(t) = f(t)$. Differential Equations \textbf{11}(11),  1996--2010 (1975).  [in Russian]

\bibitem{rut-sing}  A.G. Rutkas, Solvability of semilinear differential equations with singularity. Ukrainian Math. J. \textbf{60}, 262--276 (2008). 

\bibitem{Vlasenko1} L.A.  Vlasenko,  Evolution Models with Implicit and Degenerate Differential Equations. System Technologies, Dnipropetrovsk, Ukraine (2006). [in Russian]
\end{thebibliography}
\end{document}